\documentclass{amsart}

\usepackage{graphics}
\usepackage[all]{xy}

\usepackage[final]{pdfpages}

\usepackage{amssymb,amsmath}
\usepackage{psfrag}
\usepackage[margin=3cm]{geometry} 
\usepackage{mathtools}

\usepackage{amsfonts}
\usepackage{amscd}
\usepackage{mathrsfs}

\usepackage{graphicx}
\usepackage{epstopdf}
\usepackage{tikz}
\usetikzlibrary{arrows,automata}

\usepackage{graphicx,epsfig,float}

 \usepackage{hyperref}
  \hypersetup{
colorlinks=true,
linkcolor=cyan,
citecolor=cyan
}

\usepackage{url}

\makeatletter
\def\url@leostyle{%
  \@ifundefined{selectfont}{\def\UrlFont{\sf}}{\def\UrlFont{\small\ttfamily}}}
\makeatother
\usepackage{lineno}

\usepackage[normalem]{ulem}

\newtheorem{thm}{Theorem}[section]
\newtheorem{lem}[thm]{Lemma}
\newtheorem{prop}[thm]{Proposition}
\newtheorem{cor}[thm]{Corollary}
\newtheorem{fact}[thm]{Fact}

\newtheorem{clm}[thm]{Claim}
\newtheorem*{mthm*}{Main theorem}

\theoremstyle{definition}

\newtheorem{defn}[thm]{Definition}

\newtheorem{nrmk}[thm]{Remark}
\newtheorem{expl}[thm]{Example}
\newtheorem{expls}[thm]{Examples}
\newtheorem{nt}[thm]{Notations}

\newtheorem*{thm*}{Theorem}

\newcommand{\into}{\longrightarrow}

\renewcommand{\hat}{\widehat}
\renewcommand{\tilde}{\widetilde}
\renewcommand{\bar}{\overline}

\newcommand{\NN}{\mathbb{N}}
\newcommand{\ZZ}{\mathbb{Z}}
\newcommand{\QQ}{\mathbb{Q}}
\newcommand{\RR}{\mathbb{R}}
\newcommand{\CC}{\mathbb{C}}
\newcommand{\UU}{\mathbb{U}}
\newcommand{\VV}{\mathbb{V}}
\newcommand{\Aa}{\mathbb{A}}
\newcommand{\Pp}{\mathbb{P}}

\newcommand{\an}{\mathrm{an}}

\newcommand{\bF}{{\mathbf F}}

\newcommand{\bM}{{\mathbf M}}

\newcommand{\bI}{{\mathbf I}}

\newcommand{\cA}{\mbox{$\mathcal{A}$}}
\newcommand{\cB}{\mbox{$\mathcal{B}$}}
\newcommand{\cC}{\mbox{$\mathcal{C}$}}

\newcommand{\cL}{\mbox{$\mathcal{L}$}}
\newcommand{\K}{\mbox{$\mathcal{K}$}}
\newcommand{\cM}{\mbox{$\mathcal{M}$}}

\newcommand{\cR}{\mbox{$\mathcal{R}$}}
\newcommand{\cU}{\mbox{$\mathcal{U}$}}
\newcommand{\cV}{\mbox{$\mathcal{V}$}}
\newcommand{\cG}{\mbox{$\mathcal{G}$}}
\newcommand{\cF}{\mbox{$\mathcal{F}$}}
\newcommand{\cI}{\mbox{$\mathcal{I}$}}

\newcommand{\T}{\mbox{$\mathcal{T}$}}

\newcommand{\cO}{\mbox{$\mathcal{O}$}}

\newcommand{\bG}{\mbox{$\mathbf{\Gamma}$}}
\newcommand{\bGi}{\mbox{$\mathbf{\Gamma}_{\infty} $}}

\newcommand{\nG}{\mbox{$\Gamma$}}
\newcommand{\nGi}{\mbox{$\Gamma _{\infty} $}}
\newcommand{\nGm}{\mbox{$\Gamma^m$}}
\newcommand{\nGim}{\mbox{$\Gamma^m_{\infty} $}}

\newcommand{\nGil}{\mbox{$\Gamma^l_{\infty} $}}

\newcommand{\nGk}{\mbox{$\Gamma^k$}}
\newcommand{\nGik}{\mbox{$\Gamma^k_{\infty} $}}

\newcommand{\nGn}{\mbox{$\Gamma^n$}}
\newcommand{\nGin}{\mbox{$\Gamma^n_{\infty} $}}

\newcommand{\nGimk}{\mbox{$\Gamma^{m-k}_{\infty} $}}

\newcommand{\nGL}{\mbox{$\Gamma^{|L|}$}}
\newcommand{\nGiL}{\mbox{$\Gamma^{|L|}_{\infty} $}}

\newcommand{\nGipiL}{\mbox{$\Gamma^{|\pi(L)|}_{\infty} $}}

\newcommand{\nGiqiL}{\mbox{$\Gamma^{|\pi'(L)|}_{\infty} $}}

\newcommand{\bSg}{\mbox{$\mathbf{\Sigma}$}}
\newcommand{\nSg}{\mbox{$\Sigma$}}
\newcommand{\nSgm}{\mbox{$\Sigma ^m$}}

\newcommand{\nSgk}{\mbox{$\Sigma^k$}}
\newcommand{\nSgmk}{\mbox{$\Sigma ^{m-k}$}}

\newcommand{\fT}{{\mathfrak T}}
\newcommand{\tfT}{\widetilde{{\mathfrak T}}}

\renewcommand{\mod}{\mathrm{Mod}}

\newcommand{\cov}{\mathrm{Cov}}
\newcommand{\coh}{\mathrm{Coh}}

\newcommand{\op}{\mathrm{Op}}
\newcommand{\rc}{\mathbb{R}\textrm{-}\mathrm{c}}

\newcommand{\df}{\mathrm{def}}
\newcommand{\Df}{\mathrm{Def}}

\newcommand{\tDf}{\tilde{\mathrm{Def}}}
\newcommand{\hDf}{\hat{\mathrm{Def}}}

\newcommand{\vgs}{\mathrm{v\!+\!g}}
\newcommand{\hvgs}{\widehat{\mathrm{v\!+\!g}}}

\newcommand{\ho}{\mathcal{H}\mathit{om}}

\newcommand{\Ho}{\mathrm{Hom}}

\renewcommand{\top}{\mathrm{top}}

\newcommand{\id}{\mathrm{id}}

\newcommand{\val}{\mathrm{val}}
\newcommand{\res}{\mathrm{res}}

\newcommand{\tp}{\mathrm{tp}}

\newcommand{\iso}{\stackrel{\sim}{\to}}

\newcommand{\supp}{\mathrm{supp}}

\newcommand{\imin}[1]{#1^{-1}}
\newcommand{\lind}[1]{\underset{#1}{\underrightarrow{\lim}}}
\newcommand{\Lind}{\underrightarrow{\lim}}  
\newcommand{\indl}[1]{\underset{#1}{``\underrightarrow{\mathrm{lim}}\mbox{''}}}

\newcommand{\lpro}[1]{\underset{#1}{\underleftarrow{\lim}}}
\newcommand{\exs}[3]{0 \to {#1} \to {#2} \to {#3} \to 0}
\newcommand{\lexs}[3]{0 \to {#1} \to {#2} \to {#3}}

\begin{document}

\title{Cohomology of algebraic varieties over non-archimedean fields}

\author[Pablo Cubides Kovascics]{Pablo Cubides Kovacsics}
\address{Pablo Cubides Kovacsics, Mathematisches Institut der Heinrich-Heine-Universit\"at D\"usseldorf, 
Universit\"atsstr. 1, 40225 D\"usseldorf, Germany. }
\email{cubidesk@hhu.de}

\author[M\'ario Edmundo]{M\'ario Edmundo}
\address{M\'ario Edmundo, Departamento de Matem\'atica\\
Faculdade de Ci\^encias,  Universidade de Lisboa\\
Campo Grande, Edif\'icio C6\\ 
P-1749-016 Lisboa, Portugal.
}
\email{mjedmundo@fc.ul.pt}

\author[Jinhe Ye]{Jinhe Ye}
\address{Jinhe Ye, Institut de Math\'ematiques de Jussieu -- Paris Rive Gauche}
\email{jinhe.ye@imj-prg.fr}

\date{\today}



\begin{abstract}
We develop a sheaf cohomology theory of algebraic varieties over an algebraically closed non-trivially valued non-archimedean field $K$ based on Hrushovski-Loeser's stable completion. In parallel, we develop a sheaf cohomology of definable subsets in o-minimal expansions of the tropical semi-group $\Gamma_\infty$, where $\Gamma$ denotes the value group of $K$. For quasi-projective varieties, both cohomologies are strongly related by a deformation retraction of the stable completion homeomorphic to a definable subset of $\Gamma_\infty$. In both contexts, we show that the corresponding cohomology theory satisfies the Eilenberg-Steenrod axioms, finiteness and invariance, and we provide natural bounds of cohomological dimension in each case. As an application, we show that there are finitely many isomorphism types of cohomology groups in definable families. Moreover, due to the strong relation between the stable completion of an algebraic variety and its analytification in the sense of V. Berkovich, we recover and extend results on the topological cohomology of the analytification of algebraic varieties concerning finiteness and invariance. 
\end{abstract}

\maketitle
\setcounter{tocdepth}{1}
{
  \hypersetup{linkcolor=black}
  \tableofcontents
}

\section{Introduction}\label{section intro}

Let $V$ be an algebraic variety over a rank 1 non-archimedean field $K$. There are two well-behaved cohomology theories which can be attached to the analytification $V^\an$ of $V$ in the sense of V. Berkovich: the singular cohomology and the \'etale cohomology. Both have been proven to carry interesting information about $V$ (see \cite{ber-weight0}, \cite{ber-etale}, \cite{ber-etaleII}, to cite just a few). 

When $K$ is a non-archimedean field of higher rank, E. Hrushovski and F. Loeser \cite{HrLo} introduced a topological space $\widehat{V}(K)$ called the \emph{stable completion of $V$}, which can be thought as a higher rank analogue of Berkovich's analytification. If $V$ is a quasi-projective variety, the main theorem in \cite{HrLo} shows a very deep connection between $\widehat{V}(K)$ and the tropical semi-group $\Gamma_\infty$ associated to the value group $\Gamma$ of $K$: there is a deformation retraction from $\widehat{V}(K)$ to a piece-wise semi-linear subset of a finite power of $\Gamma_\infty$. 
An analogous result was earlier proved by Berkovich \cite{Berkovich_contractible} for $V^\an$ under strong algebraic restrictions on $V$.  

The aim of this article is to introduce a well-behaved cohomology theory of the space $\widehat{V}(K)$ which coincides with the singular cohomology when $K$ has rank 1. In particular, we show our cohomology satisfies finiteness, invariance and has natural bounds on cohomological dimension. As an application of the above, we extend results of Berkovich on the singular cohomology, and obtain a tameness result in families showing there are only finitely many cohomology groups in fibers.   

Typically (e.g., when $K$ has rank bigger than one), the underlying topological space of $\widehat{V}(K)$ is not even locally compact (see later Remark \ref{rem:connected}). We bypass such a problem by introducing a suitable site on $\widehat{V}(K)$ (the $\hvgs$-site) and using a formalism of sheaves developed by M. Edmundo and L. Prelli (\cite{ep2}), which builds on a modification of M. Kashiwara and P. Schapira's (\cite{ks2}). Recent applications such as \cite{ba-bru-tsi} and \cite{benoi-witten-I,benoi-witten-II} are implicitly working within this formalism. 

 Our approach is reminiscent to J. Tate's initial attempt of studying rigid analytic spaces via $G$-topologies \cite{tate1971}. However, it is worthwhile to mention that, in contrast with cohomology theories defined via the analytification of algebraic varieties, our definition is intrinsically algebraic and makes no use of Tate algebras or affinoid domains. This confirms Hrushovski-Loeser's idea that a ``rigid \emph{algebraic} geometry exists as well'' (see \cite[Chapter 1]{HrLo}).

\section{Main results}

\subsection{Cohomology theories}

Let $K$ be an algebraically closed non-archimedean field of arbitrary rank, $\Gamma$ be its value group and $\Gamma_{\infty}=\Gamma\cup\{\infty\}$ the associated tropical semi-group. In order to fully exploit Hrushovski-Loeser's main theorem, part of the present article consists in developing in parallel a sheaf cohomology of definable (i.e. semi-linear) sets of $\Gamma_\infty$. More generally, we lay out a sheaf cohomology for definable sets in o-minimal expansions of $\Gamma_\infty$ over the site of open definable sets, extending  theorems from \cite{ejp,ep1,ep3,ep4} on sheaves on o-minimal expansions of $\Gamma$, including the formalism of the six Grothendieck operations. It is worth mentioning that the passage from $\Gamma$ to $\Gamma_\infty$ is more subtle than one might expect, as already noted by Hrushovski and Loeser (see the first paragraph in \cite[Section 4.1]{HrLo}). 

The following is the main theorem concerning our cohomology theories:

\begin{mthm*}
For algebraic varieties (possibly over a subfield of $K$), their stable completions and definable subsets of $\Gamma_\infty$, the corresponding cohomology theory satisfies: 
\begin{enumerate}
    \item the Eilenberg-Steenrod axioms (Theorem \ref{thm:Eilenberg-Steenrod});%
    \item finiteness (resp. Theorems \ref{thm finiteness in acvf}  and \ref{thm fg coho def groups}); 
    \item invariance (resp. Theorems \ref{thm inv in acvf} and \ref{thm bf inv gp-int});
    \item bounds on cohomological dimensions (resp. Theorems \ref{thm vanish vg-normal}, \ref{thm vanish vg-loc closed} and \ref{thm fg coho def groups}).
\end{enumerate}
\end{mthm*}


\subsection{Applications} We obtain the following applications:

\medskip
\noindent
(I) In rank 1, Hrushovski and Loeser show finiteness of homotopy types in definable families (see \cite[Theorems 14.3.1 and 14.4.4]{HrLo}). As they observed (see the beginning of \cite[Section 14.3]{HrLo}), such a result is no longer true in higher ranks. However, using invariance of cohomology with definably compact support, we obtain a tameness result showing that
given a definable family $Z_w$ of $\vgs$-locally closed  subsets  of a variety, as $w$ runs through $W$ there are finitely many isomorphisms types for the  $\hvgs$-cohomology $H^*_c(\widehat{Z}_{w};{\mathbb Q})$ with definably compact supports. (See later Theorem \ref{thm: finite cohom complexity hat}). 

Results of A. Abbes and T. Saito \cite{abbes-saito}, later generalized by J. Poineau \cite{poineau}, showed that for families parameterized by the norm of an analytic function, there is a finite partition of $\mathbb{R}$ into intervals such that on each piece $\pi_0$ of the fibers are in canonical bijections. For definable families parameterized by $\nGi$, we get analogues of such results for $H_c^*$ and $\pi_0^\df$ (both independently of the rank). See Corollaries \ref{cor: finite cohom complexity hat} and \ref{cor: finite pi0 complexity hat}.

It is worth pointing out that assuming finiteness and invariance of cohomology without support, similar methods lead to bounds on the $\vgs$-Betti numbers of a uniformly definable family of $\vgs$-locally closed subsets of a variety. Such a result will be an analogue, in all ranks, of a result obtained by S. Basu and D. Patel (\cite[Theorem 2]{basu}) in the case where the valued field has rank one, using the usual topological (singular) Betti numbers instead. In comparison with \cite[Theorem 2]{basu}, this approach will provide a much simpler proof by taking advantage of the fact that we can work in higher elementary extensions where we can endow the the tropical semi-group $\bGi$ with the structure of a real closed field, and therefore apply earlier results of Basu himself on bounds in o-minimal expansions of real closed fields (\cite[Theorem 2.2]{basu-o-minimal}).  

\medskip

\noindent
(II)  We recover results of Berkovich \cite{Berkovich_contractible} on the topological cohomology of $V^\an$ concerning finiteness and invariance (see later Theorems \ref{thm finiteness and inv for Bf} and \ref{thm finiteness and inv for Bf with comp supp} and Corollaries \ref{cor spec seq bf and hvgs} and \ref{cor spec seq bf and hvgs c-sup}). 

\medskip
\noindent
(III) We extend Berkovich's results in two different ways which we briefly explain. Let $F$ be a non-archimedean field of rank 1 (not necessarily complete) and $V$ be an algebraic variety over $F$. Hrushovski and Loeser also introduced a topological space, called the \emph{model-theoretic Berkovich space associated to $V$}, which they denoted by $B_\bF(V)$ (see later Section \ref{sec:rel-berkovich} for its formal definition). When $F$ is complete, $B_\bF(V)$ is homeomorphic to the underlying topological space of $V^\an$. Letting $F^{\max}$ be a maximally complete algebraically closed extension of $F$ with value group $\RR$, they show the existence of a closed continuous surjection $\pi\colon \widehat{V}(F^{\max})\to B_\bF(V)$. Such a map allowed them to transfer results from $\widehat{V}(F^{\max})$ to $B_\bF(V)$, and to draw conclusions about the homotopy type of $V^\an$ for $V$ quasi-projective. When $F=F^{\max}$, the topological spaces $\widehat{V}(F)$, $B_\bF(V)$ and $V^\an$ are all homeomorphic, and we show that, in this particular case, our cohomology coincides with the topological cohomology of $V^\an$ (Theorems \ref{thm spec seq bf and hvgs c-sup} and \ref{thm spec seq bf and hvgs}). Furthermore, the map $\pi$ allowed us to transfer results from the $\hvgs$-cohomology on $\widehat{V}(F^{\max})$ to the topological cohomology of $B_\bF(V)$ and hence, when $F$ is complete, to the topological cohomology of $V^\an$. This is mainly how we recover Berkovich's results. Note however, that our results hold more generally for $B_\bF(V)$, a space which is well-defined without assuming $F$ to be complete. 

A second way in which we extend Berkovich's results is that our theorems actually hold for semi-algebraic subsets of $B_\bF(V)$ (resp. for semi-algebraic subsets of $\widehat{V}(K)$, for any non-trivially valued algebraically closed field extension of $F$). Here, semi-algebraic subset should be understood in the sense of A. Ducros \cite{duc} (see later Fact \ref{fact B_F(V) and V^an}). Using that $\widehat{V}(K)$ has finitely many definable connected components (Lemma \ref{lem vgs components}), we also recover the analogue result due to Ducros \cite{duc} for $V^\an$. In a similar vein, but using completely different methods, F. Martin \cite{martin} extended results of Berkovich concerning the \'etale cohomology of analytic spaces to the larger category of their semi-algebraic subsets. We would like to explore in subsequent research if the techniques here employed can also be used to develop an analogue of the \'etale cohomology over the spaces $\widehat{V}(K)$ and $B_\bF(V)$.

\subsection{Layout of the article}

Throughout, we will use a formalism of sheaves developed by M. Edmundo and L. Prelli (\cite{ep2}), which builds on a modification of M. Kashiwara and P. Schapira's (\cite{ks2}) notion of $\T$-topology (a Grothendieck topology). The semi-algebraic site, the sub-analytic site, the o-minimal site and the $\hvgs$-site mentioned above will be examples of such $\T$-topologies. The needed background and properties of such a formalism are presented in Section \ref{section more on t-sheaves}. 

The novelty here is the introduction of the notions of $\T$-normality and of families of $\T$-normal supports on spaces equipped with a $\T$-topology. When we pass to the $\T$-spectrum of such Grothendieck topologies, the corresponding categories of sheaves are isomorphic, and we get normal spectral spaces (resp. families of normal and constructible supports on such spectral spaces). The later notion will play the role that paracompactifying families of supports plays in sheaf theory in topological spaces. We study the notion of $\Phi$-soft sheaves when $\Phi$ is a normal and constructible family of supports and setup the tools required to obtain Base change formula, Vietoris-Begle theorem and, later in Section \ref{section cohomo finiteness and invariance}, Mayer-Vietoris sequences and bounds on cohomological $\Phi$-dimension. These results were already known in the o-minimal case (\cite{ep1}) and in the semi-algebraic case (\cite{D3}) where $\T$-normality and  families of $\T$-normal supports corresponds to definably normal definable spaces and families of definably normal supports (resp. regular and paracompact (locally) semi-algebraic spaces and paracompactifying (locally) semi-algebraic families of supports). Here we present a unifying approach to this theory.

In Section \ref{section t-top in bGi}, we define the o-minimal site on a definable set in an o-minimal expansion $\bGi$ of $\nGi$, we recall the notions of definable connectedness and  definable compactness. Moreover, in order to be able to apply the tools of Section \ref{section more on t-sheaves}, we study the notion of definable normality in $\bGi$. Unlike in o-minimal expansions $\bG$ of $\nG$, in $\bGi$ there are open definable sets which are not definably normal. The main result of  Section \ref{section t-top in bGi}, which ensures that we can later apply the tools of Section \ref{section more on t-sheaves}, is  Theorem \ref{thm basis of open normal} showing that every  definably locally closed set in $\bGi$ is the union of finitely many relatively open definable subsets which are definably normal. The proof of this result is rather long and is based in ideas from \cite{ep4}. 

The $\hvgs$-site on the stable completion $\widehat{V}(K)$ of an algebraic variety $V$ over an algebraically closed non-archimedean field $K$ (and more generally on the stable completion $\widehat{X}$ of a definable subset $X\subseteq V\times \Gamma_\infty^m$) is defined and studied in Section \ref{section t-top in acvf}. Here, for the readers convenience, we recall some background from \cite{HrLo} on stable completions and on the notions of definable connectedness and definable compactness. In particular, we include a proof (missing in \cite{HrLo}) of the fact that the stable completion $\widehat{X}$ of a definable subset  $X\subseteq V\times \nGin$ has finitely many definably connected components (Lemma \ref{lem vgs components}). The main result of this section is Corollary \ref{cor:mix-normal} showing  that if $X$ is a  $\vgs$-locally closed subset, then $X$ is the union of finitely many basic $\vgs$-open subsets which are weakly $\vgs$-normal. When we pass to $\widehat{X}$ equipped with the $\hvgs$-site this gives the required conditions in order to apply the tools of Section \ref{section more on t-sheaves}.

Section \ref{section cohomo finiteness and invariance} is devoted to showing the Eilenberg-Steinrod axioms, finiteness,  invariance and vanishing results for the associated sheaf cohomologies in the algebraic and $\bGi$ cases. The main difficulty concerns the homotopy axiom which, as usual, is proved using the Vietoris-Begle theorem. However, extra work is needed to verify the assumptions of the version of this theorem presented in Section \ref{section more on t-sheaves}. The results concerning finiteness and invariance in $\bGi$ are obtained adapting the methods used in \cite{bf} and \cite{ep3}. After verifying that the strong deformation retraction given by the main theorem of \cite{HrLo} is a morphism of $\hvgs$-sites, we obtain finiteness and invariance for stable completions in the quasi-projective case. The general case is then obtained using  Mayer-Vietoris sequences based on the tools of Section \ref{section more on t-sheaves}.

Finally, Sections \ref{sec:rel-berkovich} and \ref{sec Betti} gather the above mentioned applications: the former is devoted to the relation with Berkovich spaces and the latter to the tameness results on cohomological complexitity in families. 

\section{\texorpdfstring{$\T$}{T}-spaces and~\texorpdfstring{$\T$}{T}-sheaves}\label{section more on t-sheaves}
In this section we will recall the notions of $\T$-space and $\T$-sheaves as well as some of the basic results obtained in \cite{ep2}. We then extend the theory of $\T$-sheaves by introducing the notion of families of $\T$-normal supports and proving several results in this context generalizing those already known in the o-minimal or locally semi-algebraic cases.

\begin{nt}
Below we let  $A$ be a commutative ring with unit.  If  $X$ is a topological space  (not necessarily Hausdorff) we denote by $\op(X)$ the category whose objects are the open subsets of $X$ and the morphisms are the inclusions, and we let  $\mod(A_X)$ be the category of sheaves of $A$-modules on $X$. 
The category $\mod(A_X)$  is a Grothendieck category, it admits enough injects, admits a family of generators, filtrant inductive limits are exact and we have on it  the classical operations
\[\ho_{A_X}(\bullet ,\bullet ), \,\, \bullet \otimes _{A_X}\bullet , \,\, f_*,\,\, f^{-1}, \,\, (\bullet )_Z, \,\,\Gamma _Z(X;\bullet ), \,\, \Gamma (X; \bullet )\]
where $Z\subseteq X$ is a locally closed subset. 

Recall also that if $Z\subseteq X$ is a locally closed subset and $i\colon Z\to X$ is the inclusion, then we have the operation 
\[i_{!}(\bullet )\colon\mod(A_Z)\to \mod(A_X)\]
 of {\it extension by zero} such that for $\cF\in \mod(A_Z),$ $i_{!}\cF$ is the unique sheaf in $\mod(A_X)$ inducing $\cF$ on $Z$ and zero on $X\setminus Z$. 
 If $\cG\in \mod(A_X),$ then
 \begin{equation}\label{loc closed rest}
\cG_Z\simeq  \cG\otimes A_Z\simeq i_{!}\circ i^{-1}\cG.
\end{equation}
If $f\colon X \to Y$ is a continuous map, $Z$ is a locally closed subset of $Y,$ 
\begin{equation*}
\xymatrix{
f^{-1}(Z)  \ar@{^{(}->}[r]^{j}  \ar[d]^{f_|} & X \ar[d]^{f} \\
Z \ar@{^{(}->}[r]^{i} & Y
}
\end{equation*}
is a commutative diagram and $\cF\in \mod (A_{Z}),$ then 
\begin{equation}\label{loc closed base change}
f^{-1}\circ i_{!}\cF\simeq  j_{!}\circ (f_|)^{-1}\cF.
\end{equation}
\end{nt}

Below we use freely these operations and refer the reader to \cite[Chapter II, Sections 2.1 - 2.4]{ks1} or to \cite[Chapter I, Sections 1 - 6]{b} (but the notation is different) for the details on  sheaves on topological spaces and on the properties of  these basic operations. 


\subsection{\texorpdfstring{$\T$}{T}-sheaves}\label{subsection t-sheaves}

Here we recall the definition of $\T$-space given in \cite{ep2}, adapting the construction of Kashiwara and Schapira \cite{ks2}, and we recall some of the results obtained in that paper for  the category of sheaves on $X_{\T}$. Those results generalize results from either the case of sub-analytic sheaves, semi-algebraic sheaves or and o-minimal sheaves. See Examples \ref{expls old tspaces} below.

\begin{defn}\label{def:T-topology} 
Let $X$ be a topological space and let $\T \subseteq \op(X)$ be a family of open subsets of $X$ such that: 
\begin{itemize}
    \item[(i)] 
    $\T$ is a basis for the topology of $X$, and $\varnothing \in \T$,
    \item[(ii)] 
    $\T$ is closed under finite unions and intersections.
\end{itemize}
Then we say that:
\begin{itemize}
\item
a $\T$-subset is a finite Boolean combination of elements of $\T$;
\item
a closed (resp. open) $\T$-subset is a $\T$-subset which is closed (resp. open) in $X$;
 \item
a $\T$-connected subset is a $\T$-subset which is not the disjoint union of two proper clopen $\T$-subsets (in the induced topology).
\end{itemize}
If  in addition $\T$  satisfies:
\begin{itemize}\label{hytau}
    \item[(iii)]  
    every $U \in \T$ has finitely many $\T$-connected components which are in $\T$,

\end{itemize}
then we say that $X$ is a {\it $\T$-space}.
\end{defn}

Let $X$ and $\T\subseteq \op (X)$ satisfying (i) and (ii) of Definition \ref{def:T-topology}. One can endow the category $\T$ with a Gro\-then\-dieck topology, called the {\it $\T$-topology}, in the following way: a
family $\{U_i\}_i$ in $\T$ is a covering of $U \in \T$ if it admits a finite subcover. We denote by $X_{\T}$ the associated site,  write for short $A_{\T}$ instead of $A_{X_{\T}}$, and let  $\rho\colon X \to X_{\T}$ be the natural morphism of sites.
We have functors
\begin{equation*}\label{rho}
\xymatrix{\mod(A_X)
\ar@ <2pt> [r]^{\mspace{0mu}\rho_*} &
  \mod(A_{\T}) \ar@ <2pt> [l]^{\mspace{0mu}\imin \rho} }
\end{equation*}
with  $\imin \rho \circ \rho_* \simeq \id$ (equivalently, the functor $\rho_*$ is fully faithful). See \cite[Proposition 2.1.6]{ep2}.\\

By \cite[Propositions 2.1.7 and 2.1.8]{ep2} we have:\\

\begin{fact}\label{UlimU} 
$\,$
\begin{enumerate}
\item
Let $\{\cF_i\}_{i \in I}$ be a filtrant inductive
system in $\mod(A_{\T})$ and let $U \in \T$. Then
\[\lind i\Gamma(U;\cF_i) \iso \Gamma(U;\lind i \cF_i).\]
\item
Let $\cF$ be a presheaf on $X_{\T}$ and
assume that
\begin{itemize}
\item[(i)] $\cF(\varnothing)=0,$
\item[(ii)] For any $U,V \in \T$ the sequence $\lexs{\cF(U\cup V)}{\cF(U) \oplus \cF(V)}{\cF(U \cap
V)}$is exact.
\end{itemize}
Then $\cF \in \mod(A_{\T})$. \qed
\end{enumerate}
\end{fact}

Let us consider the category $\mod(A_X)$ of sheaves of $A_X$-modules on $X$, and denote by $\K$ the subcategory whose
objects are the sheaves $\cF=\oplus_{i \in I} A_{U_i}$ with $I$
finite and $U_i \in \T$ for each $i$. 

Following  \cite{ks2} one defines:

\begin{defn} 
Let $X$ be a $\T$-space, and let $\cF \in \mod(A_X)$. Then we say that:
\begin{itemize}
\item[(i)] $\cF$ is $\T$-finite if there exists an epimorphism
$\cG \twoheadrightarrow \cF$ with $\cG \in \K$.
\item[(ii)] $\cF$ is $\T$-pseudo-coherent if for any morphism
$\psi\colon \cG \to \cF$ with $\cG \in \K$, $\ker \psi$ is $\T$-finite.
\item[(iii)] $\cF$ is $\T$-coherent if it is both $\T$-finite
and $\T$-pseudo-coherent.
\end{itemize}
One denotes by
$\coh(\T)$ the full subcategory of $\mod(A_X)$ consisting of
$\T$-coherent sheaves. 
\end{defn}


By results in \cite[Subsection 2.2]{ep2} we have:

\begin{fact}\label{fact coh stable}
Suppose that $X$ is a $\T$-space. Then the category $\coh(\T)$ is additive and stable by finite sums, kernels, cokernels and  $\bullet \otimes_{A_X} \bullet $ in $\mod(A_X)$. In particular, it is abelian and the natural functor $\coh(\T)\to \mod(A_X)$ is exact. 

Moreover, the functor $\rho_*$ is fully faithfull and exact on $\coh(\T)$, $\coh(\T)$ contains $\K$ and every $\cF\in \coh(\T)$ admits a finite resolution 
\[
\cG^{\bullet}\coloneqq 0 \to \cG^1 \to \cdots \to \cG^n \to \cF \to 0
\]
consisting of objects belonging to $\mathcal{K}$.\qed
\end{fact}

As in \cite{ks2}, we can define the indization of the category $\coh(\T)$. Recall that the category 
\[
{\rm Ind}(\coh(\T)),
\] 
of ind-$\T$-coherent sheaves is the category whose objects are filtrant inductive limits of functors
\[
\lind i \Ho_{\coh(\T)}(\bullet,\cF_i) \ \ \ \ \text{($\indl i \cF_i$ for short)},
\]
where $\cF_i \in \coh(\T)$, and the morphisms are the natural transformations of such functors. Note that since $\coh(\T)$ is a small category, ${\rm Ind}(\coh(\T))$ is equivalent to the category of $A$-additive left exact contravariant functors from $\coh(\T)$ to $\mod(A)$.  See \cite{ks3} for a complete exposition on indizations of categories.

We can extend the functor $\rho_*\colon\coh(\T)\to \mod(A_{\T})$ to 
\[\lambda:{\rm Ind}(\coh(\T)) \to \mod(A_{\T})\]
 by setting $\lambda(\indl i \cF_i)\coloneqq \lind i \rho_* \cF_i$.\\

By \cite[Corollary 2.2.10]{ep2} we have:\\

\begin{fact}
Suppose that $X$ is a $\T$-space. Then the functor  $\lambda:{\rm Ind}(\coh(\T)) \to \mod(A_{\T})$ is an equivalence of categories. In particular, for every $\cF \in \mod(A_{\T})$ there exists a small filtrant
inductive system $\{\cF_i\}_{i \in I}$ in $\coh(\T)$ such
that $\cF \simeq \lind i \rho_*\cF_i$.\qed
\end{fact}

Recall that any Boolean algebra $\cA$ has an associated topological space, that we denote by $S(\cA)$, called its Stone space. The points in $S(\cA)$ are the ultrafilters $\alpha$ on $\cA$.
The topology on $S(\cA)$ is generated by a basis of open and closed sets consisting of all sets of the form
\[
\widetilde{C} = \{\alpha \in S(\cA): C \in \alpha \},
\]
where $C\in \cA$. The space
$S(\cA)$ is a compact\footnote{Throughout the paper, for topological spaces, we say compact to mean quasi-compact and Hausdorff.} totally disconnected space. Moreover, for each $C \in \cA$, the subspace $\widetilde{C}$ is compact. 
See \cite{Jo82} for an introduction to this subject. \\

Let $X$ and $\T\subseteq \op (X)$ satisfying (i) and (ii) of Definition \ref{def:T-topology}. Then as in \cite{ep2}  we let
\begin{equation*}\label{Tloc}
\T _{loc}=\{U\in \op(X): U \cap W \in \T\,\,\, \textrm{for every} \,\,\, W \in \T\}
\end{equation*}
and we make the following definitions:
\begin{itemize}
\item
a subset $S$ of $X$ is a $\T_{loc}$-subset if and only if  $S\cap V$ is a $ \T$-subset for every $V\in \T$;

\item
a closed (resp. open) $\T_{loc}$-subset is a $\T_{loc}$-subset which is closed (resp. open) in $X$;

\item
a $\T_{loc}$-connected subset is a $\T_{loc}$-subset which is not the disjoint union of two proper clopen $\T_{loc}$-subsets (in the induced topology).
\end{itemize}

Clearly, $\varnothing, X \in \T_{loc}$, $\T\subseteq \T_{loc}$ and $\T_{loc}$ is closed under finite intersections. Moreover, if $\{S_i\}_i$ is a family of $\T_{loc}$-subsets such that $\{i:S_i\cap W\neq \emptyset \}$ is finite for every $W\in \T$, then the union and the intersection of the family $\{S_i\}_i$ is a $\T_{loc}$-subset. Also the complement of a $\T_{loc}$-subset is a $\T_{loc}$-subset. Therefore the $\T_{loc}$-subsets  form a Boolean algebra.

\begin{defn}\label{def:tilde}
Let $X$ be a topological space with $\T\subseteq \op (X)$ as above and let $\cA$ be the Boolean algebra of $\T_{loc}$-subsets of $X$. The topological space $\widetilde{X}_{\T}$ is the data consisting of:
\begin{itemize}
\item the points of $S(\cA)$ such that $U \in \alpha$ for some $U \in \T$,
\item a basis for the topology is given by  the family of subsets $\{\widetilde{U}:U\in\T\}$.
\end{itemize}
We call $\tilde{X}_{\T}$ the $\T$-spectrum of $X$.
\end{defn}


With this topology, for $U \in \T$, the set $\widetilde{U}$ is quasi-compact in $\widetilde{X}_{\T}$ since it is quasi-compact in $S(\cA)$. Hence:

\begin{fact}\label{fact spec spaces} 
$\widetilde{X}_{\T}$ has a basis of quasi-compact open subsets given by $\{\widetilde{U}:U\in\T\}$ closed under finite intersections. Moreover, if $X\in \T$, then $\widetilde{X}_{\T}=\tilde{X}$  is a spectral topological space.\qed
\end{fact}

Furthermore, by \cite[Proposition 2.6.3]{ep2} we also have:

\begin{fact}\label{isosites} 
There is an equivalence of categories 
\begin{equation*}\label{sites iso}
\pushQED{\qed} 
\mod(A_{\T}) \simeq \mod(A_{\widetilde{X}_{\T}}).
\qedhere
\popQED
\end{equation*}
\end{fact}

Note that in the statement of \cite[Proposition 2.6.3]{ep2} $X$ is assumed to be a $\T$-space, but this is not necessary: all that is used is Fact \ref{fact spec spaces} and \cite[Corollary 1.2.11]{ep2} for $\tilde{V}$ with $V\in \T$.\\

We now recall the main  known examples of $\T$-spaces. 

\begin{expls}\label{expls old tspaces}
$\,$
\begin{enumerate}
\item
Let $R=(R,<, 0,1,+,\cdot )$ be a real closed field, $X$ be a locally semi-algebraic space and  $\T= \{U \in \op(X): U \,\, \text{is  semi-algebraic}\}$. Then $X$ is a $\T$-space, the associated site $X_{\T}$ is the semi-algebraic site on $X$ of \cite{D3,dk5} (the $\T$-subsets of $X$ are  the semi-algebraic subsets of $X$ (see \cite{BCR}) and the $\T_{loc}$-subsets of $X$ are  the locally semi-algebraic subsets of $X$ (see \cite{dk5}). When $X$ is semi-algebraic, then:
\begin{itemize}
\item[(i)] $\widetilde{X}_{\T} =\widetilde{X}$ is the semi-algebraic spectrum of $X$ from \cite{cr}. For example, if $V\subseteq R^n$ is an affine real algebraic variety over $R$, then $\tilde{V}$ is homeomorphic to ${\rm Sper}\,R[V]$, the real spectrum of the coordinate ring $R[V]$ of $V$ (see \cite[Chapter 7, Section 7.2]{BCR} or \cite[Chapter I, Example 1.1]{D3}).
\item[(ii)]
there is an equivalence of categories $\mod(A_{\T}) \simeq \mod(A_{\widetilde{X}})$ (see \cite[Chapter 1, Proposition 1.4]{D3})
\end{itemize}
\item
Let $X$ be a real analytic manifold and consider $\T=\{U \in \op(X): U\,\, $is  sub-analytic relatively compact$\}$. Then $X$ is a $\T$-space and the associated site $X_{\T}$ is the sub-analytic site $X_{\mathrm{sa}}$ of \cite{ks2,Pr1}. In this case the $\T_{loc}$-subsets are the sub-analytic subsets of $X$.

Moreover, the family $\coh(\T)$ corresponds to the family $\mod_{\rc}^c(A_X)$ of ${\mathbb R}$-constructible sheaves with compact support, and for any $\cF \in \mod(A_{X_{\mathrm{sa}}})$ there exists a filtrant inductive system $\{\cF_i\}_{i\in I}$ in $\mod_{\rc}^c(A_X)$ such that $\cF \simeq \lind i \rho_* \cF_i$.\\

Another example in the sub-analytic context is the conic sub-analytic site (\cite{Pr07}).

\item
Let $\bG= (\nG,<, \ldots )$ be an arbitrary o-minimal structure without end points, $X$ a definable space and  $\T=\{U \in \op(X): U\,\, \text{is  definable}\}$. Then $X$ is a $\T$-space and the associated site $X_{\T}$ is the o-minimal site $X_{\rm def}$ of \cite{ejp}. Note also that the $\T$-subsets of $X$ are exactly the definable subsets of $X$ (see Remark \ref{nrmk local def comp}). Therefore, $\widetilde{X}_{\df} =\widetilde{X}$ is the o-minimal spectrum of $X$ from \cite{p,ejp} (i.e., the points are types over $\nG$ concentrated on $X$). Moreover,  there is an equivalence of categories $\mod(A_{X_{\df}}) \simeq \mod(A_{\widetilde{X}})$ (\cite{ejp}).\\

\end{enumerate}
\end{expls}

\subsection{\texorpdfstring{$\T$}{T}-normality and families of \texorpdfstring{$\T$}{T}-supports}\label{subsection t-supp}
Let $\fT$ be the category whose objects are pairs $(X,\T)$ with $X$ a topological space, $\T \subseteq \op (X)$ is a family of open subsets of $X$ satisfying conditions (i) and (ii) of Definition \ref{def:T-topology} and such that $X\in \T$ and,  morphisms $f\colon(X, \T)\to (Y, \T')$ are (continuous) maps $f\colon X\to Y$ such that $f^{-1}(\T')\subseteq \T$.

Note that  a morphism $f\colon(X,\T)\to (Y,\T')$ of $\fT$ defines a morphism of sites 
\[f\colon X_{\T}\to Y_{\T'}\]
which extends to a homomorphism from the boolean algebra of $\T'$-subsets of $Y$ to the boolean algebra of $\T$-subsets of $X,$ namely, $f^{-1}(C)$ is a $\T$-subset of $X$ whenever $C$ is a $\T'$-subset of $Y$. It follows that  we have an induced map 
\[\tilde{f}\colon \widetilde{X}_{\T}\to \widetilde{Y}_{\T'}\]
between the corresponding spectral topological spaces, where $\tilde{f}(\alpha )$ is the ultrafilter in $\widetilde{Y}_{\T'}$ given by the collection 
 \[\{C: C\,\, \textrm{is a}\,\, \T'\textrm{-subset and}\,\, f^{-1}(C)\in \alpha \}.\] 
 The map $\tilde{f}\colon\tilde{X}_{\T}\to \tilde{Y}_{\T'}$ is continuous since 
if  $V\in \T',$  then clearly $\tilde{f}^{-1}(\tilde{V})=\tilde{f^{-1}(V)}$.

Later we will also require the following generalization of the above constructions:

\begin{nrmk}\label{nrmk site iso tilde}  
Let  $\zeta \colon X_{\T}\to Y_{\T'}$ be a morphism of sites which extends to a homomorphism from the boolean algebra of $\T'$-subsets of $Y$ to the boolean algebra of $\T$-subsets of $X$. If $\zeta (Y)=X$ and $\zeta (A)\subseteq \zeta (B)$ whenever $A$ and $B$ are $\T'$-subsets such that $A\subseteq B,$ then $\zeta $ induces a continuous map $\tilde{\zeta}\colon\tilde{X}_{\T}\to \tilde{Y}_{\T'}$ where $\tilde{\zeta }(\alpha )$ is the ultrafilter of $\T'$-subsets given by the collection 
 \[\{C: C\,\, \textrm{is a}\,\, \T'\textrm{-subset and}\,\, \zeta (C)\in \alpha \}.\] 
If in addition $\zeta \colon X_{\T}\to Y_{\T'}$ is an isomorphism of sites, then $\tilde{\zeta }$ is a homeomorphism.  
\end{nrmk}

Recall also the following facts:

\begin{nrmk}[Constructible subsets and topology]\label{nrmk constructible}
Given $(X,\T)$  an object of $\fT,$ then $\widetilde{X}$ equipped with the constructible topology is a compact totally disconnected  space in which each constructible subset is clopen. Recall that by a {\it constructible subset of $\tilde{X}$} we mean a subset of the form $\tilde{Z}$ where $Z$ is a $\T$-subset of $X$ and,  the {\it constructible topology on $\tilde{X}$} is the topology generated by the constructible subsets of $\tilde{X}$.  \\
\end{nrmk}
 
\begin{nrmk}\label{tilde on sets and maps}
Let $(X, \T)$ be an object of $\fT$. Then the restriction of the tilde functor $\fT \to \tfT $ induces an isomorphism between the boolean algebra of $\T$-subsets of $X$ and  the boolean algebra of constructible subsets of its $\T$-spectrum $\tilde{X}$ preserving open (resp. closed) and, $X$ is $\T$-connected if and only if its  $\T$-spectrum $\tilde{X}$ is connected.  Moreover,  if $f\colon(X,\T)\to (Y,\T')$ is a morphism of $\fT$ then:
\begin{itemize}
\item
if  $C$ is a $\T'$-subset of $Y,$ then $\tilde{f^{-1}(C)}=\tilde{f}^{-1}(\tilde{C});$
\item
if $\beta \in \tilde{Y}$ then $\tilde{f}^{-1}(\beta )$ is a quasi-compact subset of $\tilde{X}$.
\end{itemize}
For the later observe that $\tilde{f}^{-1}(\beta )=\bigcap \{\tilde{f^{-1}(C)}:C\in \beta \}$ and so it is compact in the constructible topology.\\
\end{nrmk}

Below we denote by 
\[\fT \to \tfT \]
the functor sending an object $(X,\T)$ of $\fT$ to $\tilde{X}_{\T}$ and sending a morphism $f\colon(X,\T)\to (Y,\T')$ of $\fT$ to $\tilde{f}\colon\tilde{X}_{\T}\to \tilde{Y}_{\T'}$.  

The goal of this subsection is to extend to  $(\fT, \tfT)$ some results from \cite{ejp} and \cite{ep1}  proved for the pair $(\Df, \tDf)$ where $\Df$ is the subcategory of $\fT$ of Example \ref{expls old tspaces} (3), i.e. is the category of definable spaces in some fixed arbitrary o-minimal structure $\bG =(\nG, <,\ldots )$ without end points, with morphisms being continuous definable maps between such definable spaces,  and $\tDf$ is the image of $\Df$ under  $\fT \to \tfT$. 

For $(\fT, \tfT),$ just like in the case of $(\Df, \tDf),$ normality will play the role that  paracompactness plays in sheaf theory on topological spaces \cite[Chapter I, Section 6 and Chapter II, Sections 9 - 13]{b} or the role that  regular and paracompact plays in sheaf theory on locally semi-algebraic spaces \cite[Chapter II]{D3}. 

So we introduce and study a notion of normality adapted to $\fT:$

\begin{defn}\label{defn tnormal}
Let $(X, \T)$ be an object of $\fT$. We say that $X$ is {\it $\T$-normal} 
if and only if one of the following equivalent conditions holds:
\begin{enumerate}
\item
for every disjoint, closed $\T$-subsets $C$ and $D$ of $X$ there are disjoint, open $\T$-subsets $U$ and $V$ of $X$ such that $C\subseteq U$ and $D\subseteq V$. 
\item
for every  $S\subseteq X$ closed $\T$-subset  and $U\subseteq X$ open $\T$-subset such that $S\subseteq U$, there are an open $\T$-subset $W$  of $X$ and a closed $\T$-subset $K$ of $X$ such that $S\subseteq W\subseteq K\subseteq U$.
\item
for every open $\T$-subsets $U$ and $V$ of $X$ such that $X=U\cup V$ there are closed $\T$-subsets $A\subseteq U$ and $B\subseteq V$ of $X$ such that $X=A\cup B$.\\
\end{enumerate}
\end{defn}

We left the reader convince her/himself of the equivalence of (1)-(3). 


As in the case of $(\Df, \tDf)$ (\cite[Theorem 2.13]{ejp}) we have the following. Here we give  a simpler proof.

\begin{prop}\label{prop main normal def normal}
Let $(X,\T)$ be an object of $\fT$. Then  $X$ is $\T$-normal if and only if its $\T$-spectrum $\tilde{X}$ is a normal topological space.
\end{prop}

\begin{proof}
By Remark \ref{tilde on sets and maps} and quasi-compactness of $\tilde{X}$ it is immediate that if $\tilde{X}$ is normal then $X$ is $\T$-normal. It is a standard fact that a spectral topological space is normal if and only if any two distinct closed points can be separated by disjoint open subsets (see \cite[Proposition 2]{cr}). So let $\alpha $ and $\beta $ be two distinct closed points in $\tilde{X}$. Since $\alpha $ and $\beta $ are closed points,  we have $\{\alpha \}=\bigcap _{i\in I}C_i$ and $\{\beta \}=\bigcap _{j\in J}D_j$ where $C_i$'s and $D_j$'s are closed constructible subsets of $\tilde{X}$. Since $\alpha $ and $\beta $ are distinct, we have 
\[\bigcap \{C_i\cap D_j: i\in I, j\in J\}=\emptyset .\]
By quasi-compactness of $\tilde{X},$ there are $C=C_{i_1}\cap \ldots \cap C_{i_k}$ and $D=D_{j_1}\cap \ldots \cap D_{j_l}$ such that $C\cap D=\emptyset$. Since $C$ and $D$ are disjoint, closed constructible subsets of $\tilde{X},$ by Remark \ref{tilde on sets and maps},  we can use the $\T$-normality of $X$ to find disjoint, open constructible subsets of $\tilde{X}$ separating $C$ and $D,$ hence $\alpha $ and $\beta $.
\end{proof}

As usual normality implies the shrinking lemma, whose proof is standard, see for example, \cite[Proposition 2.17]{ejp} or \cite[Chapter 6, (3.6)]{vdd}. In the $\T$-spectra due to the quasi-compactness of the base, we get a stronger result:

\begin{cor}[Shrinking lemma]\label{cor shrinking lemma}
Suppose that $(X, \T)$ is an object of $\fT$ which is $\T$-normal. If $\{U_i:i=1,\dots ,n\}$ is a covering of $X$ (resp. of $\tilde{X}$) by open $\T$-subsets of $X$ (resp. open subsets of $\tilde{X}$), then there are open $\T$-subsets (resp. open constructible subsets) $V_i$ and closed $\T$-subsets (resp.
closed constructible subsets) $K_i$ of $X$ (resp. of $\tilde{X}$) ($1\leq i\leq n$) with $V_i\subseteq K_i\subseteq U_i$ and $X=\cup \{V_i:i=1,\dots, n\}$ (resp. $\tilde{X}=\cup \{V_i:i=1,\dots, n\}$).\qed
\end{cor}

Later we will require the following weaker notion:

\begin{defn}\label{defn wtnormal}
Let $(X, \T)$ be an object of $\fT$. We say that $X$ is {\it weakly $\T$-normal} 
if and only if one of the following equivalent conditions holds:
\begin{enumerate}
\item
for every disjoint, $\T$-closed subsets $C$ and $D$ of $X$ there are disjoint, $\T$-open subsets $U$ and $V$ of $X$ such that $C\subseteq U$ and $D\subseteq V$. 
\item
for every  $S\subseteq X$ $\T$-closed subset  and $U\subseteq X$ $\T$-open subset such that $S\subseteq U$, there are an $\T$-open subset $W$  of $X$ and a $\T$-closed subset $K$ of $X$ such that $S\subseteq W\subseteq K\subseteq U$.
\item
for every $\T$-open subsets $U$ and $V$ of $X$ such that $X=U\cup V$ there are $\T$-closed subsets $A\subseteq U$ and $B\subseteq V$ of $X$ such that $X=A\cup B$.
\end{enumerate}
\end{defn}

We continue with the $\fT$ versions of  some  definitions  from \cite{ep1} that are needed in the sequel.

\begin{defn}\label{defn def supports}
Let $(X,\T)$ be an object of $\fT$. A {\it family of $\T$-supports on $(X,\T)$} is a  family of closed $\T$-subsets of $X$ such that:

\begin{enumerate}
\item
every closed $\T$-subset contained in a member of $\Phi$ is in $\Phi $;

\item
$\Phi $ is closed under finite unions.

\vspace{.15in}
\noindent
$\Phi  $ is said to be a {\it family of  $\T$-normal supports} if in addition:

\item
each element of $\Phi $ is $\T$-normal;

\item
for each element $S$ of $\Phi  $, if $U\subseteq X$ is an open $\T$-subset such that $S\subseteq U,$  then there are an open $\T$-subset $W$ of $X$ and a closed $\T$-subset $K$ of $X$ such that $S\subseteq W\subseteq K\subseteq U$ and $K\in \Phi$.
\end{enumerate}
\end{defn}

\begin{expl}\label{expl supp c}
Let $(X,\T)$ be an object of $\fT$. Then the family of all closed $\T$-subsets of $X$ is a family of $\T$-supports on $(X,\T)$. Moreover, if $X$ is $\T$-normal, then this family is a family of  $\T$-normal supports on $(X,\T)$. For other examples that play a crucial role in the paper see Remarks \ref{nrmk c in Gi} and \ref{nrmk c in ACVF and hats}.
\end{expl}

Let $(X,\T)$ be an object of $\fT$. If $Z$ is a $\T$-subset of  $X$  and $\Phi  $ is a family of $\T$-supports on $(X,\T),$ then we have families of $\T\cap Z$-supports
\[\Phi  \cap Z=\{S\cap Z:S\in \Phi \}\]
and
\[\Phi _{|Z}=\{S\in \Phi : S\subseteq Z\}\]
on $(Z,\T \cap Z)$.

If $f\colon(X,\T)\into (Y,\T')$ is a morphism in $\fT$  and $\Psi $ is a family of $\T'$-supports on $(Y,\T'),$
then we have a family of $\T$-supports
\[f^{-1}\Psi  =\{S\subseteq X: S\,\,\textrm{is a closed}\,\, \T\textrm{-subset  and}\,\,\exists B\in \Psi \,\,(S\subseteq f^{-1}(B)\}\]
on $(X,\T)$.

\begin{nrmk}[Constructible family of supports]\label{nrmk def supp and cons supp}
Note that a family of $\T$-supports $\Phi $ on an object $(X,\T)$ of  $\fT$ determines a family of supports
\[\tilde{\Phi }=\{A\subseteq \tilde{X}: A\,\,\textrm{ is closed and}\,\,\exists B\in \Phi \,\,(A\subseteq \tilde{B})\}\]
on the topological space $\tilde{X}=\tilde{X_{\T}}$.  By Remark \ref{tilde on sets and maps}  it follows that
\[\tilde{\Phi  \cap Z}=\tilde{\Phi }\cap \tilde{Z},\,\, \tilde{\Phi  _{|Z}}=\tilde{\Phi }_{|\tilde{Z}}\,\,{\rm and}\,\,\tilde{f^{-1}\Psi }=\tilde{f}^{-1}\tilde{\Psi }.\]
We will say that the family of supports $\Psi$ on $\tilde{X}$ is a {\it constructible family of supports}  on $\tilde{X}$ if $\Psi=\tilde{\Phi }$ for some family of $\T$-supports on $(X,\T)$.
\end{nrmk}

\begin{nrmk}[Normal families of supports]\label{nrmk paracompact and normal supp}
Let us call a family $\Psi$ of supports on a topological space $Z$ a {\it normal family of supports} if: 
\begin{enumerate}
\item
each element of $\Psi $ is  normal;

\item
each element $S$ of $\Psi $ has a fundamental system of  normal (closed) neighborhood in $\Psi ,$ i.e., if $U\subseteq Z$ is an open  subset such that $S\subseteq U,$  then there are an open  subset $W$ of $Z$ and a closed subset $K$ of $Z$ such that $S\subseteq W\subseteq K\subseteq U$ and $K\in \Psi$.
\end{enumerate}

We point out that instead of (2) above, just like for the notion of paracompactifying family of supports in  topology (\cite[Chapter I, Section 6 and Chapter II, Sections 9 - 13]{b}) or the notion of regular and paracompact family of supports in sheaf theory on locally semi-algebraic spaces (\cite[Chapter II]{D3}), we could have taken the following  apparently weaker  condition:
\begin{itemize}
\item[(2)*]
each element $S$ of $\Psi  $ has a (closed) neighborhood  which is in $\Psi ,$ i.e. there are an open subset $W$ of $Z$ and a closed subset $K$ of $Z$ such that $S\subseteq W\subseteq K$ and $K\in \Psi $.
\end{itemize}
These conditions however are equivalent. Assume (2)*. Let $S\in \Psi$  and let $U\subseteq Z$ be an open subset such that $S\subseteq U$. Then there are an open subset $W'$ of $Z$ and a closed subset $K'$ of $Z$ such that $S\subseteq W'\subseteq K'$ and $K'\in \Psi $.  Then $S\subseteq W'\cap U\subseteq K',$ $S$ is closed in $K'$ and $K'$ is normal by (1). Therefore, there are $W\subseteq K'$ an open subset in $K'$ and $K\subseteq K'$ a closed subset in $K'$ such that $S\subseteq W\subseteq K\subseteq W'\cap U$.  It follows that $W$ is open in $Z,$  $K$ is closed in $Z$ and we get the stronger (2).



Observe also that as in the case of a paracompact family of supports (\cite[Chapter I, Section 6, Proposition 6.5]{b}) if $Y$ is a locally closed subset of $Z$ and $\Psi $ is a normal family of supports on $Z,$ then $\Psi _{|Y}$ is also a normal family of supports on $Y$. Indeed, $Y$ is closed in an open $U$ and $\Psi_{|Y}=(\Psi _{|U})_{|Y}=(\Psi _{|U})\cap Y$. 
\end{nrmk}

By Proposition \ref{prop main normal def normal} it follows that:

\begin{prop}\label{normal supp}
Let $(X,\T)$ be an object of $\fT$ and $\Phi$ a family of $\T$-supports on $(X,\T)$. Then $\Phi $ is $\T$-normal if and only if $\tilde{\Phi}$ is normal. 

Moreover, if $\Phi$ is $\T$-normal, then the following holds: 
\begin{itemize}
\item[($\ast$)]
each element $S$ of $\tilde{\Phi  }$ has a fundamental system of normal and constructible (closed) neighborhoods in $\tilde{\Phi},$ i.e.,  if $U\subseteq \tilde{X}$ is an open  subset such that $S\subseteq U,$  then there are an open constructible subset $W$ of $\tilde{X}$ and a closed constructible subset $K$ of $\tilde{X}$ such that $S\subseteq W\subseteq K\subseteq U$ and $K\in \Psi$.
\end{itemize}
\end{prop}

\begin{proof}
Suppose that $\Phi$ is $\T$-normal. Let $S\in \tilde{\Phi }$ and let $U\subseteq \tilde{X}$ be an open  subset such that $S\subseteq U$. Since $S$ is quasi-compact and $U$ is a union of constructible open subsets of $\tilde{X},$ there is an open constructible subset $U'$ of $\tilde{X}$ such that $S\subseteq U'\subseteq U$. Since $S\in \tilde{\Phi},$ there is $S'\in \tilde{\Phi}$ constructible such that $S\subseteq S'$. Hence $S=\bigcap \{S'\in \tilde{\Phi}: S'\,\,\textrm{is constructible and}\,\,S\subseteq S'\}$. Since $(\tilde{X}\setminus U')\cap S=\emptyset ,$ by quasi-compactness of $\tilde{X}\setminus U'$ there are $S'_1, \ldots, S'_l\in \tilde{\Phi},$ constructible with $S\subseteq S_i'$  such that $(\bigcap _{i=1}^lS'_i)\cap (\tilde{X}\setminus U')=\emptyset$. Let $S'=\bigcap _{i=1}^lS'_i$. Then $S\subseteq S'\subseteq U'\subseteq U$. Since $\Phi$ is $\T$-normal and $S'$ and $U'$ are constructible, there are an open constructible subset $W$ of $\tilde{X}$ and a closed constructible subset $K$ of $\tilde{X}$ such that $S\subseteq S'\subseteq W\subseteq K\subseteq U'\subseteq U$. This also shows $(\ast)$.

Conversely, suppose $\tilde{\Phi }$ is normal. If $S\in \tilde{\Phi }$ is constructible and  $U\subseteq \tilde{X}$ is an open constructible subset such that $S\subseteq U,$ then there are an open subset $W'\subseteq \tilde{X}$ and a closed subset $K'\subseteq \tilde{X}$ such that $S\subseteq W'\subseteq K'\subseteq U$. By quasi-compactness as above, there are   an open constructible subset $W\subseteq \tilde{X}$ and a closed constructible subset $K\subseteq \tilde{X}$ such that $S\subseteq W\subseteq W'\subseteq K'\subseteq K\subseteq U$.
\end{proof}

Below we say that $\Psi $ is  a {\it family of  normal and constructible supports} on the spectral topological space $\tilde{X}$  if $\Psi=\tilde{\Phi }$ for some family of $\T$-normal supports on  $(X,\T)$.

\begin{expl}\label{expl supp tilde c}
The main example of a family of normal and constructible supports on a spectral topological space $\tilde{X}$ is the family of all closed subsets of $\tilde{X}$ when $X$ is $\T$-normal. See also Remarks \ref{nrmk c in Gi} and \ref{nrmk c in ACVF and hats}.
\end{expl}

\subsection{Sheaf cohomology with \texorpdfstring{$\T$}{T}-supports}\label{subsection cohomo with t-supp}
Let $(X,\T)$ be an object of $\fT$. Due to the isomorphism 
\[\mod (A_{X_{\T}})\simeq \mod (A_{\tilde{X}}),\] 
in analogy to what happened in  the case $(\Df, \tDf)$  (\cite{ejp} for  o-minimal sheaf cohomology without supports and in  \cite{ep1} in  the presence of  families of definable supports), in this subsection we will  develop sheaf  cohomology on $X_{\T}$  via this tilde isomorphism. \\

\begin{defn}\label{defn  sheaf cohomo Phi}
Let $(X,\T)$ be an object of $\fT,$ $\Phi $ a family of $\T$-supports in $(X,\T)$  and $\cF \in \mod (A_{X_{\T}})$.  We define the {\it $\fT$-sheaf cohomology groups with $\T$-supports in $\Phi$ } via the tilde isomorphism by
\[H^*_{\Phi}(X;\cF)=H^*_{\tilde{\Phi}}(\tilde{X};\tilde{\cF}),\]
where $\tilde{\cF}$ is the image of $\cF$ via the isomorphism $\mod (A_{X_{\T}})\simeq \mod (A_{\tilde{X}})$.\\

If $f\colon(X,\T)\into (Y,\T')$ is a morphism in $\fT,$ we define the induced homomorphism
\[f^*\colon H^*_{\Phi }(Y;\cF)\into H^*_{f^{-1}\Phi }(X;f^{-1}\cF)\]
in cohomology to be the
same as the  homomorphism
\[\tilde{f}^*\colon H^*_{\tilde{\Phi }}(\tilde{Y};{\tilde \cF})\into
H^*_{\tilde{f}^{-1}\tilde{\Phi }}(\tilde{X};\tilde{f}^{-1}{\tilde \cF})\]
in cohomology induced by the continuous map
$\tilde{f}\colon \tilde{X}\into \tilde{Y}$ of topological spaces. \\
\end{defn}

Below we shall use a few  facts about sheaf cohomology in topological spaces that can be found, for example, in \cite[Chapter II, Sections 1 to 8]{b}.
We will also need the following $\fT$ versions of \cite[Chapter II, 9.5]{b} and  \cite[Chapter II, 9.21]{b} respectively.  

\begin{lem}\label{clm extending sections no supports}
Let $X$ be an object of $\tfT$. Let  $Z$ be a subspace of $X$  and $Y$  a quasi-compact subset of $Z$ having a fundamental system of normal and constructible locally closed neighborhoods in $X$.  Then for every $\cG\in {\mod}(A_Z)$ the canonical morphism
\[\lind {Y\subseteq U}\Gamma(U\cap Z;\cG)
\into \Gamma(Y;\cG_{|Y})\,\,\,\]
where $U$ ranges through the family of open constructible subsets of $X$, is an isomorphism.
\end{lem}

\begin{proof}
Since $Y$ is quasi-compact, the family of open neighborhoods of $Y$ in  $Z$ of the form $V\cap Z$, where $V$ is an open constructible subset of $X$, is a fundamental system  of neighborhoods of $Y$ in $Z$. Hence, the morphism of the lemma is certainly injective.

To prove that it is surjective, consider a section $s\in \Gamma (Y;\cG _{|Y})$.
There is a covering $\{U_j:j\in J\}$ of $Y$ by open constructible subsets of
$X$ and sections $s_j\in \Gamma (U_j\cap Z; \cG_{|U_j\cap Y})$, $j\in J$, such
that $s_{j|U_j\cap Y}=s_{|U_j\cap Y}$. Since $Y$ is quasi-compact, we can
assume that $J$ is finite, and so $\cup \{U_j:j\in J\}$ is an open
constructible subset of $X$. Therefore, there are an open constructible subset $O'$ in $X$ and  a normal, constructible locally closed subset $X'$  in $X$ such that $Y\subseteq O'\subseteq X'\subseteq \cup \{U_j:j\in J\}$. For each $j\in J$ let $U'_j=U_j\cap X'$. 
Since $X'$ is normal and constructible, by the shrinking lemma, there are open constructible subsets
$\{V'_j:j\in J\}$ of $X'$ and closed constructible subsets $\{K'_j:j\in J\}$ of $X'$ such that $V'_j\subseteq K'_j\subseteq U'_j$ for
every $j\in J$ and $X'= \cup \{V'_j:j\in J\}$. Since $X'$ is a constructible locally closed subset of $X,$ it is a  constructible closed subset of an open  subset of $X,$ which by quasi-compactness we may assume is also constructible. Replacing $X$ by that constructible open subset, we may assume that $X'$ is a constructible closed subset of $X$. For each $j\in J$ let $V_j=V'_j\cap O'$ and $K_j=K'_j$. Then for each $j\in J,$ $V_j$ is an open constructible subset of $X$ and $K_j$ is a closed constructible subset of $X$ with $V_j\subseteq K_j\subseteq U_j$ and $Y\subseteq \cup \{V_j:j\in J\}$. 

For $x \in Z\cap O'$ set $J(x)=\{j\in J:x\in K_j\}$ and let
\[
W_x = (\bigcap_{x
\in V_l}V_l \cap \bigcap_{j\in J(x)}U_j) \setminus \bigcup_{k \notin J
(x)}K_k.
\] 
Then  $W_x$ is a constructible neighborhood of $x$ in $X$ such that  if $y\in W_x$ then $J(y)\subseteq J(x)$. Indeed, suppose that $y\in W_x$ and let $i\in J(y)$. Then $y\in K_i$ and either $x\in K_i$ in which case $i\in J(x)$ or $x\not \in K_i$ in which case $i\not \in J(x)$ and so $y\in W_x\cap K_i=\emptyset$.

Observe that for all $i,j \in J(x)$ we have that $W_x$ is an open subset of both $U_i$ and  $U_j$. Hence, for every $i,j \in J(x)$ we have $s_{i|W_x\cap Y}= s_{|W_x\cap Y}=s_{j|W_x\cap Y}$. So, for  $y \in W_x \cap Y$, we have  $(s_i)_y=(s_j)_y$ for any $i,j \in J
(x)$.
This implies that the set
\[W=\{z\in (\bigcup_{j\in J}V_j)\cap Z: (s_i)_z=(s_j)_z\,\,\textrm{for any}\,\, i,j\in J(z)\}\]
contains $Y$ (clearly $Y\subseteq \bigcup _{x\in Z} W_x\cap Y\subseteq (\bigcup_{j\in J}V_j)\cap Z$). On the other hand,  if $z\in W$ then for any $i,j\in J(z)$ we have $(s_i)_z=(s_j)_z$ and so $s_i=s_j$  in an open neighborhood of $z$. Since $J(z)$ is finite, $z$ has an open  neighborhood in $Z$  on which $s_i=s_j$ for any $i,j \in J(z)$. Thus $W$ is an open neighborhood of $Y$ in $Z$.
Since $Y$ is quasi-compact we may assume that  $W$ is of the form $U\cap Z$ for
some open constructible subset $U$ of $X$. Since $s_{i_{|W \cap V_i \cap V_j}}
=s_{j_{|W \cap V_i \cap V_j}}$
there exists $t\in \Gamma (W;\cG )$ such that $t_
{|W \cap V_j}=s_{j|W \cap V_j}$. This proves that the morphism is surjective.
\end{proof}

A general form of Lemma \ref{clm extending sections no supports} is:

\begin{lem}\label{lem extending sections}
Assume that $X$ is an object of $\tfT$. Let $Z$ be a subspace of $X$, $\Phi $ a normal and constructible family of supports on $X$ and $Y$ a subset of $Z$ such that  for every constructible $D\in \Phi ,$ $D\cap Y$ is a quasi-compact  subset of $Z$ having a fundamental system of normal and constructible locally closed neighborhoods in $X$. Then for every $\cG \in \mod (A_Z)$, the canonical morphism
\[\lind {Y\subseteq U}\Gamma_{\Phi \cap U \cap Z}(U\cap Z;\cG)
\into \Gamma_{\Phi \cap Y}(Y;\cG_{|Y})\,\,\,\]
where $U$ ranges through the family of open constructible subsets of $X$, is an isomorphism.
\end{lem}

\begin{proof}
Let us prove injectivity. Let $s \in \Gamma_{D\cap U \cap Z}(U\cap Z;\cG)$, with $D \in \Phi$ and $U \supset Y$ open constructible subset of $X$ and such that $s_{|D \cap Y}=0$. By quasi-compactness of $D \cap Y$,  there exists an open constructible neighborhood $V$ of $D\cap Y$ in $X$ such that $s_{|V\cap D\cap Z}=0$. Replacing $V$ with its intersection with $U$ if necessary we may assume that  $V \subseteq U$.
 Set $W=V \cup(U \setminus D)$. Then  $W$ is open constructible in $X$, $Y\subseteq W\subseteq U$ and $s_{|W\cap Z}=0$.

Let us prove that the morphism is surjective. Let $s\in \Gamma _{\Phi \cap Y}(Y;\cG_{|Y})$ and consider a closed, normal and constructible set $C\in \Phi$ such that the support of $s$ is contained in $C\cap Y$. Since $\Phi$ is normal and constructible, there exist a closed, normal  and constructible set $D\in \Phi$ and an open constructible set $O\subseteq X$ such that $C\subseteq O\subseteq D$. 
We shall find $\widetilde{t} \in \Gamma_D(U \cap Z;\cG)$ such that
$\widetilde{t}_{|Y}=s$. 
After applying Lemma \ref{clm extending sections no
supports} above to $X,$ $D\cap Z$ and $D \cap Y$
we see that there exists an open, constructible neighborhood $V$ of $D\cap Y$ in $X$ and a section
$t\in \Gamma (V\cap D\cap Z;{\mathcal G})$ such that $t_{|D\cap Y}=s_{|D\cap Y}$.
Since $t_{|(D\setminus O)\cap Y}=0$, then each point $x$ of $(D\setminus O) \cap Y$ has
an open constructible neighborhood $W_x \subseteq V$ such that $t_{|W_x\cap D\cap Z}=0$. Using
quasi-compactness of $(D\setminus O)\cap Y$ (it is closed on the quasi-compact set $D\cap Y$), there exists a finite number of points
$x_1,\ldots,x_n$ such that $(D\setminus O) \cap Y \subseteq \bigcup_{i=1}^nW_{x_i}$.  Then $W\coloneqq \bigcup_{i=1}^nW_{x_i}$ is an open constructible subset and 
$t_{|W\cap D \cap Z}=0$. Let $U_1=(V\cap O)\cup W$. Then $U_1$ is open constructible and $D\cap
Y\subseteq U_1\subseteq V$. Define $t' \in \Gamma (U_1\cap Z; {\mathcal G})$ by: $t'_{|V\cap O\cap Z}=t_{|V\cap O\cap Z}$ and $t'_{|W\cap Z}=0$. This is well defined since $(V\cap O\cap Z)\cap (W\cap Z)\subseteq W\cap D \cap Z$ and $t_{|W\cap D \cap Z}=0$. Observe also that $t'_{|U_1\cap D\cap Z}=t$. Let $U_2=X\setminus D$. Then $U=U_1\cup U_2$ is
open constructible, $Y\subseteq U$, $U_1\cap U_2\subseteq W$ and we can define
$\widetilde{t} \in \Gamma(U\cap Z;\mathcal{G})$ in the following way:
$\widetilde{t}_{|U_1\cap Z}=t'_{|U_1\cap Z}$, $\widetilde{t}_{|U_2\cap Z}=0$. It
is well defined since  $U_1 \cap U_2 \subseteq W$ and $t'_{|W\cap Z}=0$. Moreover
$\mathrm{supp}\,\widetilde{t} \subseteq D$ and $\widetilde{t}_{|Y}=s$ as
required. 
\end{proof}

For the locally semi-algebraic analogues of Lemmas \ref{clm extending sections no supports} and \ref{lem extending sections} see \cite[Chapter II, Lemma 3.1 and Proposition 3.2]{D3}.  Note that in the locally semi-algebraic case one only requires that $Y$ has a paracompact and regular locally semi-algebraic neighborhood in $X,$ instead of a fundamental system of such neighborhoods, since under this assumption the family of all open locally semi-algebraic neighborhoods of $Y$ in $X$ is a fundamental system of neighborhoods of $Y$ in $X$ (\cite[Chapter I, Lemma 5.5]{D3}) and moreover these are paracompact and regular  i.e., they are normal and one can apply the shrinking lemma (\cite[Chapter I, Proposition 5.1]{D3}). The assumption for $\Phi$ implies that the same holds for each $D\cap Y$.

The  o-minimal versions are \cite[Lemmas 3.2 and 3.3]{ep1}. Unfortunately in the statement of these lemmas, and also in \cite[Proposition 3.7]{ep1}, the assumption that $Y$ (resp. $D\cap Y$) has a fundamental system of normal and constructible locally closed neighborhoods in $X$ is missing. Note however that this does not affect the main results of that/the paper since  in those results  $Y=X$ and each $D$ satisfies the hypothesis as $\Phi$ is normal and constructible. The only results affected are the Vietoris-Begle theorem and the homotopy axiom (\cite[Theorems 2.13 and 2.14]{ep1}) which we fix below.\\

Recall that a  sheaf $\cF$ on a topological space $X$ with a family of supports $\Phi $ is {\it $\Phi $-soft} if and only if  the restriction $\Gamma (X;\cF)\into \Gamma (S;\cF_{|S})$ is surjective for every   $S\in \Phi $.
If $\Phi $ consists of all closed subsets of $X$, then $\cF$ is simply called {\it soft}.\\

The topological analogue of the next result is  \cite[Chapter II, 9.3]{b}.

\begin{prop}\label{prop phi soft}
Let $X$ be a topological space and  $\cF$ be a sheaf in $\mod (A_X)$. Suppose that $\Phi$ is a family of supports on $X$ such that  every $C\in \Phi$ has a neighborhood $D$ in $X$ with $D\in \Phi$. Then the following are equivalent:
\begin{enumerate}
\item
 $\cF$ is $\Phi $-soft;
 \item
 $\cF_{|S}$ is soft for every  $S\in \Phi $;
  \item
 $\Gamma _{\Phi }(X;\cF)\into \Gamma _{\Phi _{|S}}(S; \cF_{|S})$ is surjective  for every closed subset $S$ of $X$;

 \medskip
 \noindent
If  in addition $X$ is an object  of $\tfT$ and $\Phi $ is a  constructible family of supports on $X$, then the above are also equivalent to:

\medskip
 \item
 $\cF_{|Z}$ is soft for every constructible subset $Z$ of $X$ which is in $\Phi $;

 \medskip
 \noindent
If  moreover  $\Phi $ is a  normal and constructible family of supports on $X$, then the above are also equivalent to:

\medskip

  \item
 $\Gamma (X;\cF)\into \Gamma (Z; \cF_{|Z})$ is surjective  for every constructible subset $Z$ of $X$ which is in $\Phi $;
\end{enumerate}
\end{prop}

\begin{proof} 
By our hypothesis on $\Phi$, the arguments in \cite[Chapter II, 9.3]{b} show the equivalence of  (1), (2) and (3). The equivalence of (2) and (4) is obvious since every $S\in \Phi $ is contained in some constructible subset  $Z$ of $X$ which is in $\Phi $. Also, (1) implies (5) is obvious. 

Assume (5) and let $S\in \Phi $ and $s\in \Gamma (S; \cF_{|S})$. Since $\Phi $ is normal and constructible, $S$ has a fundamental system of normal and constructible closed neighborhoods in $\Phi. $ In particular, there is an open constructible subset $V$ of $X$ and a closed constructible subset  $D$ of $X$ such that $S\subseteq V\subseteq D$ and  $D\in \Phi $. By Lemma \ref{clm extending sections no supports}, $s$ can be extended to a section $t\in \Gamma (W;\cF)$ of ${\mathcal F}$ over an open constructible subset $W$ of $X$ such that  $S\subseteq W\subseteq V$. Applying  the shrinking lemma in $D$ we find a closed constructible neighborhood $Z$ of $S$ in $W$. Since $D\in \Phi $ we have $Z\in \Phi $. So $t_{|Z}\in \Gamma (Z;\cF_{|Z})$, $(t_{|Z})_{|S}=s$ and $t_{|Z}$ can be extended to $X$ by (5).  
\end{proof}

The locally semi-algebraic analogue of Proposition \ref{prop phi soft} is \cite[Chapter II, Propositions 4.1 and 4.2]{D3}, and the o-minimal version is \cite[Proposition 3.4]{ep1}.\\

The following topological result is often useful. It is  an immediate consequence of  Proposition \ref{prop phi soft}  and we omit the proof and refer the reader to \cite[Proposition 3.6]{ep1} (compare also with \cite[Chapter II, Propositions 9.2 and 9.12 and Corollary 9.13]{b}).

\begin{fact}\label{fact soft loc and phi-dim}
Let $X$ be a topological space and  $\Phi $  is a family of supports on $X$ such that  every $C\in \Phi$ has a neighborhood $D$ in $X$ with $D\in \Phi$.
Let $W$ be a locally closed subset of $X$. The following hold:

\begin{itemize}
\item[(i)]
if $\cF \in \mod(A_X)$ is $\Phi$-soft, then $\cF_{|W}$ is $\Phi_{|W}$-soft.

\item[(ii)]
$\cG$ in $\mod(A_W)$ is $\Phi _{|W}$-soft if and only if  $i_{W!}\cG$ is $\Phi $-soft.

\item[(iii)]
if $\cF\in \mod(A_X)$ is $\Phi $-soft, then $\cF_W$ is $\Phi $-soft.
\end{itemize}
\end{fact}

We now came to the main result here. Its topological analogue is \cite[Chapter II, 9.6, 9.9 and 9.10]{b} and the locally semi-algebraic analogue is \cite[Chapter II, Propositions 4.8, 4.12 and Corollary 4.13]{D3}. The o-minimal version is \cite[Proposition 3.7]{ep1} (as mentioned above, the assumption that $D\cap Y$ has a fundamental system of normal and constructible locally closed neighborhoods in $X$ is missing).

\begin{prop}\label{prop injective for soft and flabby}
Assume that $X$ is an object  in $\tfT$. Let $Z$ be a subspace of $X,$ $\Phi $ be a normal and constructible family of supports on $X$ and $Y$ be a subspace of $Z$ such that for every constructible $D\in \Phi$, $D\cap Y$ is a quasi-compact subset of $Z$ having a fundamental system of normal and constructible locally closed neighborhoods in $X$. Then the
full additive subcategory of $\mod (A_Y)$ of $\Phi \cap Y$-soft sheaves
is $\Gamma _{\Phi \cap Y}(Y;\bullet )$-injective, i.e.:
\begin{enumerate}
\item
For every $\cF\in \mod (A_Y)$ there exists a  $\Phi \cap Y$-soft
$\cF'\in \mod (A_Y)$ and an exact sequence
$0\rightarrow \cF\rightarrow \cF'$.
\item
If $0\rightarrow \cF'\stackrel{\alpha}\rightarrow \cF\stackrel{\beta}\rightarrow \cF''
\rightarrow 0$ is an exact sequence in $\mod (A_Y)$ and $\cF'$ is
 $\Phi \cap Y$-soft, then
\[0\into \Gamma  _{\Phi \cap Y}(Y;\cF')\stackrel{\alpha}\into \Gamma  _{\Phi \cap Y}(Y;\cF)\stackrel{\beta}\into
\Gamma  _{\Phi \cap Y}(Y;\cF'')\into 0\] is an exact sequence.
\item
If $0\rightarrow \cF'\rightarrow \cF\rightarrow \cF''
\rightarrow 0$ is an exact sequence in $\mod(A_Y)$ and $\cF'$ and
$\cF$  are $\Phi \cap Y$-soft, then $\cF''$ is  $\Phi \cap Y$-soft.
\end{enumerate}
\end{prop}

\begin{proof}
Point (1) follows since on any topological space $Y,$ the full additive subcategory of $\mod(A_Y)$ of injective (and flabby) $A$-sheaves  is co-generating (see for example \cite[Proposition 2.4.3]{ks1}) and injective $A$-sheaves are flabbly and so $\Phi \cap Y$-soft by Lemma \ref{lem extending sections}. On the other hand,  (3) follows as usual from (2) by a simple diagram chase using Proposition  \ref{prop phi soft} (3). 

We now prove (2). Let $s'' \in \Gamma_{\Phi\cap Y}(Y,\cF'')$. We have to find  $s \in \Gamma_{\Phi\cap Y}(Y,\cF)$ such that $\beta (s)=s''$. Since $\Phi$ is normal and constructible, there exists an  open constructible subset $U$ of $X$ and a closed constructible subset $C$ of $X$ such that $\supp\,s''  \subseteq U\subseteq C$ and  $C \in \Phi$. 

Suppose there is $t\in \Gamma (C\cap Y; \cF)$ such that $\beta (t)=s''_{\,\,|C\cap Y}$. Then $\beta (t_{|(C\setminus U)\cap Y})=0,$ and so $t_{|(C\setminus U)\cap Y}=\alpha (s')$ for some $s'\in \Gamma ((C\setminus U)\cap Y;\cF')$. Since $\cF'$ is $\Phi \cap Y$-soft, $s'$ can be extended to a section $s'\in \Gamma (C\cap Y; \cF')$.
Then $(t-\alpha (s'))_{|(C\setminus U)\cap Y}=0$ and so can be extended, by $0,$ to a section $s\in \Gamma (Y;\cF)$ with $\supp \,s\subseteq C\cap Y\in \Phi\cap Y$ and $\beta (s)=s''$.
 
Now,  considering the exact sequence
\[
0\to\cF'_{C \cap Y}\stackrel{\alpha}\to\cF_{C \cap Y}\stackrel{\beta}\to\cF''_{C\cap Y}\to 0,
\]
since by Fact \ref{fact soft loc and phi-dim} (iii)  $\cF'_{C \cap Y}$ is still $\Phi \cap Y$-soft, replacing $\cF',\cF,\cF''$ with $\cF'_{C \cap Y},$ $\cF_{C \cap Y}$ and $\cF''_{C\cap Y}$ respectively, we are reduced to proving that the sequence
\[
0\to\Gamma(C\cap Y;\cF')\stackrel{\alpha}\to\Gamma(C\cap Y;\cF)\stackrel{\beta}\to\Gamma(C\cap Y;\cF'')\to 0
\]
is exact.\\

Let $s'' \in \Gamma(C\cap Y;\cF'')$. There is a covering $\{U_j:j\in J\}$ of $C\cap Y$ by open constructible subsets of
$X$ and sections $s_j\in \Gamma (U_j\cap (C\cap Y); \cF)$, $j\in J$, such
that $\beta (s_{j})=s''_{\,\,|U_j\cap (C\cap Y)}$. By the assumptions on $C\cap Y$ we can
assume that $J$ is finite, and so $\cup \{U_j:j\in J\}$ is an open
constructible subset of $X,$ and furthermore, there are an open constructible subset $O'$ in $X$ and  a normal, constructible locally closed subset $X'$  in $X$ such that $C\cap Y\subseteq O'\subseteq X'\subseteq \cup \{U_j:j\in J\}$. For each $j\in J$ let $U'_j=U_j\cap X'$. 
Since $X'$ is normal and constructible, by the shrinking lemma, there are open constructible subsets
$\{V'_j:j\in J\}$ of $X'$ and closed constructible subsets $\{K'_j:j\in J\}$ of $X'$ such that $V'_j\subseteq K'_j\subseteq U'_j$ for
every $j\in J$ and $X'= \cup \{V'_j:j\in J\}$. Now we may replace $X$ by a constructible open subset and assume that $X'$ is a constructible closed subset of $X$. For each $j\in J$ let $V_j=V'_j\cap O'$ and $K_j=K'_j$. Then for each $j\in J,$ $V_j$ is an open constructible subset of $X$ and $K_j$ is a closed constructible subset of $X$ with $V_j\subseteq K_j\subseteq U_j$ and $C\cap Y\subseteq \cup \{V_j:j\in J\}$. 

We now proceed by induction on $\#J$. Suppose that $J=I\cup \{j\}$ and let $V_I=\cup \{V_i:i\in I\},$ $K_I=\cup \{K_i:i\in I\}$ and $U_I=\cup \{U_i:i\in I\}$. By induction hypothesis let $s_I\in \Gamma (K_I\cap (C\cap Y); \cF)$ be such
that $\beta (s_{I})=s''_{\,\,|K_I\cap (C\cap Y)}$.  Since $\beta ((s_I-s_j)_{|K_I\cap K_j})=0,$ there is $s'\in \Gamma(K_I \cap K_j\cap (C\cap Y);\cF')$ such that $\alpha (s')=(s_I-s_j)_{|K_I\cap K_j}$ which extends to $s' \in \Gamma(C\cap Y;\cF')$ since $\cF'$ is $\Phi \cap Y$-soft. Replacing $s_I$ with
$s_I-\alpha (s'_{\,|K_I\cap (C\cap Y)})$ we may suppose that $s_I=s_j$ on $K_I \cap K_j\cap (C\cap Y)$. Then, since $(K_I\cup K_j)\cap (C\cap Y)= C\cap Y,$
there exists $s \in \Gamma(C\cap Y;\cF)$ such that $s_{| K_I\cap (C\cap Y)}=s_I$ and $s_{|K_j\cap (C\cap Y)}=s_j$. Thus the induction proceeds.
\end{proof}

We now include several corollaries that will be useful later.

\begin{cor}\label{cor hphi and !}
Assume that $X$ is an object  in $\tfT$. Suppose either that   $\Phi $ is a normal and  constructible family of supports on $X$ and  $Z$ is a (constructible) locally closed subset of $X$ or that $\Phi$ is any family of supports on $X$ and $Z$ is a closed subset of $X$. If $\cF\in \mod (A_Z)$, then
\[H_{\Phi }^*(X;i_{Z!}\cF)=H_{\Phi |Z}^*(Z; \cF).\]
\end{cor}

\begin{proof}
The second case is covered by \cite[Chapter II, 10.1]{b}.  If $Z$ is closed in an open subset $U$ of $X$, then $\Phi _{|U}$ is a normal and constructible family of supports on $U$ and $\Phi _{|Z}=\Phi _{|U}\cap Z$. In particular, if $D\in \Phi_{|U},$ then $D\cap Z\in \Phi _{|U}$ and so $D\cap Z$ is a quasi-compact  subset of $Z$ having a fundamental system of normal and constructible locally closed neighborhoods in $U$. Now the result follows from Proposition \ref{prop injective for soft and flabby}, Fact \ref{fact soft loc and phi-dim} (ii) and the fact that $\Gamma _{\Phi}(X;i_{Z!}\cF)\simeq \Gamma _{\Phi |Z}(Z;\cF)$ (\cite[Chapter I, Proposition 6.6]{b}). \end{proof}

It follows also that if $X,$ $Z,$ $Y$ and $\Phi$ are as in Proposition \ref{prop injective for soft and flabby}, then $H^q(Y; \cG )$ is the $q$-th cohomology  of the cochain complex  
 \[0\rightarrow  \Gamma (Y; \cI ^0)\rightarrow \Gamma (Y;\cI^1)\rightarrow \ldots \]
 where
 \[0\rightarrow \cG \rightarrow \cI ^0_{|Y}\rightarrow \cI^1_{|Y}\rightarrow \ldots \]
is a resolution  of $\cG $ by $\Phi\cap Y$-soft sheaves in $\mod (A_Y)$. 

Moreover, if $\cG$ is any flabby sheaf in $\mod (A_X)$ then the restriction $\Gamma _{\Phi }(X; \cG)\to \Gamma _{\Phi \cap Y}(Y;\cG _{|Y})$ is surjective (by Lemma \ref{lem extending sections}) and $H^p_{\Phi \cap Y}(Y;\cG_{|Y})=0$ for all $p>0$ (since $\cG _{|Y}$ is $\Phi \cap Y$-soft by Lemma \ref{clm extending sections no supports}). This means that $Y$ is $\Phi$-taut in $X$ (see \cite[Chapter II, Definition 10.5]{b}).\\

Since on any topological space $X$ an open subset is $\Psi $-taut in $X$ for any family of supports $\Psi$ on $X,$ it follows from \cite[Chapter II, Theorem 10.6]{b} that:

\begin{cor}\label{cor coho around}
Assume that $X$ is an object  in $\tfT$. Let $Z$ be a subspace of $X,$ $\Phi $  a normal and constructible family of supports on $X$ and $Y$ a subspace of $Z$ such that for every constructible $D\in \Phi ,$ $D\cap Y$ is a quasi-compact  subset of $Z$ having a fundamental system of normal and constructible locally closed neighborhoods in $X$. Then for every $\cG\in \mod (A_Z)$ the canonical homomorphism
\[\lind {Y\subseteq U}H^q_{\Phi \cap U\cap Z}(U\cap Z;\cG) \into H^q_{\Phi \cap Y}(Y;\cG_{|Y})\,\,\,\]
where $U$ ranges through the family of open constructible subsets of $X,$ is an isomorphism for every $q\geq 0$.
\end{cor}



For the locally semi-algebraic analogue of the above compare with \cite[Chapter II, Theorem 5.2]{D3}. 

Since  fibers of morphisms in $\tfT$ are quasi-compact (Remark \ref{tilde on sets and maps}) applying Corollary \ref{cor coho around} we obtain the following result (compare with \cite[Chapter II, Theorem 7.1]{D3}). 

\begin{thm}[Base change formula]\label{thm base change}
Let $f\colon X\to Y$ be a morphism in $\tfT$. Assume that $f$ maps  constructible closed subsets of $X$ to closed subsets of $Y$ and that  every $\alpha \in Y$ has an open constructible neighborhood $W$ in $Y$ such that $\alpha $ is closed in $W$ and $f^{-1}(W)$ is a normal and constructible subset of $X$. Let $\cF \in \mod (A_X)$. Then, for every $\alpha \in Y,$ the canonical homomorphism
\[(R^qf_*\cF)_{\alpha } \into H^q(f^{-1}(\alpha );\cF_{|f^{-1}(\alpha )})\,\,\,\]
is an isomorphism for every $q\geq 0$.
\end{thm}

\begin{proof}
Fix $\alpha \in Y$. First notice that 
\[\{f^{-1}(V): V\subseteq W\,\,\textrm{an open constructible neighborhood of}\,\,\alpha \}\]
is a fundamental system of open neighborhoods of $f^{-1}(\alpha )$. 
Indeed, let $U$ be an open neighborhood of $f^{-1}(\alpha)$. Since $f^{-1}(\alpha )$ is quasi-compact, we may assume that  $U$ is an open constructible neighborhood of $f^{-1}(\alpha)$. Then by assumption, $f(X\setminus U)$ is a closed subset of $Y$ not containing $\alpha$. Let $V$ be an open constructible neighborhood of $\alpha $ in $W$ contained in $Y\setminus f(X\setminus U)$. Then $f^{-1}(V)\subseteq U$. 

Notice also that $R^qf_*\cF$ is the sheaf associated to the pre-sheaf sending $V$ to $H^q(f^{-1}(V);\cF)$. So 
\[(R^qf_*\cF)_{\alpha }=\lind {\alpha \subseteq V\subseteq W}H^q(f^{-1}(V);\cF)=\lind {f^{-1}(\alpha )\subseteq U}H^q(U;\cF)\]
where $V$ (resp. $U$) ranges through the family of open constructible subsets of $Y$ (resp. $X$).

Now by assumption, $\alpha \in W$ is closed in $W$ and $f^{-1}(W)$ is an open  normal and constructible subset of $X$ containing $f^{-1}(\alpha)$. So the family of all closed subset of $f^{-1}(W)$ is a normal and constructible family of supports on $f^{-1}(W)$ (Example \ref{expl supp tilde c}) and,  since $f^{-1}(\alpha )$ is closed in $f^{-1}(W),$ for every constructible closed subset $D\subseteq  f^{-1}(W),$ $D\cap f^{-1}(\alpha )$ is a quasi-compact  subset of $f^{-1}(W)$ having a fundamental system of normal and constructible locally closed neighborhoods in $f^{-1}(W)$.

Therefore the result follows from Corollary \ref{cor coho around} applied to $X=Z=f^{-1}(W),$ $\Phi$ the family of all closed subsets of $f^{-1}(W)$ and $Y=f^{-1}(\alpha )$.
\end{proof}

Using classical arguments the Base change formula implies  the following form of the  Vietoris-Begle theorem:

\begin{thm}[Vietoris-Begle theorem]\label{thm vietoris-b}
Let $f\colon X\into Y$ be a surjective  morphism  in $\fT$. Assume that $f$ maps  constructible closed subsets of $X$ onto  closed subsets of $Y$ and that  every $\alpha \in Y$ has an open constructible neighborhood $W$ in $Y$ such that $\alpha $ is closed in $W$ and $f^{-1}(W)$ is a normal and constructible subset of $X$.    Let $\cF\in \mod (A_Y),$ and suppose that  $f^{-1}({\alpha  })$ is connected and $H^q(f^{-1}({\alpha  }); f^{-1}\cF_{|f^{-1}({\alpha   })})$ $=0$ for $q>0$ and all $\alpha \in Y$. Then for any   constructible family of supports $\Psi $ on $Y$ the induced map
\[f^*\colon H^*_{\Psi }(Y;\cF)\into H^*_{f^{-1}\Psi }(X;f^{-1}\cF)\]
is an isomorphism.
\end{thm}

\begin{proof}
The homomorphism $f^*\colon H^*_{\Psi}(Y;\cF)\into H^*_{f^{-1}\Psi}(X;f^{-1}\cF)$ is the composition $\epsilon \circ \eta $ where 
$$\epsilon \colon  H^*_{\Psi }(Y; f_*(f^{-1}\cF))\to H^*_{f^{-1}\Psi}(X,f^{-1}\cF)$$
is the canonical edge homomorphism $E_2^{*,0}\to E^*$ in the Leray spectral sequence 
\[H^p_{\Psi}(Y;R^qf_*(f^{-1}\cF))\Rightarrow H^{p+q}_{f^{-1}\Psi }(X; f^{-1}\cF)\]
of $f^{-1}\cF$ with respect to $f$ and $\eta \colon  H^*_{\Psi }(Y;\cF)\to H^*_{f^{-1}\Psi }(X;f_*(f^{-1}\cF))$ is induced by the canonical adjunction homomorphism 
\[\cF \to f_*f^{-1}\cF.\]

By Theorem \ref{thm base change} and the  hypothesis, $R^qf_*(f^{-1}\cF)=0$ for all $q>0$ and so the Leray sequence splits and $\epsilon $ is an isomorphism. On the other hand, by Theorem \ref{thm base change} and since $f^{-1}(\alpha )$ is connected, we have:
\begin{eqnarray*}
(f_*f^{-1}\cF)_{\alpha } &=&(R^0f_*f^{-1}\cF)_{\alpha }\\
&\simeq & H^0(f^{-1}(\alpha ); (f^{-1}\cF)_{|f^{-1}(\alpha )})\\
 &\simeq & H^0(f^{-1}(\alpha ); \cF_{\alpha })\\
&\simeq & \cF _{\alpha}.
\end{eqnarray*}
Hence adjunction homomorphism $\cF \to f_*f^{-1}\cF$ is an isomorphism and $\eta $ is also an isomorphism.
\end{proof}

We end this section with the following, which follows quickly from previous results is exactly the same way as its topological analogue (\cite[Chapter II, 16.1]{b}): 

\begin{prop}\label{prop soft open and phi-dim}
Assume that $X$ is an object  in $\tfT$,  $\Phi $ is a normal and  constructible family of supports on $X$ and ${\mathcal F}$ is a sheaf in $\mod (A_X)$. Then the  following are equivalent:
\begin{enumerate}
\item
${\mathcal F}$ is $\Phi $-soft;
\item
${\mathcal F}_U$ is $\Gamma _{\Phi }$-acyclic for all open and constructible $U\subseteq X$;
\item
$H^1_{\Phi }(X;{\mathcal F}_U)=0$ for all open and constructible $U\subseteq X$;
\item
$H^1_{\Phi |U}(U,{\mathcal F}_{|U})=0$ for all open and constructible $U\subseteq X$.
\end{enumerate}
\end{prop}

\begin{proof}
(1) $\Rightarrow $ (2) follows from Proposition \ref{prop injective for soft and flabby} and  Fact \ref{fact soft loc and phi-dim} (iii). (2) $\Rightarrow $ (3) is trivial. (3) is equivalent to (4) by Corollary \ref{cor hphi and !}. For (4) $\Rightarrow $  (1), consider a constructible closed set $C$ in $\Phi $ and the exact sequence $0\into {\mathcal F}_{X\setminus C}\into {\mathcal F}\into {\mathcal F}_C\into 0$.
The associated long exact cohomology sequence
$$\dots \rightarrow \Gamma _{\Phi }(X;{\mathcal F})\rightarrow \Gamma _{\Phi |C}(X;{\mathcal F}_{|C})\rightarrow H_{\Phi |X\setminus C}^1(X\setminus C;{\mathcal F}_{|X\setminus C})\rightarrow \dots $$
shows that $\Gamma _{\Phi }(X;{\mathcal F})\into  \Gamma _{\Phi |C}(X;{\mathcal F}_{|C})$ is surjective. Hence ${\mathcal F}$ is $\Phi $-soft by Proposition \ref{prop phi soft} (5).
\end{proof}

\section{A site on definable sets in~\texorpdfstring{$\bGi$}{G}}\label{section t-top in bGi}

We will assume some familiarity with o-minimality. We refer the reader to classical texts like \cite{vdd}. Note that we allow o-minimal structures to have end points, which is a subtle difference from \cite{vdd}. Most results from \cite{vdd} still hold in this context.  

Let $\bG=(\nG , < , \ldots)$ be an  o-minimal structure without end points.  Let $\infty $ be a new symbol and set $\nGi=\nG \cup \{\infty \}$ with $x<\infty $ for all $x\in \nG$.  The example that we have in mind is the case where $\bG=(\nG, <, +)$ is the value group of algebraically closed field and $\infty $ is the valuation of $0$. 

When $(\nG, <)$ is non-archimedean as a linear order (i.e., does not embed into the reals), then infinite definable sets in the structure $\bGi$ induced by $\bG$ on $\nGi$ with their natural topology are in general totally disconnected and not locally compact. The goals in  this section are: (i) to introduce an appropriate site which will replace the topology just mentioned and show that definable sets with this site are $\T$-spaces (Subsection \ref{subsection o-min site in bGi}); (ii) to show that for the associated notion of $\T$-normality,  when $\bG$ is an o-minimal expansion of an ordered group, $\T$-locally closed subsets are finite unions of $\T$-open subsets which are $\T$-normal (Subsection \ref{subsection  def-comp and def-normal in bGi}). We point out  that normality in $\bGi$ is more complicated than in $\bG$ (see Example \ref{expl def normal} below), which justifies the extra work that will be done below.

\subsection{The o-minimal site on~\texorpdfstring{$\bGi$}{G}-definable sets}\label{subsection o-min site in bGi}
Here we introduce the natural first-order logic structure $\bGi$ induced by $\bG$ on $\nGi$ and show that $\bGi$-definable sets are equipped with a site making them $\T$-spaces.\\ 
 
For $x=(x_1, \ldots, x_m)\in \nGim,$ the {\it support of $x$}, denoted $\supp\, x,$ is defined by 
\[
\supp \, x=\{i\in \{1, \ldots, m\}:  x_i\neq \infty \}.
\]
For $L\subseteq \{1,\ldots , m\}$ let 
\[
(\nGim )_{L}=\{x\in \nGim: \supp \, x= L\}.
\]
Then $\nGim=\bigsqcup _{L\subseteq \{1,\ldots , m\}}(\nGim )_{L}$,  $(\nGim )_{\{1,\ldots, m\}}=\nGm$ and,  if $\tau  _{L}\colon \nGim\to \nGiL$ is the projection onto the $|L|$ coordinates  in $L,$ then the restriction  $\tau  _{L|}\colon (\nGim)_{L}\to \nGL$ is a bijection and $\tau  _{\{1,\ldots,m\} |}\colon \nGm \to \nGm$ is the identity.
 
If $\pi \colon \nGim \to \nGik$  is the projection onto the first $k$ coordinates, $\pi '\colon \nGim \to \Gamma ^{m-k}_{\infty }$  is the projection onto the last $m-k$ coordinates,  and for $L\subseteq \{1,\ldots, m\},$ $\pi (L)=\{1,\ldots, k\}\cap  L$ and $\pi '(L)=-k+\{k+1,\ldots, m\}\cap L,$ then:
\begin{itemize}
 \item[(*)]
\[\supp \,x=\supp \, \pi (x)\sqcup (k+\supp \, \pi'(x))\] 
and moreover,
\item[]
\[
\xymatrix{
(\nGim )_L\ar@{^{(}->}[r] \ar[d]^{\pi_|} &\nGim \ar[r]^{\tau _L} \ar[d]^{\pi } & \nGiL \ar[d]^{\pi ^L_{k}}  \\
(\nGik )_{\pi (L)}\ar@{^{(}->}[r]  &\nGik \ar[r]^{\tau _{\pi (L)}\,\,} & \nGipiL
}
\]
where $\pi ^L_k$ is projection onto the first $\# \pi (L)$ coordinates and 
\[
\xymatrix{
(\nGim )_L\ar@{^{(}->}[r] \ar[d]^{\id } &\,\,\,\,\nGim \,\,\,\, \ar[r]^{\tau _L} \ar[d]^{\id } & \,\,\,\,\nGiL \ar[d]^{\id}  \\
(\nGik )_{\pi (L)}\times (\nGimk )_{\pi '(L)}\ar@{^{(}->}[r]  &\,\,\,\, \nGik \times \nGimk \,\,\,\, \ar[r]^{\tau _{\pi (L)}\times \tau _{\pi '(L)}} & \,\,\,\,\,\,\,\,\,\,\,\,\nGipiL\times \nGiqiL
}
\]
are commutative diagrams.\\
\end{itemize}

If for $X\subseteq \nGim$ and for $L\subseteq \{1,\ldots , m\}$ we set 
\[X_{L}=X\cap (\nGim)_L\]
then $X=\bigsqcup _{L\subseteq \{1,\ldots , m\}}X_{L}$ and  $X_{\{1,\ldots, m\}}=X\cap \Gamma ^m$. Furthermore, the restriction 
$\tau  _{L|}\colon X_{L}\to \tau _{L}(X_{L})$
 is a bijection, $\tau _L(X_L)\subseteq \nG ^{|L|}$ and $\tau  _{\{1,\ldots,m\} |}\colon X\cap \nGm\to X\cap \nGm$ is the identity.\\

For each $m$ let 
\[
\mathfrak{G}_m=\{X\subseteq \nGim:   \tau _{L}(X_{L})\subseteq \nG ^{|L|}\,\,\textrm{is $\bG$-definable}\,\,\textrm{for every}\,L\subseteq \{1, \ldots ,m\}\}.
\] 

Recall that an o-minimal structure (possibly with end points) has definable Skolem functions if given a definable family $\{Y_t\}_{t\in T}$, there is a definable function $f\colon T\to \bigcup_{t\in T} Y_t$ such that $f(t)\in Y_t$ for each $t\in T$. By the (observations before the) proof of \cite[Chapter 6, (1.2)]{vdd} (see also Comment (1.3) there), the o-minimal structure has definable Skolem functions if and only if every nonempty definable set $X$  has a definable  element  $e(X)\in X$.  

The following Proposition is left to the reader. 

\begin{prop}\label{prop bGi-def}
$\bGi =(\nGi,  (\mathfrak{G}_m)_{m \in {\mathbb N}})$ is an o-minimal structure  with right end point $\infty$. Moreover
if $\bG$ has definable Skolem functions then $\bGi$ has definable Skolem functions. \qed
\end{prop}

Results in o-minimality usually are stated and proved for o-minimal structures without end points. Nearly all of those results can be checked to hold with exactly the same proof when there is an endpoint. Nevertheless,  for convenience,  we will now introduce a new structure without end points which contains as a substructure a copy of $\bGi$.\\

Let $\nSg=\{0\}\times \nGi \cup \{1\}\times \nG$ be equipped with the natural order extending $<$ on $\nG$ and on $\nGi$ such  that  $(0,x) < (1,y)$ for all $x\in \nGi$ and all $y\in \nG$. 

Set $\nSg _0=\{0\}\times \nGi,$ $\nSg _1=\{1\}\times \nG$ and for $\alpha  \in 2^m,$  which below we identify with a sequence of 0's and 1's of length $m,$ let
\[
\nSg _{\alpha }=\Pi _{i=1}^m\nSg _{\alpha (i)}.
\]
Then $\nSgm=\bigsqcup _{\alpha  \in 2^m}\nSg _{\alpha },$ $\nSg _{\bar{0}}=\nSg _0^m$ and $\nSg _{\bar{1}}=\nSg _1^m$. If 
\[\sigma _m\colon \nSgm \to \nGim \colon  ((\alpha (1),x_1), \ldots, (\alpha  (m),x_m))\mapsto (x_1,\ldots ,x_m)\]
 is the natural projection, then the restriction  $\sigma _{m|}\colon \nSg _{\alpha }\to \nGim$
 is injective, $\sigma _{m|}\colon \nSg _{\bar{0}}\to \nGim$ is a bijection with inverse the natural inclusion $\iota _m\colon \nGim \to \nSg _0^m\colon  (x_1,\ldots , x_m)\mapsto ((0,x_1), \ldots , (0,x_m))$ and $\sigma _{m|}\colon \nSg _{\bar{1}}\to \nGm$ is a bijection with inverse the natural inclusion $\iota _m\colon \nGm \to \nSg _1^m\colon  (x_1,\ldots , x_m)\mapsto ((1,x_1), \ldots , (1,x_m))$\\
 
If  $\pi \colon \nSgm \to \nSgk$  is the projection onto the first $k$ pairs of coordinates,  $\pi '\colon \nSgm \to \nSgmk$  is the projection onto the last $m-k$ pairs of coordinates, for $\alpha \in 2^m,$   $\pi (\alpha )\in 2^k$ is given by $\pi (\alpha  )=\alpha  _{|\{1,\ldots, k\}}$ and $\pi '(\alpha)\in 2^{m-k}$ is  given by $\pi '(\alpha  )(s)=\alpha (s+k),$ then: 
\begin{itemize}
 \item[(**)]
\[\alpha =\pi (\alpha )* \pi'(\alpha ),\] 
where $*$ is concatenation of sequences, and moreover,  
\[
\xymatrix{
\nSg _{\alpha }\ar@{^{(}->}[r] \ar[d]^{\pi_|} &\nSgm \ar[r]^{\sigma _m} \ar[d]^{\pi } & \nGim \ar[d]^{\pi }  \\
\nSg _{\pi (\alpha )}\ar@{^{(}->}[r]  &\nSgk \ar[r]^{\sigma _{k}\,\,} & \nGik
}
\]
and 
\[
\xymatrix{
\nSg _{\alpha }\ar@{^{(}->}[r] \ar[d]^{\id } &\,\,\,\,\nSgm \,\,\,\, \ar[r]^{\sigma _m} \ar[d]^{\id } & \,\,\,\,\nGim \ar[d]^{\id}  \\
\nSg _{\pi (\alpha )}\times \nSg _{\pi '(\alpha )}\ar@{^{(}->}[r]  &\,\,\,\, \nSgk \times \nSgmk \,\,\,\, \ar[r]^{\sigma  _{k}\times \sigma _{m-k}} & \,\,\,\,\,\,\,\,\,\,\,\,\nGik\times \nGimk
}
\]
are commutative diagrams.\\
\end{itemize}

If for $X\subseteq \nSgm$ and for $\alpha \in 2^m$ we set 
\[X_{\alpha }=X\cap \nSg _{\alpha }\]
then $X=\bigsqcup _{\alpha \in  2^m}X_{\alpha }$. Furthermore, the restriction 
$\sigma   _{m|}\colon X_{\alpha }\to \sigma _{m}(X_{\alpha })$
 is a bijection, $\sigma  _m(X_{\alpha })\subseteq \nGim$.\\

For each $m$ let 
\[
\mathfrak{S}_m=\{X\subseteq \nSgm:   \sigma _{m}(X_{\alpha })\subseteq \nGim\,\,\textrm{is $\bGi$-definable}\,\,\textrm{for every}\,\sigma \in 2^m\}.
\]

Similarly to Proposition \ref{prop bGi-def} we have:

\begin{prop}\label{prop bSg-def}
$\bSg=(\nSg,  (\mathfrak{S}_m)_{m\in {\mathbb N}})$ is an o-minimal structure. Moreover, if $\bG$ has definable Skolem functions then $\bSg$ has definable Skolem functions. \qed
\end{prop}

The following remarks are obvious from the constructions above:

\begin{nrmk}[Natural embeddings]\label{nrmk embeddings}
We have:
\begin{itemize}
\item[(a)]
The structure $\bG$ and the substructure $(\nG,  (\mathfrak{G}_m \cap \nGm )_{m\in {\mathbb N}})$ of $\bGi$ have the same definable sets.
\item[(b)]
Under the inclusion $\iota _1\colon \nGi\to \nSg _0,$ the structure $\bGi$ and the substructure $(\nSg _0,  (\mathfrak{S}_m \cap \nSg _0^m )_{m\in {\mathbb N}})$ of $\bSg$ have the same definable sets.
\end{itemize}
So when convenient, below  we will often identify $\bGi$ with its copy in $\bSg$.
\end{nrmk}

We now make a couple of observations comparing $\bG,$ $\bGi $ and $\bSg$.

\begin{nrmk}[The topologies]\label{nrmk top}
Let $\bG=(\nG, <,\ldots )$ be an arbitrary o-minimal structure without end points. Then the topology of $\nG$ is generated by the open intervals $(-\infty , a), (a,b)$ and $(b, \infty )$ with $a,b\in \nG$  and the topology of $\nG ^n$ is generated by the products of $n$ such open intervals. 

In $\bGi$ the topology of $\nGi$ is generated by the open intervals $(-\infty , a), (a,b)$ and $(b, \infty ]$ with $a,b\in \nG$ and the topology of $\nGin$ is generated by the products of $n$ such open intervals.

It follows that:  
\begin{itemize}
\item
$\nGm$ is open in $\nGim$ and the topology on $\nGm$ is the induced topology from $\nGim;$ 
\item
$\nGim$ is closed in $\nSgm$ and the topology on $\nGim$ is the induced topology from $\nSgm$.
\end{itemize}
\end{nrmk}

Recall the notion of {\it cell in $\nGin$} defined inductively as follows:
\begin{itemize}
 \item
 a  cell in $\nGi$ is either a singleton $\{a\}$ or $\{\infty \},$  or an open interval of the form $(-\infty , a)$ or $(a,b)$  for $a,b\in \nG$ with $a<b,$
 \item
 a cell in $\Gamma ^{n+1}_{\infty }$ is a set of the form 
 \[\Gamma(h)=\{(x,y)\in \Gamma ^{n+1}_{\infty }: x\in X\,\,\textrm{and}\,\,y=h(x)\},\] 
 or 
 \[(f,g)_X=\{(x,y)\in \Gamma ^{n+1}_{\infty }: x\in X\,\,\textrm{and}\,\,f(x)<y<g(x)\},\] 
 or
 \[(-\infty , g)_X=\{(x,y)\in \Gamma ^{n+1}_{\infty }: x\in X\,\,\textrm{and}\,\,y<g(x)\},\] 
 for some $\bGi$-definable and continuous maps $f, g,h\colon X\to \nGi$ with $f<g$, where $X$ is a  cell in $\nGin$.
\end{itemize}
In either case, $X$ is called \emph{the domain} of the defined cell. 

\begin{nrmk}[Cell decompositions]\label{nrmk cd}
As in \cite[Chapter 3, (2.11)]{vdd} (the proof also works in the case with end points) we have:
\begin{itemize}
\item
Given any $\bGi$-definable sets $A_1, \dots, A_k\subseteq  \nGin$, there is a decomposition $\cC$ of $\nGin$ that partitions each $A_i$.
\item
Given a $\bGi$-definable map $f\colon A\to  \nGi$, there is a decomposition $\cC$ of $\nGin$ that partitions $A$ such that the restriction $f_{|B}$ to each $B\in \cC$ with $B\subseteq  A$ is continuous.
\end{itemize}
Here (see \cite[Chapter 3, (2.10)]{vdd}) a decomposition of $\nGi$ is a partition of $\nGi$ into cells and a decomposition of $\Gamma ^{k+1}_{\infty }$ is a partition of $\Gamma ^{k+1}_{\infty }$ into cells such that its projection  to $\nGik$ is a decomposition of $\nGik$. \\

By the inductive construction of cells it is easy to see that:
\begin{itemize}
\item[(a)] If $\cC$ is a cell decomposition of $\nGim$ then: (i) $\cC$ partitions the sets in $\{(\nGim )_L : L\subseteq \{1,\ldots, m\}\},$ in particular, $\cC _{|\nGm}$ is a cell decomposition of $\nGm;$ (ii) the cells in $\cC$ are built using $\bGi$-definable continuous maps $f\colon X\to \nGi$ with $X\subseteq (\nGim)_L$ a cell and such that $f$ is constant with value $\infty$, or 
\[
f=f_L\circ\tau_L
\]
where $f_L\colon \tau _L(X)\to \nG$ is a $\bG$-definable continuous map. 
\item[(b)]
If $\cB$ is a cell decomposition of $\nSgm$ that partitions the sets in $\{(\nGim )_L : L\subseteq \{1,\ldots, m\}\}$ then $\cB _{|\nGim}$ is a cell decomposition of $\nGim$.\\
\end{itemize}
\end{nrmk}


In an arbitrary o-minimal structure a definable set $X$  is {\it definably connected} if and only if the only clopen definable subsets of $X$ are $\emptyset$ and $X$. We say that a $\bGi$-definable subset of $\nGim$ is  {\it $\bGi$-definably locally closed} if and only if it is the intersection of an open $\bGi$-definable subset of $\nGim$ and a closed $\bGi$-definable subset of $\nGim$ or equivalently if and only if it  is open in its closure  in $\nGim$. A similar definition applies in $\bG$.

\begin{nrmk}[Definable connectedness]\label{nrmk def conn}
From the above remark about the topologies,  a $\bG$-definable subset of $\nGm$ is $\bG$-definably connected if and only if it is $\bGi$-definably connected and a $\bGi$-definable subset of $\nGim$ is $\bGi$-definably connected if and only if it is $\bSg$-definably connected. Hence, just like in $\bG,$ by \cite[Chapter 3, Proposition 2.18]{vdd}:
\begin{itemize}
\item
 every $\bGi$-definable subset of $\nGim$ has finitely many $\bGi$-definably connected components which are clopen and partition the $\bGi$-definable set.
 \end{itemize}
\end{nrmk}

\begin{nrmk}[Definably locally closed subsets]\label{nrmk def loc closed}
Since cells in $\nGm$ are $\bG$-definably locally closed (\cite[page 51]{vdd}), by cell decomposition theorem, every $\bG$-definable subset of $\nGm$ is a finite union of $\bG$-definably locally closed sets. Working in $\bSg$ it follows, by Remark \ref{nrmk top}, that:
\begin{itemize}
\item 
every $\bGi$-definable subset of $\nGim$ is a finite union of $\bGi$-definably locally closed sets.
\end{itemize}
\end{nrmk}

Just like in $\bG,$ if  $(\nG,<)$ is non-archimedean as a linear order, then infinite $\bGi$-definable spaces with the topology generated by open definable subsets are not in general totally disconnected or locally compact. So, as in Example \ref{expls old tspaces} (3), to develop cohomology one studies $\bGi$-definable spaces $X$ equipped with the o-minimal site:
 
\begin{defn}[O-minimal site on $\bGi$-definable subsets]\label{def T-space in Gi}
Let $\bG= (\nG,<, \ldots )$ be an arbitrary o-minimal structure without end points.  If $X$ is a $\bGi$-definable subset of $\nGim,$ then the o-minimal site $X_{\df}$ on $X$  is the category $\op (X_{\df})$ whose objects are open (in the topology of $X$ mentioned above) $\bGi$-definable subsets of $X$, the morphisms are the  inclusions and the admissible covers $\cov (U)$ of $U\in \op (X_{\df})$ are covers by open $\bGi$-definable subsets of $X$ with  finite sub-covers.
\end{defn}

\begin{prop}\label{prop def t-space in Gi}
Let $\bG= (\nG,<, \ldots )$ be an arbitrary o-minimal structure without end points.  If $X$ is a $\bGi$-definable subset of $\nGim,$ let 
\[
\T=\{U \in \op(X): U\,\, \text{is $\bGi$-definable}\}.
\]
Then $X$ is a $\T$-space, $X_{\T}=X_{\df}$ and $\widetilde{X}_{\T}$ is the o-minimal spectrum $\widetilde{X}$ of $X$ i.e., its points are types over $\nGi$ concentrated on $X$. Furthermore,  there is  an equivalence of categories 
\[
\mod(A_{X_{\df}}) \simeq \mod(A_{\widetilde{X}}).
\]
\end{prop}

\begin{proof}
By Remark \ref{nrmk top} $\T$ is a basis for a topology of $X,$ and $\emptyset \in \T;$ it is clear that $\T$ is closed under finite unions and intersections; by Remark \ref{nrmk def loc closed} the $\T$-subsets of $X$ are exactly the $\bGi$-definable subsets of $X$. Therefore from Remark \ref{nrmk def conn} it follows that every $U\in \T$ has finitely many $\T$-connected components. The rest follows from  the definitions and results in Subsection \ref{subsection t-sheaves}.
\end{proof}

The following remark will allow us to work in $\bGi$ instead of in $\bG$ or in $\bSg$ instead of in $\bGi$ when convenient:

\begin{nrmk}\label{nrmk top and cohom}
Let $\bG =(\nG, <,\ldots )$ be an arbitrary o-minimal structure without end points. By Remark \ref{nrmk top} we have:
\begin{itemize}
\item
If  $X\subseteq \nGm$ is a $\bG$-definable subset, then the o-minimal site of $X$ in $\bG$ is the same as the o-minimal site of $X$ in $\bGi$ and so  the o-minimal sheaf and cohomology theories of $X$  in $\bG$ are the same as the o-minimal sheaf and cohomology theories of $X$ in $\bGi$. 
\item
If  $X\subseteq \nGim$ is a $\bGi$-definable subset, then the o-minimal site of $X$ in $\bGi$ is the same as the o-minimal site of $X$ in $\bSg$ and so  the o-minimal sheaf and cohomology theories of $X$  in $\bGi$ are the same as the o-minimal sheaf and cohomology theories of $X$ in $\bSg$. 
\end{itemize}
\end{nrmk}

\subsection{Definable compactness and definable normality}\label{subsection  def-comp and def-normal in bGi}
In this subsection we recall  the notion of definable compactness  for  definable sets in $\bGi$ (defined in analogy to the case $\bG$) and make a couple of remarks about locally definably compact and definable completions that will be used later.  We introduce the notion of definable normality (which corresponds to $\T$-normality for the associated $\T$-topology) and prove that, when $\bG$ is an o-minimal expansion of an ordered group, then every definably locally closed  subsets of $\nGin$ is the union of finitely many open definable subsets which are definably normal (Theorem \ref{thm basis of open normal}).

As in \cite{ps} we say that a $\bGi$-definable subset $X\subseteq \nGim$ is {\it $\bGi$-definably compact} if and only if for every $\bGi$-definable and continuous map from an open interval in $\nGi$ into $X$ the limits at the endpoints of the interval exist in $X$. Recall  that $X$ is  {\it locally $\bGi$-definably compact} if and only if  every point in $X$ has a $\bGi$-definably compact neighborhood.

\begin{nrmk}[Definable compactness]\label{nrmk def comp}
Note that in the above definition it is enough to consider open intervals of the form $(-\infty ,a), (a,b)$ or $(b, \infty )$ where $a,b \in \nG$ and $a<b$. Therefore:
\begin{itemize}
\item
a $\bG$-definable subset $X\subseteq \nGm$ is $\bG$-definably compact if and only if it is $\bGi$-definably compact.
\end{itemize} 
On the other hand, since for any interval $I$ in $\nSg=\nSg _0\sqcup \nSg _1$ we have that $I\cap \nSg _0$ and $I\cap \nSg _1,$ if non empty, are intervals in $\nGi,$ it follows that:
\begin{itemize}
\item
a $\bGi$-definable subset $X\subseteq \nGim$ is $\bGi$-definably compact if and only if it is $\nSg$-definably compact.
\end{itemize} 
In particular, by \cite[Theorem 2.1]{ps} we have:
\begin{itemize}
\item
a $\bGi$-definable subset $X\subseteq \nGim$ is $\bGi$-definably compact if and only if it is closed and {\it bounded in $\nGim$}, i.e.,  $X\subseteq \Pi _{i=1}^m[c_i, \infty ]$ for some $c_i\in \Gamma $ ($i=1,\ldots, m$). 
\end{itemize}

\end{nrmk}
Recall that a  type $\beta $ on $X$ (i.e. an ultrafilter of definable subsets of $X$) is a {\it definable type on $X$} if  and only if for every uniformly definable family $\{Y_t\}_{t\in T}$ of definable subsets of $X,$ the set $\{t\in T: Y_t\in \beta \}$ is a definable set. And we say a definable type $\beta$ on $X$ has limit $a \in X$ if for any definable neighborhood $a\in U$, we have $\beta$ concentrates on $U$. It is worth pointing out that the definition in \cite{HrLo} of definable compactness of $X$ is that any definable type on $X$ has a limit in $X$. At first glance, this seems to be a stronger criterion than defined above. However, both of them turns out to be the same as being closed and bounded. The main advantage of the definable type definition is that one cannot use definable path on $X$ to describe the topological closure of $X$, yet the set of limits of definable types on $X$ is the topological closure of $X$. This failure is also noted in \cite[Theorem 2.3]{ps}. Nonetheless, this does not affect the characterization of definable compactness in $\bGi$.

\begin{nrmk}[Locally definably compact]\label{nrmk local def comp}
As in topology, locally definably compact is equivalent to definably locally closed.  
Suppose that $X$ is locally definably compact. Let $a\in X$ and $K$ be a definably compact definable neighborhood of $a$ in $X$. Let $U$ be an open definable neighborhood of $a$ such that $U\cap X\subseteq K$. Taking closures in $\bar{X}$ and noting that $\overline{U\cap \bar{X}}=\overline{U\cap X}$ we obtain
\[
U\cap \bar{X}\subseteq \overline{U\cap X} \subseteq \bar{K} = K \subseteq X, 
\]
which shows that $X$ is open in $\bar{X}$. The other implication is also easy.

\end{nrmk}


\begin{nrmk}[Definable completions]\label{nrmk def completions}
Let $X\subseteq \nGim$ be a $\bGi$-definable set which  is bounded in $\bGi$  and is locally $\bGi$-definably compact. Then the inclusion $i\colon X\hookrightarrow \bar{X}$ into the closure of $X$ in $\nGim$ is a $\bGi$-definable completion of $X,$ i.e.: 
\begin{itemize}
\item[(i)] 
$\bar{X}$ is  $\bGi$-definably compact; 
\item[(ii)]  
 $i\colon X\hookrightarrow \bar{X}$ is a $\bGi$-definable open immersion (i.e. $i(X)$ is open in $\bar{X}$ and $i\colon X\to i(X)$ is a $\bGi$-definable homeomorphism);
 \item[(iii)]  
 the image $i(X)$ dense in $\bar{X}$. 
 \end{itemize}
Everything here is  clear, and $i(X)$ is open in $\bar{X}$ by Remark \ref{nrmk local def comp}. 

It follows that:
\begin{itemize}
\item
If $\bG=(\nG, <, +,  \ldots )$ is an o-minimal expansion of an ordered group $(\nG , <, + ),$ then every $\bGi$-definably locally closed subset of $\nGim$ is $\bGi$-definably completable.
\end{itemize}
Indeed, let $p\colon \nGi \to \nGi \times \nGi$ be the $\bGi$-definable map given by 
\begin{equation*}
p(x)=
\begin{cases}
(-x,0) \qquad  \textrm{if} \,\,x<0\\
\\
(0,x) \qquad \,\,\,\, \textrm{otherwise.}
\end{cases}
\end{equation*}
Then $p(\nGi )\subseteq [0,\infty ]\times [0,\infty ]$ and $p\colon \nGi\to p(\nGi)$ is a $\bGi$-definable homeomorphism.  So every $\bGi$-definable subset of $\nGim$ is $\bGi$-definably homeomorphic to a $\bGi$-definable subset of $[0,\infty ]^{2m}$ (which is  in particular bounded in $\bGi$).

Note however that  $\bG$-definable completions do not exist in general  for locally $\bG$-definably  compact spaces.  Suppose that $\bG=(\nG, <, +,  \ldots )$ is a semi-bounded o-minimal expansion of an ordered group $(\nG , <, +)$ (see \cite[Definition 1.5]{e1}). Let $X=\nG$  and suppose that there is a $\bG$-definable completion $\iota \colon X\hookrightarrow P$  of $X$. Since $P$ is $\bG$-definably compact, the limit $\lim _{t\to +\infty}\iota (t)$ exists in $P$, call it $a$. Let $(P_i, \theta _i)$ be a $\bG$-definable chart of $P$ such that $a\in P_i$. Let $\Pi _{j=1}^{n_i}(c_j^-, c_j^+)$ be an open and bounded box in $\Gamma ^{n_i}$ containing $\theta _i(a)$. By continuity, there exists $b\in \nG$ be such that $\theta _i\circ \iota ((b, +\infty ))\subseteq \Pi _{j=1}^{n_i}(c_j^-, c_j^+)$. Hence, $\theta _i\circ \iota _{|}\colon (b, +\infty )\to \Gamma ^{n_i}$ has bounded image and so, by  \cite[Fact 1.6 and Proposition 3.1 (1)]{e1}, this map must be eventually constant. This contradicts the fact that the map  is injective. 
\end{nrmk}





Recall that, in an arbitrary o-minimal structure,  a  definable subset $X$ with the induced topology is \textit{definably normal} if and only if  every two disjoint closed definable subsets of $X$ can be separated by disjoint open definable subsets of $X$. This corresponds to  $X$ being  $\T$-normal (Definition \ref{defn tnormal}) for the $\T$-topology given by the o-minimal site on $X$.

\begin{nrmk}[Definable normality]\label{nrmk def normal}
By Remark \ref{nrmk top} definable normality of a definable set does not change when moving between $\bG,$ $\bGi$ and $\bSg$. Therefore, by Remark \ref{nrmk def comp} and \cite[Theorem 2.11]{emp} we have:
\begin{itemize}
\item
If $\bG$ has definable Skolem functions, then every $\bGi$-definably compact set is $\bGi$-definably normal. 
\end{itemize}

On the other hand, if $\bG=(\nG, <, +,  \ldots )$ is an o-minimal expansion of an ordered group $(\nG , <, + ),$ then by \cite[Chapter 6 (3.5)]{vdd}, every $\bG$-definable set is  $\bG$-definably normal. However,  this is not true in $\bGi$ (and so it is also not true in $\bSg$), as shows the following example.
\end{nrmk}
 
\begin{expl}\label{expl def normal}
Let $a\in \nG$ and let $U=\Gamma ^2_{\infty }\setminus \{(\infty, a)\}$. Let $C=\ \nG \times \{a\}$ and let $D=\{\infty \}\times (\nGi \setminus \{a\})$. Then $U$ is an open $\bGi$-definable subset of $\Gamma ^2_{\infty }$ and  $C$ and $D$ are closed and disjoint $\bGi$-definable subsets of $U$. We claim that there are no disjoint open $\bGi$-definable subsets $V$ and $W$ of $U$ such that $C\subseteq V$ and $D\subseteq W$.

For a contradiction, suppose that such $V$ and $W$ exist. Let $b\in \nG$ be such that $b<a$. Let $f\colon \Gamma \to [b, a)$ be the $\bG$-definable function given by $f(t)=\inf \{x\in (b, a): \{t\}\times (x, a)\,\, \subseteq V\cap \nG ^2\}$. By \cite[Fact 1.6 and Proposition 3.1 (1)]{e1}, there is $s\in \nG$ and $c\in [b,a)$ such that $f_{|\,(s,\infty )}$ is constant and equal to $c$. So $(s,\infty )\times (c,a) \,\, \subseteq V$. Let $(\infty , u)\in \{\infty \}\times (c,a) \,\, \subseteq D$ be any point.  Since $W$ is an open $\bGi$-definable neighborhood of $D$ in $U$ it is an open $\bGi$-definable neighborhood of $(\infty ,u)$ in $U$. So there are  $c<u_1<u<u_2<a$ and $v<\infty $ such that $(v,\infty]\times (u_1,u_2)\,\, \subseteq W$. Let $w=\max \{s,v\}$. Then $(w,\infty)\times (u_1,u_2) \,\, \subseteq V\cap W,$ contradicting the fact that $V$ and $W$ are 
disjoint. 
\end{expl}





Let us now recall the following result from \cite{ep4} which will be true in $\bGi$ by working in $\bSg$. After a couple of lemmas we will prove a result (Theorem \ref{thm def normal ext} below) generalizing this fact.

\begin{fact}\cite[Theorem 2.20]{ep4}\label{fact CxN is normal}
Let $Z$ and $K$ be definable spaces in an o-minimal structure with definable Skolem functions with $Z$ definably normal and $K$ Hausdorff and definably compact. Then $Z\times K$ is definably normal.\qed
\end{fact}

So from now on let $\bM =(M, <, \ldots )$ be an arbitrary o-minimal structure  without end points and with definable Skolem functions.

\begin{lem}\label{prop basis of open normal}
If $I\subseteq M$ is an interval, then $I$ is definably normal. In particular, every open definable subset of $M$ has a finite cover by open, definable subsets which are definably normal. \footnote{Here the existence of definable Skolem functions is not needed.}
\end{lem}

\begin{proof}
Let $C$ and $D$ be disjoint, closed definable subsets of $I$. Then $C$ is a finite union of intervals. The result follows by induction on the number of such intervals. Suppose that $C$ is an interval. Let $c$ be the left end point of $C$ and $c'$ the right end point. Let $d=\sup \{x\in D:x<c\}$ and let $d'=\inf \{x\in D:c'<x\}$. Since $C$ and $D$ are closed in $I$ and disjoint, we have $d<c\leq c'<d'$. Now take $d<u<v<c$ and $c'<v'<u'<d'$. Then $U=((-\infty , u)\cup (u',\infty ))\cap I$ and $V=(v,v')$ are disjoint, open definable subset of $I$ such that $D\subseteq U$ and $C\subseteq V$. 

If $C$ is a union of $n+1$ intervals, let $C'$ be the union of the leftmost $n$ intervals and let $C''$ be the rightmost interval. By induction, there are open definable subsets $U', V'$ and $U'', V''$ of $I$ such that $D\subseteq U',$ $C'\subseteq V',$ $D\subseteq U'',$ $C''\subseteq V'',$  $U'\cap V'=\emptyset$ and $U''\cap V''=\emptyset $. Let $U=U'\cap U''$ and $V=V'\cup V''$. Then $U$ and $V$ are disjoint, open definable subset of $I$ such that $D\subseteq U$ and $C\subseteq V$.
\end{proof}

Below we use the following notation: for each $i\in \{1, \ldots, m\}$ let 
\[\pi_i\colon M^m\to M^{m-1}\]
be the projection  omitting the $i$'th coordinate and let
\[\pi '_i\colon M^m\to M\]
be the projection onto the $i$'th coordinate.
If $Z\subseteq M^{m-1}$ is a definable subset and  $f, g\colon Z\to M$ are continuous definable maps with $f<g$ 
then for each $i\in \{1,\ldots, m\},$ we let
\[[f, g]^i_{Z}=\{x\in M^m: \pi _i(x)\in Z\,\,\textrm{and}\,\,f\circ \pi _i(x)\leq x_i\leq g\circ \pi _i(x) \}\]
and we define 
\[(f,g)^i_Z, \,\,\,(f,g]^i_Z\,\,\,\textrm{and}\,\,\,[f,g)^i_Z\]
in a similar way.





\begin{lem}\label{lem pii closed}
The restriction 
\[\pi_{i|}\colon [f, g]^i _Z\to Z\]
 is a continuous, closed definable map.
\end{lem}

\begin{proof}
For $z\in Z$ let 
\[D(z)=\{((d_1^-,d_1^+),\ldots ,(d^-_{m-1},d^+_{m-1})\in M^{2(m-1)}: z\in \Pi _{i=1}^{m-1}(d^-_i,d^+_i)\cap Z\}\]
and for $d\in D(z)$ let
\[
U(z,d)=\Pi _{i=1}^{m-1}(d^-_i,d^+_i)\cap Z.
\]
Then $\{U(z,d)\}_{d\in D(z)}$ is a uniformly definable system of fundamental open neighborhoods of $z$ in $Z$. Moreover, since the relation $d\preceq d'$ on $D(z)$ given by $U(z,d)\subseteq U(z,d')$ is a definable downwards directed order on $D(z),$ by \cite[Lemma 4.2.18]{HrLo} (or \cite[Lemma 2.19]{Hr04}), there is a definable type $\beta$ on $D(z)$ such that for every $d\in D(z)$ we have $\{d' \in D(z): d'\preceq d\}\in \beta $. 

Let $S\subseteq [f, g]^i _Z$ be a closed definable subset. Suppose that $\pi_i(S)$ is not closed in $Z$. Then there is $z\in Z\setminus \pi_i(S)$ such that for all $d\in D(z)$ we have $U(z,d)\cap \pi_i(S)\neq \emptyset $. Then by definable Skolem functions, there is a definable map 
\[
h\colon D(z)\to S\subseteq [f,g]^i_Z
\]
 such that for every $d\in D(z)$ we have $\pi _i(h(d))\in U(z,d)\cap \pi _i(S)$. 
 
 Let $\alpha $ be the definable type on $S$ determined by the collection $\{A\subseteq S: h^{-1}(A)\in \beta \}$. Let $\alpha _1$ be the  definable type on $\pi _i(S)$ determined by the collection $\{A\subseteq \pi _i(S): (\pi _i)^{-1}(A)\in \alpha \}$ and let $\alpha _2$ be the  definable type on $M$ determined by the collection $\{A\subseteq M: (\pi '_i)^{-1}(A)\in \alpha \}$.

\begin{clm}\label{clm limit of def type}
The limit of $\alpha $ exists, it is of the form $(z,c)$ and belongs to $[f,g]^i_Z$.
\end{clm}

\begin{proof}
We have that $z$ is the limit of $\alpha _1$ i.e., for every open definable subset  $V$ of $Z$ such that $z\in V$ we have $V\in \alpha _1$. Indeed, given any such $V$ there is $d'$ such that $U(z,d')\subseteq V$ and 
\begin{eqnarray*}
h^{-1}((\pi _i)^{-1}(V))&\supseteq &h^{-1}((\pi _i)^{-1}(\pi _i(S)\cap U(z,d')))\\
&\supseteq & h^{-1}((S\cap (\pi _i)^{-1}(U(z,d'))))\\
&\supseteq & \{d''\in D(z):d''\preceq d'\}.
\end{eqnarray*}

Since $[f,g]^i_Z\in \alpha $ it follows that for every open definable subset  $V$ of $Z$ such that $z\in V$ we have $[f,g]^i_V=(\pi _i)^{-1}(V)\cap [f, g]^i_Z\in \alpha $.
On the other hand, since $\alpha _2$ is a definable type on  $M,$ by \cite[Lemma 2.3]{MarStein},  $\alpha _2$ is not a cut  and so $\alpha _2$ is determined by either : (i)  $\{b<x: b\in M\};$ (ii) $x=a; $ (iii) $\{b<x<a: b\in M, \,\,b<a\};$  (iv) $\{a<x<b: b\in M, \,\,a<b\};$ (v) $\{x<b:b\in M\}$ where $a\in M$. In the case (i) limit of $\alpha _2$ is $+\infty ,$ in the  cases (ii), (iii) and (iv)  $a$ is the limit of $\alpha _2$ and in case (v) $-\infty $ is the limit of $\alpha _2$.  Let $c$ be the limit of $\alpha _2$. 

We show that $f(z)\leq c\leq g(z)$. If $g(z)< c$ let $l$ be such that $g(z)<l<c$. Then since $g$ is continuous, there is an open definable subset  $V$ of $Z$ such that $z\in V$ and $g(v)<l$ for all $v\in V$. Case (v) does not occur since we cannot have $g(z)<-\infty ;$ in  the remaining cases  we would have $\emptyset =[f,g]^i_V \cap (\pi '_i)^{-1}((l, +\infty )) \in \alpha $ which is a absurd. If $c<f(z)$  let $l$ be such that $c<l<f(z)$. Then since $f$ is continuous, there is an open definable subset  $V$ of $Z$ such that $z\in V$ and $c<l<f(v)$ for all $v\in V$. Case (i)  does not happen; in the remaining  cases we would have $\emptyset =[f,g]^i_V \cap (\pi '_i)^{-1}((-\infty , l)) \in \alpha $ which is a absurd. 
\end{proof} 

It follows that  $(z,c)\in [f,g]^i _Z$ is the limit of $\alpha $. Since $S$ is closed in $[f, g]^i _Z$ and $\alpha $ is a definable type on $S,$ its limit  $(z,c)$ is in $S$. But then $z\in \pi _i(S)$ which is a contradiction. 
\end{proof}

The previous result allow us to obtain a slight generalization of \cite[Lemma 2.23]{ep4}. In that Lemma instead of $[f,g]^i _Z$ we have $Z\times [a,b]$. The proof of this new version is exactly the same using Lemma \ref{lem pii closed} instead of the fact that the projection $\pi \colon Z\times [a,b]\to Z$ is a continuous, closed definable map. For the readers convenience we include the details.

\begin{lem}\label{lem closed fiber in open}
Let $Z\subseteq M^{m-1}$ be a definably normal definable subset. Let $S\subseteq [f,g]^i_Z$ be a closed definable subset and $W\subseteq [f,g]^i_Z$ an open definable subset. Then for every closed definable subset $F\subseteq \pi _i(S)$ such that $S\cap (\pi _i)^{-1}(F)\subseteq W$ there is an open definable neighborhood $O$ of $F$ in $Z$ such that $O\subseteq \bar{O}\cap Z\subseteq \pi _i(W)$ and  $S\cap (\pi _i)^{-1}(\bar{O}\cap Z)\subseteq W$.
\end{lem}

\begin{proof}
Let  $W^c=([f,g]^i_Z)\setminus W$. If $S\subseteq W$ then since $\pi _i(S) \subseteq \pi _i(W)$ is closed in $Z$ (by Lemma \ref{lem pii closed}) and $Z$ is definably normal, there is an open definable neighborhood $O$ of $\pi _i(S)\supseteq F$ in $Z$ such that $O\subseteq \bar{O}\cap Z\subseteq \pi _i(W)$ and  so $S\cap (\pi _i)^{-1}(\bar{O}\cap Z)=S\subseteq W$. So we may  suppose that $S\cap W^c\neq \emptyset $.

For $z\in Z$ let $\{U(z,d)\}_{d\in D(z)}$ be the uniformly definable system of fundamental open neighborhoods of $z$ in $Z$ given above. Recall that by \cite[Lemma 4.2.18]{HrLo} (or \cite[Lemma 2.19]{Hr04}), there is a definable type $\beta $ on $D(z)$ such that for every $d\in D(z)$ we have $\{d' \in D(z): d'\preceq d\}\in \beta $. 

Suppose that $z\in F$ and  for all $d\in D(z),$ we have $(S\cap (\pi_i)^{-1}(U(z,d)))\cap W^c\neq \emptyset $.  Then by definable Skolem functions, there is a definable map \[h\colon D(z)\to S\cap W^c\subseteq [f,g]^i_Z\]
 such that for every $d\in D(z)$ we have $h(d)\in (S\cap (\pi_i)^{-1}(U(z,d)))\cap W^c$. 
 
Let $\alpha$ be the definable type on $S\cap W^c$ determined by the collection $\{A\subseteq S\cap W^c: h^{-1}(A)\in \beta \}$. We are in the set up of Claim \ref{clm limit of def type}, so the limit of $\alpha $ exists, it is of the form $(z,c)\in [f,g]^i_Z$. Since $S\cap W^c$ is closed and $\alpha$ is a definable type on $S\cap W^c$, its limit $(z,c)$ is in $S\cap W^c$. But then $(z,c)\in S\cap (\pi _i)^{-1}(z)\subseteq W^c$, which contradicts the assumption on $F$.

So for each $z\in F$ there is $d\in D(z)$ such that $S\cap (\pi _i)^{-1}(U(z,d))\subseteq W$. By definable Skolem functions there is a definable map $\epsilon \colon F\to M^{2(m-1)}$ such that for each $z\in F$ we have $\epsilon (z)\in D(z)$ and $S\cap (\pi _i)^{-1}(U(z,\epsilon (z)))\subseteq W$.  Then 
\[U(F,\epsilon )=\bigcup _{z\in F}U(z,\epsilon (z))\]
 is an open definable neighborhood of $F$ in $Z$ such that 
\[
S\cap (\pi _i)^{-1}(U(F,\epsilon )) =\bigcup _{z\in F}S\cap (\pi _i)^{-1}(U(z,\epsilon (z)))\subseteq W.
\]

Since $U(F,\epsilon )\cap \pi _i(W)$ is an open definable neighborhood of $F$ in $Z,$ $F$ is closed in $Z$ and $Z$ is definably normal, there is an open definable neighborhood $O$ of $F$ in $Z$ such that $O\subseteq \bar{O}\cap Z\subseteq U(F,\epsilon )\cap \pi _i(W)\subseteq \pi _i(W)$ and  $S\cap (\pi _i)^{-1}(\bar{O}\cap Z)\subseteq W$.
\end{proof}

We now obtain the following generalization of the fact that if $Z\subseteq M^{m-1}$ is definably normal, then $Z\times [a,b]$ is definably normal. We omit the proof since it is exactly the same as in \cite[Lemma 2.24, proof of Proposition 2.21]{ep4} using Lemma \ref{lem closed fiber in open} and basic o-minimality from \cite{vdd}.

\begin{prop}\label{prop def normal ext}
If $Z\subseteq M^{m-1}$ is definably normal, then $[f,g]^i_Z$ is also definably normal.\qed
\end{prop}

We now extend this result a bit further. First we introduce some notation. Given  $I\subseteq \{1,\ldots, m\},$ we let $I'=\{1,\ldots ,m\}\setminus I,$ we let 
\[\pi _I\colon M\to M ^{m-|I|}\]
be the projection omitting the coordinates in $I$ and we let 
\[\pi '_I\colon M\to M ^{|I|}\]
be  the projection onto the coordinates in $I$ (so $\pi '_I=\pi _{I'}$). Given  $Z\subseteq M ^{|I|}$ a definable set and families $\{f^l\colon Z\to M\}_{l\in I'}$ and $\{g^l\colon Z\to M\}_{l\in I'}$ of continuous definable functions, we set
{\small
\[
[\{f^l\}_{l\in I'}, \{g^l\}_{l\in I'}]^I_{Z}=\{x\in M^m: \pi _{I'}(x)\in Z\,\,\textrm{and}\,\,f^l\circ \pi _{I'}(x)\leq x_l\leq g^l\circ \pi _{I'}(x), \,\,\forall l\in I'\}.
\]
}
Similarly, we define
\[
(\{f^l\}_{l\in I'},\{g^l\}_{l\in I'})^I_Z, \,\,\,(\{f^l\}_{l\in I'},\{g^l\}_{l\in I'}]^I_Z\,\,\,\textrm{and}\,\,\,[\{f^l\}_{l\in I'},\{g^l\}_{l\in I'})^I_Z.
\]
If $g^l=g$ for all $l\in I'$ we write
\[
[\{f^l\}_{l\in I'}, g]^I_{Z}, (\{f^l\}_{l\in I'},g)^I_Z, \,\,\,(\{f^l\}_{l\in I'},g]^I_Z\,\,\,\textrm{and}\,\,\,[\{f^l\}_{l\in I'},g)^I_Z
\]
instead and similarly for the case $f^l=f$ for all $l\in I'$ (though we will not need this other case).

Note that if $I=\{1,\ldots, m\},$ then 
\[
[\{f^l\}_{l\in I'}, \{g^l\}_{l\in I'}]^I_Z=(\{f^l\}_{l\in I'}, \{g^l\}_{l\in I'}]^I_Z=\ldots =Z.
\]

\begin{thm}\label{thm def normal ext}
If $Z\subseteq M^{|I|}$ is definably normal, then $[\{f^l\}_{l\in I'},\{g^l\}_{l\in I'}]^I_Z$ is also definably normal.
\end{thm}

\begin{proof}
The proof is by induction on $m$. The case $m=0$ is clear, so assume $m>0$ and the result holds for $m-1$.  If $|I'|=0,$ then $I=\{1,\ldots, m\}$ and so  $[\{f^l\}_{l\in I'}, \{g^l\}_{l\in I'}]^I_Z=Z$. So suppose that $|I'|>0$ and choose $i\in I'$ and set $J=I$ and   $J'=I'\setminus \{i\}$. Then $X=[\{f^l\}_{l\in J'},\{g^l\}_{l\in J'}]^J_Z\subseteq M^{m-1}$ is definably normal by the induction hypothesis. 

Now let $F\colon X\to M$ be given by $F(u)=f^{i}\circ \pi _{J'}(u)$ and $G\colon X\to M$ be given by $G(u)=g^i \circ \pi _{J'}(u)$. Then $F$ and $G$ are continuous definable functions, $F\circ \pi _{i}(x)=f^i\circ \pi _{J'}(\pi _{i}(x))=f^i\circ \pi _{I'}(x)$ and $G\circ \pi _{i}(x)=g^i\circ \pi _{I'}(x)$. Therefore, 
\[[\{f^l\}_{l\in I'}, \{g^l\}_{l\in I'}]^I_Z=[F,G]^i_X\]
and so, by Proposition \ref{prop def normal ext}, this set is definably normal.
\end{proof}

To proceed we need to recall the following results: 

\begin{fact}\cite[Theorem 2.2]{Edal17}\label{thm open cells}
Let $U$ be an open definable subset of $M^m$. 
Then $U$ is a finite union of open definable sets definably homeomorphic, by reordering of coordinates,  to open cells. \qed
\end{fact}

Let  $\pi \colon M^m \to M^{m-1}$ be the projection onto the first $m-1$ coordinates. Clearly the following also holds with $\pi $ replaced by each $\pi _i\colon M^m\to M^{m-1}$ as defined above.

\begin{fact}\cite[Theorem 3.4]{dem}\label{fact continuous-choice}
Let $U$ be an open definable subset of $M^m$.  Then there is a finite cover $\{U_j:j=1,\ldots, l\}$ of $\pi (U)$ by open definable subsets such that for each $i$ there is a continuous definable section $s_j\colon U_j\to U$  of $\pi $ (i.e.  $\pi\circ s_j={\rm id}_{U_j}$).\qed
\end{fact}

We now go back to the setting $\bG,$ $\bGi$ and $\bSg$. Note that by the several remarks made previously  regarding definability and definable topological notions (including connectedness, compactness, normality) we may safely omit the prefix $\bG$ (resp. $\bGi$ and $\bSg$) and simply say  definable (resp. definably connected, definably compact, definably normal) when talking about subsets of $\nGm$ or $\nGim$ and $\nSgm$.\\

Recall that given $I\subseteq \{1,\ldots, m\},$  $I'=\{1,\ldots ,m\}\setminus I,$
\[\pi _I\colon \nSgm\to \Sigma ^{m-|I|}\]
is the projection omitting the coordinates in $I$ and
\[\pi '_I\colon \nSgm\to \Sigma ^{|I|}\]
is the projection onto the coordinates in $I$ (so $\pi '_I=\pi _{I'}$). If $Z\subseteq \Sigma ^{|I|}$ is a definable set and $\{f^l\colon Z\to \nSg\}_{l\in I'}$ a family of continuous definable functions, we set
\[
(\{f^l\}_{l\in I'}, \infty]^I_{Z}=\{x\in \nSg^m: \pi _{I'}(x)\in Z\,\,\textrm{and}\,\,f^l\circ \pi _{I'}(x)< x_l\leq \infty \,\,\textrm{for all}\,\,l\in I'\}.
\]
Recall also that, if $I=\{1,\ldots, m\},$ then $(\{f^l\}_{l\in I'},\infty ]^I_Z=Z$. \\

\begin{prop}\label{prop infty open gen cells}
Let $O\subseteq \nGim$ be an open definable subset. Then $O$ is a finite union of open definable subsets of the form $(\{f^l_V\}_{l\in I'},\infty ]^I_{V}$ with $I\subseteq \{1,\ldots, m\}$ and $V$ an open definable subset of $\Gamma ^{|I|}$. 
\end{prop}

\begin{proof} Let $I\subsetneq \{1,\ldots, m\}$ be such that $O_I\neq \emptyset$.   For a $A\subseteq \tau_I(O_I)$ and $B\subseteq \Gamma_\infty^{|I'|}$, we let $A\star_I B$ denote the set defined by 
\[
(a_1,\ldots, a_m)\in A\star _I B \Leftrightarrow 
\begin{cases}
a_i\in B & \text{ if $i\in I'$ }\\
a_i\in \pi_i'(A) & \text{ if $i\in I$}
\end{cases}
\] 
where $\pi_i'\colon \tau_I(O_I)\to \Gamma$ is the restriction of the projection onto the $i$'th coordinate. 
Consider the definable set
\[
U_I\coloneqq \left\{(a,b)\in \tau_I(O_I)\times \Gamma: \begin{array}{l}
\text{there is an open neighborhood $U$ of $a$ in $\tau_I(O_I)$}\\
\text{and $c\in \Gamma$ such that $c<b$ and $U\star _I (c,\infty]^{|I'|}\subseteq O$}
\end{array}
\right\}.
\]
Let $\pi\colon U_I\to \tau_I(O_I)\subseteq \Gamma ^{|I|}$ be the projection to the first $|I|$-coordinates. 

\begin{clm} $\pi(U_I)=\tau_I(O_I)$.  
\end{clm} 

The left-to-right inclusion is trivial. From right-to-left, let $a$ be an element of $\tau_I(O_I)$. Then there is $a'\in O$ such that $\tau_I(a')=a$. Since $O$ is open, let $V=\prod_{i=1}^m J_i\subseteq O$ be a product of basic open sets containing $a'$. Then $U\coloneqq \tau_I(V)$ is a neiboourhood of $a$ in $\tau_I(O_I)$. Moreover, for $i\in I'$, we may suppose $J_i=(b_i,\infty]$ for some $b_i\in \Gamma$. Let $c=\max\{b_i: i\in I'\}$. Then $U\star _I (c,\infty]^{|I'|}\subseteq O$. In particular, $(a,b)\in U_I$ for every element $b\in (c,\infty)$, which shows that $a\in \pi(U_I)$. This proves the claim. 

\begin{clm} 
The set $U_I$ is open. 
\end{clm} 

Let $(a,b)$ be an element in $U_I$. By definition, let $U$ be a neighborhood of $a$ in $\tau_I(O_I)$ and $c\in \Gamma$ such that $c<b$ and $U\star _I (c,\infty]^{|I'|}\subseteq O$. We let the reader verify that $U\times (c,\infty)\subseteq U_I$ which shows the claim.

By Fact \ref{fact continuous-choice} applied to $U_I$ and the projection $\pi\colon U_I\to \Gamma^{|I|}$, there exists a finite cover $\cV_I$ of $\pi(U_I)=\tau _I(O_I)$ by open definable subsets and for each $V\in \cV_I$ there is a continuous definable section $s_V\colon V\to U_I$ of $\pi$. We let the reader verify that the previous claims imply that the family
\begin{equation}\label{eq:family}\tag{$\ast$}
\{(\{s_V\}_{i\in I'},\infty ]^I_V\}_{V\in \cV_I}
\end{equation} 
is a finite family of open definable subsets of $O$ covering $O_I$. 

The union of $O_{\{1\,\ldots,m\}}$ and all families as in \eqref{eq:family} for all $I\subsetneq \{1,\ldots,m\}$ provide the required family of open sets. 
\end{proof}

We are finally ready to prove the main goal of this subsection:

\begin{thm}\label{thm basis of open normal}
Suppose that  $\bG=(\nG, <, +,  \ldots )$ is an o-minimal expansion of an ordered group $(\nG , <, + )$. Let $Z$ be a definably locally closed  subset of $\nGim$. Then $Z$ is the union of finitely many relatively open, definable subsets which are definably normal. 
\end{thm}

\begin{proof}
Since a definably locally closed set is of the form $O\cap S$ with $O$ open definable and $S$ closed definable, it is a closed definable subset of an  open definable subset. Therefore it is enough to prove the result for an open definable  subset $O$ of $\nGim$.

By Proposition \ref{prop infty open gen cells} it is enough to show that  open definable subsets of $O$ of the form $(\{f^l_V\}_{l\in I'},\infty ]^I_{V}$ with $I\subseteq \{1,\ldots, m\}$ and $V$ an open definable subset of $\Gamma ^{|I|}$ are definably normal. So fix $I\subseteq \{1,\ldots, m\}$ and $V$ an open definable subset of $\Gamma ^{|I|}$. If $I'=\emptyset ,$ the $(\{f^l_V\}_{l\in I'},\infty ]^I_{V}=V\subseteq \nGm$ is definably normal (Remark \ref{nrmk def normal}). Suppose otherwise and let $P=(\{f^l_V\}_{l\in I'},\infty ]^I_{V}$ and consider
\[X=(\{f^l_V\}_{l\in I'}, \{2f^l_V\}_{l\in I'})^I_V,\,\,\,\,\,Y=(\{\frac{3}{2}f^l_V\}_{l\in I'}, \{\frac{5}{2}f^l_V\}_{l\in I'})^I_V\]
and
\[
Z=[\{2f^l_V\}_{l\in I'}, \infty ]^I_V.
\]
Since $X, Y\subseteq \nGm,$ by Remark \ref{nrmk def normal} they are both definably normal, and by Theorem \ref{thm def normal ext}, $Z$ is also definably normal.

Let $C,D\subseteq P$ be closed, disjoint definable subsets.  Since $C\cap Y, D\cap Y\subseteq Y$ are closed, disjoint definable subsets and $Y$ is definably normal, there are $U_Y,V_Y\subseteq Y$ open, disjoint definable subsets of $Y$ such that $C\cap Y\subseteq U_Y$ and $D\cap Y\subseteq V_Y$. Similarly,  there are $U_Z, V_Z\subseteq Z$ open, disjoint  definable subsets of $Z$ such that $C\cap Z\subseteq U_Z$ and  $D\cap Z\subseteq V_Z$ and there are $U_X, V_X\subseteq X$ open, disjoint  definable subsets of $X$ such that $C\cap X\subseteq U_X$ and  $D\cap X\subseteq V_X$.  Again by definable normality, let $U'_X\subseteq X$ be an open definable subset of $X$ such that $C\cap X\subseteq U'_X\subseteq \bar{U'_X}\subseteq U_X$ and similarly let $U'_Z\subseteq Z$ be an open definable subset of $Z$ such that $C\cap Z\subseteq U'_Z\subseteq \bar{U'_Z}\subseteq U_Z$ and  let $U'_Y\subseteq Y$ be an open definable subset of $Y$ such that $C\cap Y\subseteq U'_Y\subseteq \bar{U'_Y}\subseteq U_Y$.

Let  $V'_Y=V_Y\setminus (\bar{U'_Z}\cup \bar{U'_X}),$ $V'_X=V_X\setminus (\bar{U'_Y}\cup \bar{U'_Z})$ and $V'_Z=V_Z\setminus (\bar{U'_Y}\cup \bar{U'_X})$.  Let 
\[
Q=\{x\in \nGm:\pi _{I'}(x)\in V\,\,\textrm{and}\,\,x_l=(2f^l_V)\circ \pi _{I'}(x)\,\,\textrm{for all}\,\, l\in I'\}.
\] 
We have that $Q$ is a closed definable subset of $P$ and of $Y$ (the $f^l_V$'s are continuous). Let $U''_Z=U'_Z\setminus Q$ and let $V''_Z=V'_Z\setminus Q$. 

Let $U=U'_Y\cup U''_Z\cup U'_X$ and let  $V=V'_Y\cup V''_Z\cup V'_X$. Since $C\cap Q\subseteq U'_Y$ we clearly have $C\subseteq U$. On the other hand, $D\cap \bar{U'_X}=D\cap \bar{U'_Z}=D\cap \bar{U'_Y}=\emptyset ,$ otherwise $D\cap U_X\neq \emptyset$ and $U_X\cap V_X\neq \emptyset $ or similarly $U_Z\cap V_Z\neq \emptyset $ or $U_Y\cap V_Y\neq \emptyset $.  Thus, since $D\cap Q\subseteq V'_Y,$ we also  have $D\subseteq V$. By construction we  have $U\cap V=\emptyset $. Now $U'_Y$ is open in $Y$ and $Y$ is open in $P,$ $U'_X$ is open in $X$ and $X$ is open in $P$ and $U''_Z$ is the interior of $U'_Z$ in $P,$ so $U$ is open in $P$. Similarly, $V$ is open in $P$.
\end{proof}


\section{Sites on definable sets in ACVF}\label{section t-top in acvf}

In this section we will introduce the $\vgs$-site on definable sets of algebraically closed non-trivially valued fields and the $\hvgs$-site on their corresponding stable completion. We will further show that the latter forms a $\T$-space in the sense of Section \ref{section more on t-sheaves}. Definable compactness and normality will be discussed in Subection \ref{subsection  vgs-comp and vgs-normal}. 

We begin by recalling the needed model-theoretic background on ACVF. Some familiarity with valued fields and their model-theory will be however assumed. For further references we refer the reader to \cite{pre-ro-84} or \cite{HHM}.

\subsection{Preliminaries on ACVF}\label{sec prelim acvf}

Given $K$ an algebraically closed field with a nontrivial valuation $\val \colon K^{\times }\to \nG$ we let $\nG =(\nG,<, 0, -, +)$ denote the value group, $\cR$ the valuation ring, $\cM$ the maximal ideal, $k$ the residue field and $\res \colon \cR\to k$ the residue map. We let $\nGi$ be defined as in Section \ref{section t-top in bGi} and extend the valuation  $\val \colon K\to \nGi$ by setting $\val (0)=\infty$.

Let $\cL_{k, \nG}$ be the three sorted language with sorts $(K,\cL_{\mathrm{ring}})$, $(k,\cL_{\mathrm{ring}})$ and $(\Gamma_\infty$, $\cL_{\mathrm{og}}^\infty)$ (i.e., the language of ordered groups with an additional constant symbol for~$\infty$), together with symbols for the valuation $\val$ and the map ${\rm Res}\colon K^2\to k$ sending $(x,y)$ to ${\rm res}(xy^{-1})$ if $\val(x)\geq \val(y)$ and $y\neq 0$, and to $0$ otherwise. Both the residue field and the value group are stably embedded, that is, every definable subset of $\nGi$ (resp. $k_K^n$) is already definable in $(\Gamma_\infty$, $\cL_{\mathrm{og}}^\infty)$ (resp. $(k,\cL_{\mathrm{ring}})$).

Let $\cL_{\cG}$ be the geometric language extending $\cL_{k, \nG}$ with the following new sorts: for each positive integer $n$ a sort $S_n=\{\Lambda\subseteq K^n:\,\,\cR\textrm{-submodule definably isomorphic to}\,\, \cR^n\}$ of lattices in $K^n,$ and a sort $T_n=  \{{\rm red}(\Lambda):\Lambda\in S_n\}$ of $k$-vector spaces where ${\rm red}(\Lambda)$ is the reduction modulo $\cM$ of $\Lambda$. By  \cite[Theorem 2.1.1]{HHM}, the theory ACVF of algebraically closed fields with a nontrivial valuation has quantifier elimination in the language $\cL _{k, \nG}$ and by \cite[Theorem 7.3]{HHM} it has quantifier elimination and elimination of imaginaries in the language $\cL_{\cG}$.

Below we fix  a monster model  $\UU$ of ACVF in the language $\cL _{\cG},$ and assume that all sets of parameters we consider are small substructures $C$ of  $\UU$ (i.e.  subsets $C$ of $\UU$ such that ${\rm dcl}(C)=C$ and $|C|<|\UU|$),  and all models $K$ of ACVF considered are elementary substructures of $\UU,$ again of smaller cardinality. 

If $C$ is a small substructure of $\UU,$ by an algebraic variety over $C$ we mean a separated reduced scheme of finite type over the valued field sort of $C$. Note that if $V$ is an algebraic variety over $C$ and $K$ is a model of ACVF containing $C$, we can view the $K$-rational points $V(K)$ of $V$ as a constructible subset of some affine $n$-space of $K$ via some affine charts. 

By a $C$-definable set we mean a $C$-definable subset of some product of sorts and of varieties over $C$. We denote by 
\[\Df _C\]
the category whose objects are $C$-definable sets and whose morphisms are $C$-definable maps between $C$-definable sets (i.e. maps between $C$-definable sets whose graphs are $C$-definable sets). Given a substructure $F$ containing $C$ we have an inclusion functor $\Df _C\to \Df _F$ since every $C$-definable set is also $F$-definable.

Given an algebraic variety $V$ over a small model $K$ we can put on $V(K)$, besides the Zariski topology, the {\it valuation topology}, which is induced via charts by the valuation topology on $\Aa ^n_K(K)$. Recall that the  valuation topology on $\Aa ^1(K)$ is generated by the open balls
\[B^{\circ}(a,\gamma)=\{x\in \Aa^1_K(K):\val(x-a)>\gamma \}\]
centered at $a\in \Aa^1_K(K)$ and with radius $\gamma \in \nGi$ and, the valuation topology on $\Aa ^n_K(K)$ is the product topology. Since regular functions are continuous with respect to the valuation topology in affine space, the valuation topology on an arbitrary variety does not depend on the choice of the embedding. The valuation topology however is not suitable for developing cohomology as it is totally disconnected (the open ball are also closed) and not necessarily locally compact. 

In Subsection \ref{subsection vgs-sites} we will introduce an appropriate site on a definable set which will replace the valuation topology. However, with this site, definable sets do not form a $\T$-space. To obtain results for the cohomology theory on such sites  we have to go through another category associated to $\Df _C,$ the category of the stable completions of objects of $\Df_C$ with morphisms the $C$-pro-definable morphisms between such objects. We then introduce another site, isomorphic to the previous one, in the stable completion of a definable set. In Subsection \ref{subsection  vgs-sites are t-top} we show, using Hrushovski and Loeser's main theorem, that the open sets of the site make the stable completion of a definable set into a  $\T$-space. Finally in Subsection \ref{subsection  vgs-comp and vgs-normal} we  show that in the stable completion of a definable set, for the associated notion of $\T$-normality, $\T$-locally closed subsets are finite unions of $\T$-open subsets which are $\T$-normal.  \\

In what follows, for simplicity and when there is no risk of confusion, the $K$ in $\Aa^n_K(K),$ $\Pp^n_K(K),$ $\Gamma (K), V(K)$ etc., will be omitted. Also, different concepts from \cite{HrLo} will not be introduced with their original definition but by an equivalent characterization better suiting the purposes of the present article. 

\subsection{The~\texorpdfstring{$\vgs$}{v}-sites}\label{subsection vgs-sites}
We start with the definition of the site replacing the valuation topology. As we will explain below this site is  based on the work of Hrushovski and Loeser \cite{HrLo}. We then recall Hrushovski and Loeser's category of the stable completions of objects of $\Df_C$ and the topology on such spaces. In general, this topology is also not suitable for cohomology (it is not locally compact) and we will replace it  by a site as well.  We work over some model $K$ of ACVF.





\begin{defn}[The $\vgs$-site]\label{def vg site}
Let $V$ be an algebraic variety over $K$. A subset of $V$ is a {\it basic} $\vgs$-{\it open subset} if it is  of the form
\[\bigcap _{j\in J}\{u\in U_j:\val (f_j(u))<\val (g_j(u))\}\]
where $J$ is a finite set and $f_j,g_j\in \cO _V(U_j)$ are regular functions on a Zariski open subset $U_j\subseteq V$. A subset of $V$ is a $\vgs$-{\it open subset} if it is a finite union of basic $\vgs$-open subsets; it is $\vgs$-{\it closed} if it is the complement of a $\vgs$-open subset.

A subset of $V\times \nGim$ is a \emph{basic $\vgs$-open subset} if and only if its pullback under $\id \times \val\colon V\times {\mathbb A} ^m\to V\times \nGim$ is a  basic $\vgs$-open subset of $V\times {\mathbb A}^m;$ $\vgs$-open and $\vgs$-closed subsets of $V\times \nGim$ are defined analogously.   

If $X\subseteq V\times \nGim$ is a definable subset, we say that a subset of $X$ is $\vgs$-open (resp. $\vgs$-closed) if and only if it is of the form $X\cap O$ (resp. $X\cap D$) where $O$ (resp. $D$) is a $\vgs$-open (resp. $\vgs$-closed) subset of $V\times \nGim$. Similarly, a \emph{basic $\vgs$-open subset of $X$} is a set of the form $X\cap O$ where $O$ is basic $\vgs$-subset of $V\times \nGim$. 

The {\it $\vgs$-site on $X$},  denoted $X_{\vgs},$ is the category $\op(X_{\vgs})$ whose objects are the $\vgs$-open subsets of $X,$ the morphisms are the inclusions and the admissible covers $\cov(U)$ of $U\in \op (X_{\vgs})$ are covers by $\vgs$-open subsets of $X$ with finite subcovers. 
\end{defn}


\begin{nrmk}[v-open and g-open subsets]\label{nrmk v+g open}
Note that if $V$ is an  algebraic variety over $K$, then:
\begin{itemize}
\item[-]
since regular functions are continuous for the valuation topology, a $\vgs$-open subset is v-{\it open} i.e. it is a definable subset which is open in the valuation topology;
\item[-]
since the constant function zero is a regular function  and $\val (f(u))<\infty $ is the same as $f(u)\neq 0,$ it follows that a $\vgs$-open subset is g-{\it open} i.e. it is a positive finite Boolean combination of Zariski closed sets, Zariski open sets and sets of the form $\{u\in U:\val (f(u))<\val (g(u))\}$ with $f, g\in \cO_V(U)$.
\end{itemize}
Thus it follows from the characterization of subsets which are both v-closed and g-closed when $V$ is affine or projective (\cite[Proposition 3.7.3]{HrLo}, see also below) that the $\vgs$-open subsets as described above are exactly the subsets which are both v-open  and g-open. The sets that are both  v-open  and g-open (resp. v-closed and g-closed) are called in \cite{HrLo} $\vgs$-open subsets (resp. v+g-closed subsets). 
\end{nrmk}

The following will be useful later:

\begin{fact}[{\cite[Proposition 3.7.3]{HrLo}}]\label{fact:vg_closed} 
Let $V$ be an an affine (resp. projective)  algebraic variety over $K$. Then a subset of $V$ is $\vgs$-closed if and only if it is 
of the form 
\[\bigcup _{i\in I}\bigcap _{j\in J}\{x\in V:h_{ij}(x)=0,\,\,\val(f_{ij}(x))\leq \val(g_{ij}(x))\}\]
where $I,J$ are finite and $h_{ij}, f_{ij}, g_{ij}\in \cO_V(V)$ are regular functions on $V$ (resp. homogeneous polynomials on $V$).\qed
\end{fact}

\begin{nrmk}[$\vgs$-opens and valuation topology]\label{nrmk v+g-topology}
The $\vgs$-open subsets of an algebraic variety are a basis for the valuation topology. However, not every definable open subset is a $\vgs$-open. For example in $\Aa ^1,$ the valuation ring $\cR=\bigcup _{a\in \cR}(a+\cM)$ is definable, open in the valuation topology but it is not $\vgs$-open.

In $\nGim,$ by Fact \ref{fact:vg_closed}, the $\vgs$-open subsets coincide with the open definable subsets of $\nGim$ for the product of the order topology on $\nGi$ (Remark \ref{nrmk top}). So the $\vgs$-site $X_{\vgs}$ on a definable subset $X\subseteq \nGim$ is the o-minimal site $X_{\df}$ as defined in Definition \ref{def T-space in Gi}. 
\end{nrmk}

\begin{nrmk}[Definable sets with the $\vgs$-site are not $\T$-spaces]
Let us say that a subset of a definable set of $V\times \nGin$ is a {\it $\vgs$-subset} if it is a finite boolean combination of $\vgs$-open subsets. Note that by quantifier elimination, the $\vgs$-subsets are exactly the definable subsets. 

It follows that the $\vgs$-open subsets of a definable set of $V\times \nGin$ do not form a $\T$-topology unless $V$ is finite, since condition (iii) of Definition \ref{def:T-topology} fails: any $\vgs$-open is always a disjoint union of clopen $\vgs$-subsets (i.e. definable subsets).
\end{nrmk} 


For the readers convenience, below we recall the definition of the stable completion $\widehat{X}$ of a definable set $X$. When $X$ is moreover a subset of $V\times\nGin$ for an algebraic variety $V$, we will equip $\widehat{X}$ with a topology and a site, making $\widehat{X}$ into a $\T$-space.

Let $B$ be any subset of $\UU$ and $x$ be a tuple of variables with length $|x|$. We denote as usual by $S_x(B)$ the space of types over $B$ in the variables $x,$ that is the Stone space of the Boolean algebra of formulas with free variables contained in $x$ and parameters from $B$ up to equivalence over ACVF. If $B$ is a model of ACVF, then $S_x(B)$ can be characterized  as the set of ultrafilters of definable subsets of $B^{|x|}$ over $B$. 

Types in $S_x(\UU)$ are called {\it global types}. If $C$ is a small substructure, a {\it global type $p\in S_x(\UU)$ is $C$-definable} (or definable over $C$) if and only if for every $\cL_{\cG}$-formula $\phi (x,y)$ there is an $\cL_{\cG}$-formula $d_p(\phi )(y)$ over $C$ such that for all $b\in \UU ^{|y|}$ we have $\phi (x,b)\in p$ if and only if $d_p(\phi )(b)$ holds in $\UU$. Equivalently, since $\UU$ is a model, a global type $p\in S_x(\UU)$ is $C$-definable if and only if for every $\emptyset $-definable family $\{X_t\}_{t\in T}$ of definable subsets of $\UU ^{|x|}$ there is a $C$-definable subset $S\subseteq T$ such that $X_t\in p$ if and only if $t\in S$.

A type $p\in S_x(B)$ {\it concentrates on} a $C$-definable set $X$ if it contains a formula defining $X$. The notation $S_X(B)$ is used to denote the subset of $S_x(B)$ of all such types. Note if $B$ is a model, then $S_X(B)$ can be identified with $\tilde{X(B)},$ the set of ultrafilters of definable subsets of $X(B)$ over  $B$ - this notation was used in previous sections in the o-minimal setting. Note for example, if $K$ is an algebraically closed field (a model of the theory ACF), then by quantifier elimination in ACF,  $\Aa_K^n=\mathrm{Spec}(K[X_1, \ldots, X_n])$ can be canonically identified with $\tilde{K^n}$. 


The {\it stable completion of $X$ over $C$}, denoted by $\widehat{X}(C),$ can be characterized in ACVF as the set of $C$-definable global types  concentrated on $X$ which are orthogonal to $\Gamma $.  Recall that a global type $p$  is {\it orthogonal} to $\Gamma $ if and only if  for every definable function $h\colon X\to \nGi $ the pushforward $h_*(p)$ is a global type  concentrated on a point of $\Gamma _{\infty };$ the pushforward $h_*(p)$ is the global type given by $\{Z\subseteq \Gamma : Z\,\,\,\textrm{definable and}\,\,\,h^{-1}(Z)\in p\}$. Recall also that, by (\cite[Proposition 2.9.1]{HrLo}, a $C$-definable global type $p$ is orthogonal to $\Gamma $ if and only if $p$ is stably dominated if and only if $p$ is generically stable. These other characterizations of orthogonality to $\Gamma $ of an $C$-definable type will be used when needed. 

If $f\colon X\to Y$ is a $C$-definable map, then we have a map $\hat{f}\colon \hat{X}(C)\to \hat{Y}(C)$ given by $\hat{f}(p)=f_*(p)$. Indeed, $f_*(p)$ is clearly an $C$-definable global type on $Y$ and for any definable $h\colon Y\to \Gamma$, we have $h\circ f\colon X\to \Gamma$ is definable and $(h\circ f)_*(p)=h_*(f_*(p))$ is a global type concentrated on a point of $\Gamma$. However there are more maps that one has to consider between the stable completions of definable sets. 
To introduce these maps we need to introduce the pro-definable structure associated to the stable completion of a definable set.

Let $\mathrm{Pro}(\Df _C)$ be the category of {\it $C$-pro-definable sets} i.e. the category whose objects are  filtrant projective limits of functors
\[
\lpro i \Ho_{\Df_C}(\bullet,X_i),
\]
where $(X_i)_{i\in I}$ is a cofiltering system in $\Df_C$ 
indexed by a small directed partially ordered set, and the morphisms are the natural transformations of such functors. By a result of Kamensky \cite{Ka07}, the functor of ``taking $\UU$-points'' induces an equivalence of categories between the category $\mathrm{Pro}(\Df_C)$ and the sub-category of the category of sets whose objects and morphisms are projective limits of $\UU$-points of definable sets indexed by a small directed partially ordered set. So on can identify a pro-definable set $X$ represented by $(X_i)_{i\in I}$ with $X(\UU)=\lpro i X_i(\UU)$. 

The sets of morphisms are related by
\[\Ho_{\mathrm{Pro}(\Df_C)}(X,Y) \simeq \,\, \lpro j \lind i \Ho_{\Df _C}(X_i,Y_j)\]
and we call the elements of $\Ho_{\mathrm{Pro}(\Df_C)}(X,Y)$ the \emph{$C$-pro-definable morphisms} between $X$ and $Y$. \\

Below, by  {\it strict $C$-pro-definable} set we mean a  $C$-pro-definable set $X$ for which there is  a cofiltering system $(X_i)_{i\in I}$ in $\Df_C$ representing $X$ with the transition maps $\pi _{i,i'}\colon X_i\to X_{i'}$ being surjective for all $i\geq i'$ in $I,$ or  equivalently, with the $C$-pro-definable projection maps $\pi _{i}\colon X\to X_i,$ represented by the transition maps $\pi _{i,i'}\colon X_i\to X_{i'}$ for all $i\geq i',$ being surjective for all $i\in I$.\\

\begin{fact}\cite[Theorem 3.1.1]{HrLo}\label{fact hat is pro-def}
 For every $C$-definable set $X$ there is a canonical strict $C$-pro-definable set $E$ and a canonical identification $\widehat{X}(F)=E(F)$ for every substructure $F$ containing $C$. Moreover, if $f\colon X\to Y$ is a morphism in $\Df _C,$ then the induced map $\hat{f}\colon \widehat{X}(F)\to \widehat{Y}(F)$ is a morphism of $C$-pro-definable sets. \qed
\end{fact}

Thus we have a category 
\[\hDf _C\] 
whose objects are of the form $\widehat{X}(C)$ for $X$ a $C$-definable set and whose morphisms are the $C$-pro-definable morphisms between such objects. Furthermore, for every substructure $F$ containing $C$, we have a functor
\begin{align*}
\Df _C\to \hDf _F \\
X\mapsto \widehat{X}(F).  
\end{align*}

\begin{expl}[The stable completion of the affine line]\label{exa:gen_ball}   
Let us describe the stable completion $\widehat{\Aa ^1}(K)$ of the affine line.  Let $\eta_{\cR}$ denote the \emph{generic type of the valuation ring} $\cR,$ that is, the global type in $S_x(\UU)$ where $|x|=1,$ determined by the following rule: a definable set $X\subseteq \UU$ is in $\eta_{\cR}$ if and only if there is $m\geq 1$ and there are $b_1,\ldots,b_m\in \cR$  such that $\val (b_i-b_j)=0$ for all $i\neq j$ and 
\[
\cR\setminus (\bigcup_{i=1}^m b_i+\cM) \subseteq X.
\]
By quantifier elimination and first-order logic compactness, for every formula $\phi(x,y)$ there is an integer $m_{\phi}$ such that for all $t\in \UU^{|y|}$, if $X_t$ is defined by $\phi(x,t)$, then for all $m>m_{\phi }$ and all $b_1,\ldots,b_m\in \cR$  such that $\val (b_i-b_j)=0$ for all $i\neq j$ we have
\[
\cR\setminus (\bigcup_{i=1}^m b_i+\cM) \not\subseteq X_t.\]
Therefore, $\eta_{\cR}$ is  a $\emptyset$-definable global type (see also \cite[Lemma 2.3.8]{HHM06}). We let the reader convince her/himself that $\eta_{\cR}$ is orthogonal to $\nG$ (see \cite[Lemma 2.5.5]{HHM06}). 

Given any closed ball 
\[B(a,\gamma)=\{x\in \Aa ^1:\val(x-a)\geq \gamma \}\]
centered at $a$ with radius $\gamma \in \nGi$, by moving the center and rescaling, there is an affine transformation $f$ such that $f(\cR)=B(a,\gamma)$. We set  $\eta_{B(a,\gamma )}=f_*(\eta_{\cR})$ and we call it \emph{the generic type of the closed ball $B(a,\gamma )$}. Thus  $\eta_{B(a,\gamma )}$ is orthogonal to $\Gamma $. Furthermore, it follows from  \cite[Lemma 2.3.3]{HHM06} that we have 
\[\widehat{\Aa^1}(K)=\{\eta_{B(a,\gamma)}: a\in K, \gamma\in \nGi  \},\]
where $B(a,\infty),$ the closed ball with radius $\infty $ centered at $a,$ is identified with $a$. 
\end{expl} 

As a set, $\widehat{\Pp ^1}$ consists of the disjoint union of $\widehat{\Aa^1}$ and the definable type concentrating on the point at infinity in $\mathbb{P}^1$. The description of $\widehat{\Aa ^n}$ for $n>1$ is more complicated (see \cite[Example 3.2.3]{HrLo}).\\


For the rest of this section, whenever there is no risk of confusion, $K$ will be omitted in $\hat{X}(K)$. In this respect, results below stated for $\widehat{X}$ hold for $\widehat{X}(K)$ for every model $K$ of ACVF over which all the objects of the statement are defined.

%
%

Let $V$ be an algebraic variety over $K$ and let  $\cO_V$ be the sheaf of regular functions on $V$.  Let $\mathcal{O}_V^{\val}$ be the sheaf of $\nGi$-valued functions defined by  
\[
\mathcal{O}_V^{\val}(U)=\{\val \circ f\mid f\in \mathcal{O}_V(U)\}. 
\]
The {\it topology on the stable completion} $\widehat{V}$ is the topology having as a basis finite unions of finite intersections of sets of the following form
\[
\{p\in \widehat{U}\mid f_*(p)\in I\}, 
\]
where $U$ is a Zariski open set of $V$, $f\in \cO_V^{\val}(U)$ and $I$ is an open  interval on $\nGi$. $\widehat{V\times \nGin}$ is  equipped  with the quotient topology induced by the map 
\[\widehat{\id\times \val }\colon \widehat{V\times \Aa^n}\to \widehat{V\times \nGin}\] 
and, for a definable subset $X\subseteq V\times \nGim$, we put on $\widehat{X}$ the induced topology from $\widehat{V\times \nGim}$. \\

By \cite[Lemma 3.5.3]{HrLo} we have:

\begin{fact}\label{fact top on hat}
If $Y\subseteq \nGim$ is a definable subset, then $\widehat{Y}=Y$. The topology on $\widehat{\nGi}= \nGi$ coincides with the order topology on $\nGi$ and the topology on $\widehat{\nGin}= \nGin$ is the product topology. 

If $X$ is a definable subset of an algebraic variety $V,$ then there is a canonical bijection from $\widehat{X}\times Y$ to $\widehat{X\times Y}$ such that the topology on $\widehat{X\times Y}=\widehat{X}\times Y$ is the product topology. \qed
\end{fact}
 
By \cite[Proposition 4.2.21 and Corollary 4.2.22]{HrLo} the topology on the stable completion is related to the $\vgs$-site by:

\begin{fact}\label{fact top on hat and vgs}
Let $V$ be an algebraic variety over $K$ and let $W$ be a definable subset of $V\times \nGim$. Then $W$ is $\vgs$-open (resp. $\vgs$-closed) if and only if $\widehat{W}$ is open (resp. closed) in $\widehat{V}\times \nGim$. Moreover, a basis for the topology on the stable completion $\widehat{X}$ of a definable subset of $V\times \nGim$ is given by 
\[
\pushQED{\qed} 
\{\widehat{U}: U\in \op(X_{\vgs})\}.
\qedhere
\popQED
\] 
\end{fact}

As we pointed out already, in general, the stable completion topology is not necessarily locally compact:

\begin{nrmk}[The stable completion topology and local compactness]\label{nrmk hat top not loc compact}

The topological space $\widehat{X}(F)$ heavily depends on $F$. Even when $F$ is an algebraically closed valued field, $\widehat{X}(F)$ might fail to be locally compact. This is why, in order to do useful cohomology theory, we will replace the above topology on the stable completions by a suitable site. 

For example, consider the affine line over $\CC_p$ (the completion of the algebraic closure of $\QQ _p$). The space $\widehat{\Aa ^1}(\CC_p)$ is \emph{not} locally compact. Indeed, suppose for a contradiction there is a compact neighbourhood $U$ of $\eta _{\cR}$, the generic type of the valuation ring (the Gauss point). Then, there are rational numbers $\gamma _1\leqslant 0 <\gamma _2\in \QQ$ and $a\in \CC_p$ such that 
\[
J\coloneqq \{\eta _B \in \widehat{\mathbb{A}^1}(\mathbb{C}_p) \mid B(0,\gamma _1)  \subseteq B \subseteq  B(a,\gamma _2) \},
\]
is contained in $U$. Here $\eta _B$ denotes the generic type of the closed ball $B$. It is not difficult to see that $J$ is closed, hence also compact. Given that the valuation $\widehat{\val }\colon  \widehat{\Aa^1}(\CC_p) \to \nGi(\CC_p)$ is continuous, $\widehat{\val } (J)=[\gamma _1,\gamma _2]$ will be a compact subset of $\nGi(\CC_p) = \QQ _\infty$, a contradiction. Similar examples can be given for models of ACVF of rank higher than 1, as in such fields no closed infinite interval of the value group is compact.
\end{nrmk}

Due to the failure of local compactness of the stable completion's topology, in order to develop a cohomology theory in the category of stable completions of definable sets, we will replace the stable completion's topology by a site:
 
 \begin{defn}[$\hvgs$-site on stable completions]\label{defn vgs site on hat}
 If $X\subseteq V\times \nGim$ is a definable subset, the {\it $\hvgs$-site on $\widehat{X}$},  denoted $\widehat{X}_{\hvgs},$ is the category $\op(\widehat{X}_{\hvgs})$ whose objects are of the form $\widehat{W}$ with $W\in \op(X_{\vgs}),$ the morphisms are the inclusions and the admissible covers $\cov(\widehat{U})$ of $\widehat{U}\in \op (\widehat{X}_{\hvgs})$ are covers by objects of $\op (\widehat{X}_{\hvgs})$ with finite subcovers. 
 \end{defn}

Of course the $\vgs$-site $X_{\vgs}$ on a definable set $X$ and the $\hvgs$-site $\widehat{X}_{\hvgs}$ on its stable completion are isomorphic categories so we will often move from on site to the other whenever convenient. 


\subsection{The~\texorpdfstring{$\hvgs$}{h}-site and~\texorpdfstring{$\T$}{T}-spaces}\label{subsection  vgs-sites are t-top}
Here we show that the stable completion of a definable subset of $V\times \nGin$ equipped with the $\hvgs$-site is a  $\T$-space. For that we use the following special case of Hrushovski and Loeser's main theorem:

\begin{fact}\cite[Theorem 11.1.1]{HrLo}\label{fact main HrLo}
Let $V$ be a quasi-projective variety over $K$ and let $X$ be a definable subset of $V\times \nGin$. Then there exists a (continuous) pro-definable deformation retraction $H\colon I\times \widehat{X}\to \widehat{X}$ with image ${\mathfrak X}$ an iso-definable subset definably homeomorphic to a definable subset of some $\nGik$. Furthermore, given finitely many definable subsets $X_1, \ldots, X_n$ of $X$ the pro-definable deformation retraction $H$ can be constructed 
preserving each of the definable subsets i.e. the restriction $H_{|}:I\times \widehat{X_i}\to \widehat{X_i}$ is still a continuous pro-definable retraction.\qed
\end{fact}

The $I$ in the theorem is a {\it generalized interval}  i.e. a one-dimensional $\bGi$-definable space obtained by considering  finitely many oriented closed sub-interval of $\nGi$ and gluing them together two by two end-to-end respecting the orientations. The set ${\mathfrak X}$ is called a {\it skeleton}  of $\widehat{X}$ and is often  identified, under the definable homeomorphism, with a definable subset of $\nGik$.

Recall also (\cite[Definition 2.2.2]{HrLo}) that if $X$ is a pro-definable set, then  $Z$ is a {\it pro-definable subset} of $X$ if and only there are  cofiltering systems of definable sets $(X_i)_{i\in I}$ and  $(Z_i)_{i\in I}$ representing $X$ and $Z$ respectively such that  for each $i,$  $Z_i\subseteq X_i$ and for all $i\geq i'$ the transition maps $Z_i\to Z_{i'}$ are the restrictions of the transitions maps $X_i\to X_{i'};$ we say that  $Z$ is an {\it iso-definable subset} of $X$ if furthermore there is $i_0$ such that restriction maps $Z_i\to Z_{i'}$ are bijections for all $i\geq i'\geq i_0,$ or equivalently, the projection maps $Z\to Z_i$ are bijections for all $i\geq i_0$.

\

Let $X$ be a pro-definable set. A pro-definable subset $Y$ of $X$ is said to be \emph{relatively definable} if for some cofiltering system $(X_i)_{i\in I}$ of definable sets representing $X,$   with transitions $\pi _{i,i'}\colon X_i\to X_{i'}$ for all $i\geq i'$ in $I,$  there is $i_0\in I$ and a definable subset $Y_{i_0}\subseteq X_{i_0}$ such that $(Y_i)_{i\geq i_0}$ where  $Y_i=\pi _{i,i_0}^{-1}(Y_{i_0}),$  represents $Y,$ or equivalently, $Y=\pi_{i_0}^{-1}(Y_{i_0})$ where $\pi_{i_0}\colon X\to X_{i_0}$ denotes the natural (pro-definable) projection map represented by the transition maps $\pi _{i,i_0}\colon X_i\to X_{i_0}$ for all $i\geq  i_0$. \\

The main examples of relatively definable subsets are obtained when we have $W\subseteq X\subseteq V\times \nGim$ are definable sets, since by construction of the pro-definable structure on the stable completions, $\widehat{W}$ is a relatively definable subset of $\widehat{X}$. The following gives another kind of examples that we require below:

\begin{nrmk}[Simple points]\label{nrmk simple pts}
Given $X\subseteq V\times \nGim$ a definable subset there is a canonical embedding $\iota_X\colon X\to \widehat{X},$ taking a point $a$ to the global type $\tp (a/\UU)$ concentrating on $a$. The points of the image are called the {\it simple} points of $\widehat{X}$ and the set $\iota_X(X)$ will be denoted by $\widehat{X}_s$. Note that $\iota_X\colon X \to \widehat{X}_s$ is a pro-definable bijection. By working on affine charts and using the fact that the induced topology from $\widehat{V\times\Gamma_\infty^m}=\widehat{V}\times\Gamma_\infty^m$ on $V\times\Gamma_\infty^m$ via $\iota_{V\times\Gamma_\infty^m}^{-1}$ coincides with the product topology between the valuation topology on $V$ and the (product) order topology on $\Gamma_\infty^m,$ we see that \cite[Lemma 3.6.1]{HrLo} can be extended also to definable subsets of $V\times \nGim:$
\begin{itemize}
\item[$\bullet $]
the set $\widehat{X}_s$ of simple points of $\widehat{X}$ is a dense, iso-definable and relatively definable subset of $\widehat{X}$.
\end{itemize}
\end{nrmk}

\begin{lem}\label{lem cap with rel-iso-def} 
Let $X$ be a pro-definable set, $Z$ be a strict pro-definable subset of $X$ and $W$ be an iso-definable subset of $X$. Then:  
\begin{enumerate}
\item[(a)] If $Y$ is a relatively definable subset of $X$, then $Y\cap Z$ is strict pro-definable. 
\item[(b)] If $Y$ is an iso-definable and relatively definable subset of $X$, then $Y\cap Z$ is iso-definable. 
\item[(c)] If $Y$ is a relatively definable subset of $X$, then $Y\cap W$ is iso-definable.
\end{enumerate}
\end{lem} 

\begin{proof} 
Suppose that  $(X_i)_{i\in I}$ is a cofiltering system of definable sets representing $X,$  with transitions $\pi _{i,i'}\colon X_i\to X_{i'}$ for all $i\geq i'$ in $I$. Let $i_0\in I$ be such that $\pi_{i_0}^{-1}(Y_{i_0})=Y$ for $Y_{i_0}$ a definable subset of $X_{i_0}$, so that if for all $i\geq i_0$,  $Y_i\coloneqq \pi_{i,i_0}^{-1}(Y_{i_0}),$  then  $Y$ is represented by $(Y_i)_{i\geq i_0}$. Without loss of generality we may also suppose that if for all $i\geq i_0,$ $Z_i\coloneqq \pi_i(Z),$ then $Z$ is represented by $(Z_i)_{i\geq i_0}$ and, if for all $\geq i_0,$ $W_i=\pi _i(Z)$ then $W$ is represented by $(W_i)_{i\geq i_0}$.  

(a) We show that for all $i\geq i_0,$ $\pi_i(Y\cap Z)=Y_i\cap Z_i$. The inclusion from left-to-right is immediate. For the converse, let $a\in Y_i\cap Z_i$. By assumption, there is $x\in Z$ such that $\pi_i(x)=a$. Since $Y$ is relatively definable, $\pi_i^{-1}(Y_i)=Y$, so $x\in Y$ and therefore $a\in \pi_i(Y\cap Z)$.

(b) 
It suffices to show that $Y\cap Z=Y\cap \pi_{i_0}^{-1}(Z_{i_0})$, since the right-hand side set $\pi _{i_0}^{-1}(Y_{i_0}\cap Z_{i_0}),$ being in pro-definable bijection under $\pi _{i_0}$ with a definable set, is iso-definable  by \cite[Corollary 2.2.4]{HrLo}. The left-to-right inclusion is clear. For the converse, suppose $a\in Y$ and $\pi_{i_0}(a)\in Z_{i_0}$. By surjectivity of $\pi_{i_0}\colon Z\to Z_{i_0},$  there is $b\in Z$ such that $\pi_{i_0}(b)=\pi_{i_0}(a)\in Y_{i_0}$. It follows that $b\in \pi _{i_0}^{-1}(Y_{i_0})=Y$. Therefore, since $\pi_{i_0}\colon Y\to Y_{i_0}$ is a bijection, $a=b$ and $a\in Y\cap Z$.

(c) We have $\pi _{i_0}^{-1}(Y_{i_0}\cap W_{i_0})=\pi _{i_0}^{-1}(Y_{i_0})\cap \pi _{i_0}^{-1}(W_{i_0})=Y\cap W$ since $\pi _{i}:W\to W_i$ is a bijection for all $i\geq i_0$.
\end{proof}

The following lemma is easy and left to the reader. 

\begin{lem}\label{lem images and pullbacks}
Let $X$ and $Y$ be pro-definable sets and $f\colon X\to Y$ be a pro-definable map. 
If $U$ is a relatively definable subset of $Y,$ then $f^{-1}(U)$ is a relatively definable subset of $X$.\qed
\end{lem}

There is a notion of definable connectedness in the stable completion introduced in the beginning of \cite[Section 10.4]{HrLo}. Note that it is introduced there only for stable completions of definable subsets of an algebraic variety $V$ over a valued field but this notion can be extended to  stable completions of definable subsets of $V\times \nGim$ in the following way:

\begin{defn}[Definable connectedness in the stable completion]\label{defn def connected in hat}
Let $V$ be an algebraic variety over $K$. A  strict pro-definable subset $Z$ of $\widehat{V}$ is {\it definably connected} if and only if the only strict pro-definable clopen subsets of $Z$ are $\emptyset $ and $Z$. 

If $X\subseteq V\times \nGim$ is a definable subset, we say that $\widehat{X}$ is {\it definable connected} if and only if the pullback of $\widehat{X}$ under $\widehat{{\rm id}\times \val }: \widehat{V\times \Aa ^m}\to \widehat{V\times \nGim}$ is definably connected.
\end{defn}

We have the following characterization of definable connectedness (slightly extending the observation at the beginning of \cite[Section 10.4]{HrLo} for the case of definable subsets of $V$). 

\begin{lem}\label{lem vgs and def conn}
Let $V$ be an algebraic variety over a $K$ and let $X$ be a definable subset of $V\times \nGim$. Then $\widehat{X}$ is definably connected if and only if the only $\vgs$-clopen subsets of $X$ are $\emptyset$ and $X$. In fact, the definably connected components of $\widehat{X}$ are of the form $\widehat{U}$ for some $\vgs$-clopen subset $U$ of $X$ such that $\widehat{U}$ is definably connected.
\end{lem}

\begin{proof} 
The left-to-right implication follows directly from Fact \ref{fact top on hat and vgs}. For the converse, suppose that the only $\vgs$-clopen subsets of $X$ are $\emptyset$ and $X$ and assume for a contradiction that there is a proper strict pro-definable clopen subset $U$ of $\widehat{X}$. By Remark \ref{nrmk simple pts} and Lemma \ref{lem cap with rel-iso-def}(b), the set $U\cap \widehat{X}_s$ is an iso-definable subset of $\widehat{X}$. Furthermore, by the density of $\widehat{X}_s$ in $\widehat{X}$ and the fact that $U$ is clopen, we have ${\rm cl}(U\cap \widehat{X}_s)=U\cap \widehat{X}=U$, where ${\rm cl}$ denotes the closure in $\widehat{X}$. Abusing of notation, let $U\cap X$ denote $\iota_X^{-1}(U\cap \widehat{X}_s)$. Since $\widehat{X}_s$ is iso-definable, $U\cap X$ is a definable subset of $X$. 

\begin{clm}
$\widehat{U\cap X}=U$.
\end{clm}
\begin{proof} 
First note that $(\widehat{U\cap X})_s= U\cap \widehat{X}_s$. Indeed, 
\begin{align*}
p\in (\widehat{U\cap X)}_s 		& \Leftrightarrow  \exists a\in U\cap X \text{ such that } p=\tp(a/\UU)\\
							 	& \Leftrightarrow  \exists a\in \iota_X^{-1}(U) \text{ such that } p=\tp(a/\UU)\\
							 	& \Leftrightarrow  p\in U\cap \widehat{X}_s. 
\end{align*}
By the density of simple points this yields that 
\[
\widehat{U\cap X}\subseteq {\rm cl}(\widehat{U\cap X})={\rm cl}(\widehat{U\cap X})_s)={\rm cl}(U\cap \widehat{X}_s)=U,
\]
which shows that $\widehat{U\cap X}\subseteq U$. For the converse inclusion, consider $W\coloneqq \widehat{X}\setminus U$. The set $W$ is clopen and $W\cap \widehat{X}_s$ is also iso-definable since $W\cap \widehat{X_s} = \widehat{X}_s\setminus U\cap \widehat{X}_s$ is in pro-definable bijection with the definable set $X\setminus U\cap X$. Let $W\cap X$ denote $\iota_X^{-1}(W\cap \widehat{X}_s)$. Then, the sets $W\cap X$ and $U\cap X$ form a definable partition of $X$. By the same argument above, $\widehat{W\cap X}\subseteq W$. This shows the converse inclusion: if there were $x\in U\setminus \widehat{U\cap X}$, since $\widehat{X}=\widehat{U\cap X}\cup \widehat{W\cap X}$, we must have that $x\in \widehat{W\cap X}$, which implies $x\in W$, a contradiction. \end{proof} 

The claim, together with Fact \ref{fact top on hat and vgs}, contradict the assumption on $X$.   
\end{proof}


A first consequence of Fact \ref{fact main HrLo} is the following:

\begin{lem}\label{lem vgs components}
Let $V$ be an algebraic variety over $K$ and let $X \subseteq V\times \nGin$ be a definable set. Then $\widehat{X}$ has finitely many definably connected components. 
\end{lem}

\begin{proof}
Suppose that $V$ is quasi-projective. Then, by Fact \ref{fact main HrLo},  let $H\colon I\times \widehat{X}\to \widehat{X}$ be a continuous pro-definable deformation retraction with image an iso-definable subset ${\mathfrak X}$ of $\widehat{X}$ and let $h\colon {\mathfrak X}\to {\mathcal X}$ be a pro-definable homeomorphism of ${\mathfrak X}$ with a definable subset ${\mathcal  X}$ of $\nGik$. Let also $i_I, e_I\in I$ be the initial and the last endpoint  of $I$. Let $f\colon \widehat{X}\to {\mathfrak X}$ and $g\colon \widehat{X}\to {\mathcal X}$ be the continuous pro-definable maps given by $f(x)=H(e_I, x)$ and $g=h\circ f$.

Since definable subsets of $\nGik$ have finitely many definably connected components, let ${\mathcal X}_1, \ldots , {\mathcal X}_k$ be the finitely many definably connected components of ${\mathcal X}$. Let ${\mathfrak X}_1, \ldots , {\mathfrak X}_k$ be the corresponding (under $h$) finitely many definably connected components of ${\mathfrak X}$ and  let $\cU$ be the set of definably connected components of $\widehat{X}$.

Now if $U\in \cU$ is a definably connected component of $\widehat{X},$ then there is $i\in \{1, \ldots,  k\}$ such that $g(U)\subseteq {\mathcal X}_{i}$. Indeed, if $i,i'$ are distinct such that $g(U)\cap {\mathcal X}_i\neq \emptyset$ and $g(U)\cap  {\mathcal X}_{i'}\neq \emptyset ,$ then since $g^{-1}({\mathcal X}_i)$ and $g^{-1}({\mathcal X}_{i'})$ are relatively definable (Lemma \ref{lem images and pullbacks}) by Lemma \ref{lem cap with rel-iso-def} (a) and continuity of $g,$ both $g^{-1}({\mathcal X}_i)\cap U$ and $g^{-1}({\mathcal X}_{i'})\cap U$  would be clopen strict pro-definable subsets of $U$ distinct from $\emptyset $ and $U$.

Fix $i$ and let $\cU _i=\{U\in \cU: f(U)\subseteq {\mathfrak X}_i\}$. For each $U\in \cU_i$ choose $x_U\in U$ and consider the definable path $\gamma _U\colon I\to \widehat{X}$ given by  $\gamma _U(t)=H(t, x_U)$. Let $x'_U=\gamma (e_I)\in f(U)\subseteq {\mathfrak X}_{i}$.  Since cells in $\nGk$ are definably path connected (\cite[Chapter 6, Proposition 3.2]{vdd}) and cells in $\nGik$ are definably homeomorphic to cells in $\nGk,$ by cell decomposition, any definable subset of $\nGik$ can be partitioned into finitely many definably path connected subsets. So ${\mathfrak X}_i$ can be partitioned into finitely many ${\mathfrak X}_{i,1}\ldots , {\mathfrak X}_{i, l}$ definably path connected subsets. 

Let $j_U$ be such that $x'_U\in {\mathfrak X}_{i, j_U}$. We claim that the map $\cU_i\to \{1,\ldots , l\}\colon U\mapsto j_U$ is injective and so $\cU_i$ is finite. Hence $\cU$ is finite as well. Indeed, suppose that $U,V\in \cU$ are distinct and $j_U=j_V=j$. Then $x'_U, x'_V\in {\mathfrak X}_{i, j}$ and so there is a definable path in ${\mathfrak X}_{i,j}$ from $x'_U$ to $x'_V$. It follows that there is a definable path $\delta \colon J\to \widehat{X}$ from $x_U\in U$ to $x_V\in V$. But by Lemma \ref{lem vgs and def conn}  both $U$ and $V$ are the stable completions of $\vgs$-clopen subsets of $X$ and so they are relatively definable subsets of $\widehat{X}$. It follows by Lemma \ref{lem images and pullbacks} and continuity of $\delta ,$ that both $\delta ^{-1}(U)$ and $\delta ^{-1}(V)$  would be clopen (relatively) definable subsets of $J$ distinct from $\emptyset $ and $J,$ contradicting the definable connectedness of $J$.

If $V$ is not quasi-projective, consider an open immersion $V\rightarrow W$ where $W$ is a complete variety and $V$ is Zariski dense. By Chow's lemma, there is an epimorphism $f\colon W'\to W$ where  $W'$ is a projective variety. Consider the quasi-projective variety $V'=f^{-1}(V)$. Now if  $X'=(f,\id)^{-1}(X) \subseteq V'\times \nGin,$ then $\widehat{X'}$ has finitely many definably connected components. Since $(f,\id)$ pulls back  $\vgs$-clopen subset to $\vgs$-clopen subsets, by Lemma  \ref{lem vgs and def conn}, we see that $\widehat{X}$ has finitely many definably connected components.
\end{proof}

\begin{prop}\label{prop vgs t-space}
Let $V$ be an algebraic variety over $K$  and let $X \subseteq V\times \nGin$ be a definable set. Then the $\hvgs$-site $\widehat{X}_{\hvgs}$ is a $\T$-space. 
\end{prop} 

\begin{proof}
As we already saw that the $\hvgs$-open sets of  $\widehat{X}$  forms a basis for the topology and that they are closed under finite unions and intersections, which shows conditions (i) and (ii) in Definition \ref{def:T-topology}. Since clopen $\widehat{\vgs}$-subsets are exactly the $\widehat{\vgs}$-clopen subsets, condition (iii) follows from Lemmas \ref{lem vgs and def conn} and \ref{lem vgs components} 
\end{proof}

Note that the previous argument does not hold for $X_{\vgs}$ because there are more clopen $\vgs$-subsets than $\vgs$-clopen subsets (for example, the maximal ideal is a clopen $\vgs$-subset but not a $\vgs$-clopen). 

\subsection{Definable compactness and~\texorpdfstring{$\hvgs$}{h}-normality}\label{subsection  vgs-comp and vgs-normal}
In this subsection we recall  the notion of definable compactness in the stable completion, we show that the notion of weakly $\vgs$-normality (Definition \ref{defn wtnormal}) corresponds to $\hvgs$-normality (Definition \ref{defn tnormal}) and we prove that every $\vgs$-locally closed  subset of $V\times \nGin$ is the union of finitely many $\vgs$-open subsets which are weakly  $\vgs$-normal.\\

Since the notion of a definably type on a definable set can be extended naturally to pro-definable sets, definable compactness of pro-definable topological spaces is defined replacing curves by definable types, see \cite[Section 4.1]{HrLo}. Recall also that if  $V$ is an algebraic variety over  a valued field, then we say that a subset $W$ of $V$ is {\it bounded} if there exists an affine cover $V=\bigcup _{i=1}^mU_i,$ and subsets $W_i\subseteq U_i$ such that $W=\bigcup _{i=1}^mW_i$ and each $W_i$ is contained in a closed $m$-ball with radius in $\nG;$ we say that a subset of $\nGim$ is {\it bounded} if it is contained in $[a, \infty ]^m$ for some $a$. More generally, we say that a subset $W$ of $V\times \nGim$ is {\it bounded} if its pull back under ${\rm id}\times \val \colon V\times \Aa ^m\to V\times \nGim$ is bounded. 

A  pro-definable topological space $Z$ is {\it definably compact} if and only if every definable type $p$ on $Z$ has a limit in $Z,$ i.e. there is a point $a\in Z$ such that for every open definable subset $U$ of $Z,$ the definable type $p$ concentrates on $U$. 

\begin{nrmk}[Definable compactness in the stable completion]\label{nrmk def compact in hat}
By \cite[Corollary 4.2.22]{HrLo}, if $V$ is an algebraic variety over  a valued field $F$ and $X$ is a definable subset of $V\times \nGim,$ then:

\begin{itemize}
\item
$\widehat{X}$ is definably compact if and only if $X$ is {\it bounded} and $\vgs$-closed. 
\end{itemize}
Furthermore (\cite[Proposition 4.2.30]{HrLo}) an algebraic variety over  a valued field is complete if and only if its stable completion is definably compact. 
\end{nrmk}






The following  is a special case of \cite[Lemma 4.2.23]{HrLo}: 

\begin{nrmk}\label{nrmk vgs closed maps} 
Let $V, V'$ be varieties over $K$, $X\subseteq V\times \nGim$ and $Y\subseteq V'\times \nGin$ be definable subsets and let $f\colon X\to Y$ be a definable map which is a morphism of $\vgs$-sites. If $X$ is bounded and $\vgs$-closed, then $f$ is {\it $\vgs$-closed} i.e. $f$ maps $\vgs$-closed subsets of $X$ to $\vgs$-closed subsets of $Y$. 


Indeed, if $Z\subseteq X$ is a $\vgs$-closed subset, then $\widehat{Z}\subseteq \widehat{X}$ is a closed relatively definable subset and $\widehat{f}(\widehat{Z})\subseteq \widehat{Y}$ is a closed subset (\cite[Lemma 4.2.23]{HrLo}). On the other hand, applying \cite[Lemma 4.2.6]{HrLo} to the surjective definable map $f\colon X\to f(X),$ we get that the map $\widehat{f}\colon \widehat{X}\to \widehat{f(X)}$ is also surjective, which implies that  $\widehat{f(Z)}=\widehat{f}(\widehat{Z})$ and so $f(Z)$ is $\vgs$-closed (Fact \ref{fact top on hat and vgs}). 
\end{nrmk} 

Recall the notion of weakly $\T$-normal from Definition \ref{defn wtnormal}. Then:

\begin{nrmk}\label{nrmk t+}
Since open (resp. closed) $\widehat{\vgs}$-subsets of $\widehat{X}$ are exactly the $\widehat{\vgs}$-open (resp. $\widehat{\vgs}$-closed) subsets of $\widehat{X}$ we have that $X$ is weakly $\vgs$-normal if and only if  $\widehat{X}$ is $\hvgs$-normal (as in Definition \ref{defn tnormal}).
\end{nrmk}

For the next result we need to  recall the schematic distance function from \cite[Section 3.12]{HrLo}. First, given  $g(x_0,\ldots,x_m)$ a homogeneous polynomial with coefficients in the valuation ring $\cR$ we define the function $\val(g)\colon \Pp^m\to [0,\infty]$ by 
\[
\val(g)([x_0 \hspace{3pt}{:}\hspace{3pt} \cdots \hspace{3pt}{:}\hspace{3pt} x_m])=\val(g(x_0/x_i,\ldots, x_m/x_i))
\]
for some (any) $i$ such that $\val(x_i)=\min\{\val(x_j): j=0,\ldots, m\}$. 

Let $V$ be a projective variety over an algebraically closed valued field $K$ and $Z$ be a closed subvariety. Fix some embedding $\iota\colon  V\to \Pp^m$ and let $f=(f_1,\ldots, f_r)$ be   tuple of homogeneous polynomials in $\cR[X_0,\ldots, X_m]$ such that $Z$ is defined by the zero locus of $f$ in $V$ (note that by rescaling one can always assume that the polynomials have coefficients in $\cR$). The \emph{schematic distance function to $Z$} is the function $\varphi_{\iota,f}\colon V\to [0,\infty]$ defined by
\[
\varphi_{\iota,f}(x) = \min\{\val(f_i)(x):1\leq i\leq r\}. 
\]
Note that the function $\varphi_{\iota,f}$ is $K$-definable, it is a morphism of $\vgs$-sites and, in addition, $\varphi_{\iota,f}^{-1}(\infty)=Z$.

\begin{thm}\label{thm quasi-normal} 
Let $V$ be a  variety over $K$ and let $U$ be a basic $\vgs$-open subset of $V$. Then $U$ is weakly $\vgs$-normal. 
\end{thm}

\begin{proof}
First we consider the case where $V$ is a quasi-projective variety, then using Chow's lemma we treat the case $V$ is a complete variety and the general case will follow by Nagata's theorem. 

So suppose that $V$ is a quasi-projective variety. Let $V'$ be a projective variety over $K$ such that $V$ is an open subset of $V'$. By fixing an embedding of $\iota \colon V'\to \Pp^m,$ we may suppose without loss of generality that $V'\subseteq \Pp^m$. Denote by $Z$ the closed subvariety of $V'$ such that $V=V'\setminus Z,$ and let $f=(f_1,\ldots, f_r)$ denote a finite tuple of homogeneous polynomials in $\cR [X_0, \ldots , X_m]$ defining $Z$ in $V'$. 

By Definition \ref{def vg site} and Fact \ref{fact:vg_closed}, let $L$ be a finite index set such that $U$ is defined by 
\[
U=\bigcap_{l\in L} \{x\in V: \val(g_l(x))<\val(h_l(x))\}, 
\]
where $g_l, h_l$ are homogeneous polynomials on $V'$ of the same degree for each $l\in L$. By multiplying by a suitable constant, we may suppose for every $l\in L$ that $g_l, h_l$ have coefficients in $\cR$ and their composition under the valuation take values on $[0,\infty]$. 

Let $C', D'$ be disjoint $\vgs$-closed subsets of $U$ and $C,D$ be $\vgs$-closed subsets of $V'$ such that $C\cap U=C'$ and $D\cap U=D'$. Again by Fact \ref{fact:vg_closed}, we may suppose  that for some finite index sets $I$ and $J$ we have 
\[
C=\bigcup _{i\in I}\bigcap _{j\in J}\{x\in V':p_{ij}(x)=0,\,\,\val(q_{ij}(x))\leq \val(r_{ij}(x))\},
\] 
where $p_{ij}, q_{ij}$ and $r_{ij}$ are homogeneous polynomial on $V'$ with the $q_{ij}$'s and $r_{ij}$'s  of the same degree. As before, we may assume they all such polynomials have coefficients in $\cR$ and their composition under the valuation take values on $[0,\infty]$.

Let $k =1+2|L|+3|I\times J|$ and consider the definable map  $d\colon V'\to [0,\infty]^k$ given by 

{\small
\[
d(x)=(\varphi_{\iota,f}(x), (\val(g_{l}(x), \val(h_l(x)))_{l\in L},(\val(p_{ij}(x)), \val(r_{ij}(x)), \val(q_{ij}(x)))_{i\in I, j\in J})  
\]
}
where $\varphi_{\iota, f}\colon  V'\to [0,\infty]$ is the schematic distance to $Z$ as defined above. (Recall that $\varphi_{\iota,f}(x)=\infty$ if and only if $x \in Z$).
 
The function $d$ is a morphism of $\vgs$-sites as every coordinate function has this property. It is also $\vgs$-closed by Remarks \ref{nrmk def compact in hat} and \ref {nrmk vgs closed maps}.

Consider the set
\[
A\coloneqq \{(\alpha_l,\gamma_l)_{l\in L} \in ([0,\infty)\times[0,\infty])^{|L|}: \bigwedge_{l\in L} \alpha_l<\gamma_l\}, 
\]
and let $Y\coloneqq [0,\infty)\times A\times [0,\infty]^{3|I\times J|}$. Note that by the choice of $d$ we have $d(U)\subseteq Y$. 

\begin{clm}\label{clm:vgnormal} 
The sets $d(C')$ and $d(D')$ are closed in $Y$ and, in addition, $d(C')\cap d(D')=\emptyset$. 
\end{clm} 

\begin{proof}
Since $d$ is a $\vgs$-closed morphism of $\vgs$-sites, $d(C)$ and $d(D)$ are closed. Moreover, by the choice of $d$
\begin{equation}
d(C')=d(C)\cap d(U)=d(C)\cap Y \text{ and } d(D')=d(D)\cap d(U)=d(D)\cap Y.
\end{equation}
This show $d(C')$ and $d(D')$ are closed in $Y$. For the second part, suppose there were $x\in C'$ and $y\in D'$ such that $d(x)=d(y)$. This implies $y\in C'$, contradicting that $C'\cap D'=\emptyset$. This completes the claim. \end{proof}

To conclude the case $V$ is quasi-projective it suffices to show that $Y$ is definably normal. Indeed, the definable normality of $Y$ and Claim \ref{clm:vgnormal} imply there are disjoint open definable sets $U_1, U_2$ of $Y$ such that $d(C')\subseteq U_1$ and $d(D')\subseteq U_2$. We obtain that $C'\subseteq d^{-1}(U_1)$, $D'\subseteq d^{-1}(U_2)$ and $d^{-1}(U_1)\cap d^{-1}(U_2)=\emptyset$. In addition both $d^{-1}(U_1)$ and $d^{-1}(U_2)$ are $\vgs$-open since $d$ is a morphism of $\vgs$-sites, which shows $U$ is weakly $\vgs$-normal. 

To show the definable normality of $Y,$ first note that the definable map  $h\colon A\to A'$ given by  
$h((\alpha_l,\gamma_l)_{l\in L})=(\alpha_l, \gamma_l-\alpha_l)_{l\in L},$ where $A'\coloneqq ([0,\infty)\times[0,\infty])^{|L|}$ is a definable homeomorphism. Therefore, $Y$ is definably normal if and only if the set 
\[
Y'\coloneqq [0,\infty)\times A'\times [0,\infty]^{3|I\times J|}
\]
is definably normal, and $Y'$ is definably normal by \cite[Theorem 2.20]{ep4}). \\

Now suppose that $V$ is a complete variety. By Chow's lemma, there is a projective variety $V'$ over $K$ and a surjective morphism $g\colon V'\to V$. Consider the quasi-projective variety $Y\coloneqq g^{-1}(V)$ and the basic $\vgs$-open subset $W\coloneqq g^{-1}(U)$ of $Y$. 

Let $U_1'$ and $U_2'$ be $\vgs$-open subsets of $U$ such that $U_1'\cup U_2'=U$. By (3) of  Definition \ref{defn wtnormal} it suffices to find $D_1\subseteq U_1'$ and $D_2\subseteq U_2'$ $\vgs$-closed subsets of $U$ such that  $D_1\cup D_2=U$. By definition of being $\vgs$-open in $U,$ there are $U_1$ and $U_2$  $\vgs$-open subsets of $V$ such that $U_1'=U_1\cap U$ and $U_2'=U_2\cap U$. Since $g$ is a morphism of $\vgs$-sites, $g^{-1}(U_1)$ and $g^{-1}(U_2)$ are $\vgs$-open. For $i=1,2$, set $W_i=g^{-1}(U_i')$. Note that $W_i=g^{-1}(U_i)\cap W$ for $i=1,2$. In particular, $W_1\cup W_2=W$ and $W_i$ is $\vgs$-open in $W$ for $i=1,2$. By the quasi-projective case and (3) of  Definition \ref{defn wtnormal}, there are $\vgs$-closed subsets $C_1'\subseteq W_1$ and $C_2'\subseteq W_2$ of $W$ such that $C_1'\cup C_2'=W$. By definition this implies there are $\vgs$-closed subsets $C_i\subseteq Y$ such that $C_i'=C_i\cap W$. 

By Remark \ref{nrmk vgs closed maps}, the map $g$ is $\vgs$-closed. Then $g(C_1), g(C_2)$ are $\vgs$-closed subsets of $V$. We let the reader check that the sets $D_i\coloneqq g(C_i)\cap U$ defined for $i=1,2$ are $\vgs$-closed subsets of $U$ which satisfy in addition that $D_i\subseteq U_i'$ and $D_1\cup D_2=U$. \\

For the general case, by Nagata's theorem, there is an open immersion $f\colon V\to V'$ where $V'$ is a complete variety. Therefore, $V$ is homeomorphic to the open subset $f(V)$ of $V'$ and $f\colon V\to f(V)$ is an isomorphism of $\vgs$-sites.  In particular, since $f(U)$ is a basic $\vgs$-open subset of $V'$ which, by the case for complete varieties, is weakly $\vgs$-normal, the result follows. 
\end{proof}

We now need to extend Theorem \ref{thm quasi-normal} to basic $\vgs$-open subset of $V\times \nGin$. In the order to do that we require a couple of preliminaries.\\

Recall that if $\VV$ is an elementary extension of $\UU$ and $p\in S_z(B)$ is a type over a subset $B$ of $\UU,$ then $p$ {\it is realized in $\VV$} if and only if there is $a\in \VV^{|z|}$ such that $p={\rm tp}(a/B)$
where 
\[{\rm tp}(a/B)=\{\psi (z): \psi (z)\,\,\,\textrm{an $\cL_{\cG}$-formula over $B$ and  $\psi (a)$ holds in $\VV$}\}\] 
is the {\it type of $a$ over $B$}. One says in this case that {\it $a$ is a realization of $p$.} Recall also that any type $p\in S_z(B)$ is realized in some elementary extension of $\UU$. 

Given definable global types $p\in S_x(\UU)$ and $q\in S_y(\UU),$ the tensor product $p\otimes q$ is the definable global type ${\rm tp}(a,b/\UU)$ where $a$ realizes $p$ and $b$ realizes $q|\UU\cup \{a\}$. Recall that, for any $B\supseteq \UU$, one denotes by $q|B$ {\it the canonical extension} of $q$ to $S_y(B)$ given by $\phi(y,b)\in q|B$ if and only if $d_q(\phi)(b)$ holds in some elementary extension $\VV$ of $\UU$ containing $B$.

Note that $p\otimes q$ is indeed a definable global type with $d_{p\otimes q}(\theta )$ given by $d_p(d_q(\theta ))$. On the other hand, by \cite[Proposition 2.9.1]{HrLo} and \cite[Proposition 3.2]{hp1}, if the definable global types $p\in S_x(\UU)$ and $q\in S_y(\UU)$ are orthogonal to $\Gamma ,$ then $p\otimes q$ is orthogonal to $\Gamma $. 

\begin{lem}\label{lem:val_basis} 
Let $(\gamma_1,\ldots,\gamma_n)\in \Gamma ^n$ and $b=(b_1,\ldots,b_n)$ be a realization of $\eta_{\gamma_1}\otimes \cdots\otimes \eta_{\gamma_n}$, where $\eta _{\gamma _i}\coloneqq \eta_{B(0,\gamma_i)}$ is as defined in Example \ref{exa:gen_ball}. Then for every polynomial $\sum a_jy^j\in K[y]$ with $y=(y_1,\ldots,y_n)$ (in multi-index notation), it holds that
\[
\val(\sum_j a_jb^j) = \min_j\{\val(a_jb^j)\}.
\]  
\end{lem} 

\begin{proof} 
Let $c_i\in K$ be such that $\val(c_i)=\gamma_i$. In particular, $\val(c_i^{-1}b_i)=0$ for all $i\in\{1,\ldots,n\}$. The definition of the tensor product of definable global types yields that the set
\[
A=\{{\rm Res}(b_1,c_1),\ldots,{\rm Res}(b_n,c_n)\}
\]
is algebraically independent over the residue field $k$. It follows that every finite subset of distinct finite products of elements in $A$ is $k$-linearly independent. Let $\sum_j e_j (c_1^{-1}b_1\cdots c_n^{-1}b_n)^j$ be a $K$-linear combination of distinct finite products of elements in $\{c_1^{-1}b_1,\cdots, c_n^{-1}b_n\}$, where $e_j\in K$ and we use the multi-index notation 
\[
(c_1^{-1}b_1\cdots c_n^{-1}b_n)^j=(c_1^{-1}b_1)^{j_1}\cdots (c_n^{-1}b_n)^{j_n}, 
\] 
with $j=(j_1,\ldots,j_n)\in \NN^n$. By \cite[Lemma 3.2.2]{ep-val}, every such $K$-linear combination satisfies
\begin{equation}\label{val_equality}
\val(\sum_j e_j (c_1^{-1}b_1\cdots c_n^{-1}b_n)^j) = \min_j\{\val(e_j (c_1^{-1}b_1\cdots c_n^{-1}b_n)^j)\}.
\end{equation}
\noindent To conclude, given a polynomial $\sum a_jy^j$ in $K[y]$, consider the polynomial $\sum_j a_j' y^j$ where $a_j'= a_j(c_1^{j_1}\cdots c_n^{j_n})$. By \eqref{val_equality}, 
\begin{align*}
\val(\sum a_jb^j)	& =\val(\sum_j a_j'(c_1^{-1}b_1\cdots c_n^{-1}b_n)^j) = \min_j\{\val(a_j'(c_1^{-1}b_1\cdots c_n^{-1}b_n)^j)\}\\
					&	= \min_j\{\val(a_jb^j)\}. 
\end{align*}

\end{proof}

A minor modification of \cite[Lemma 3.5.2]{HrLo} shows the following:

\begin{lem}\label{lem:section} 
Let $V$ be an algebraic variety over $K$. The map 
\begin{align*}
s\colon & \widehat{V}\times \nGin \to \widehat{V\times \Aa ^n} \\    
& (p, \gamma_1,\ldots,\gamma_n) \mapsto p\otimes \eta_{\gamma_1}\otimes \cdots\otimes \eta_{\gamma_n}
\end{align*}
where $\eta _{\gamma _i}\coloneqq \eta_{B(0,\gamma_i)}$ is as in Example \ref{exa:gen_ball}, is an  injective morphism of $\hvgs$-sites which is a section of   
\[
\widehat{\id_V\times \val}\colon \widehat{V\times \Aa ^n} \to \widehat{V\times \nGin}\cong \widehat{V}\times \nGin.
\] 
\end{lem}

\begin{proof} 
That $s$ is  a section of   $\widehat{\id_V\times \val}$ and is injective is straightforward. Let us show it is a morphism of $\hvgs$-sites. Since taking pre-images preserves Boolean combinations, it is enough to consider the case $V=\Aa ^m$ and show that if $U$ is a $\vgs$-open subset of $\Aa ^{m+n}$ of the form 
\[
U=\{(x,y)\in \Aa^{m+n}\mid \val (f_1(x,y))<\val (f_2(x,y))\}
\] 
with $f_1,f_2\in K[x,y]$, then we have $s^{-1}(\widehat{U})=\widehat{W}$ for a $\vgs$-open subset $W$ of $\Aa^m\times \nGin$. 

Suppose that $f_l(x,y)=\sum_j h_{l,j}(x)y^{j}$ with $y=(y_1, \ldots, y_n)$ (in multi-index notation). Let $p\otimes \eta_{\gamma_1}\otimes \cdots\otimes \eta_{\gamma_n}\in s(\widehat{V}\times \Gamma_\infty^n)\cap \widehat{U}$ and let $(a,b)$ be a realization of it, $b=(b_1,\ldots, b_n)$ and $\val (b_i)=\gamma _i$. By Lemma \ref{lem:val_basis} (applied to the field $K(a)$) we have
\begin{align*}
\val(f_l(a,b)) & =\min_{j} \val(h_{l,j}(a)b^{j})\\
                     & =P_l( (\val (h_{l,j}(a)))_j,\gamma_1,\ldots,\gamma_n)
\end{align*}
where the function $P_l\colon \Gamma_\infty^{d_l+m}\to \Gamma_\infty$ is obtained by composition of the natural continuous extensions of $\min$ and $+$. Here $d_l$ is the number of terms in the variable $y$ in the polynomial $f_l$.

It follows that, if $(p,\gamma)\in \widehat{V}\times \nGin $ with $\gamma =(\gamma _1, \ldots, \gamma _n),$ then  we have  
\begin{align*}
(p,\gamma )\in s^{-1}(\widehat{U})	& \Leftrightarrow p\otimes \eta_{\gamma_1}\otimes \cdots\otimes \eta_{\gamma_n} \in \widehat{U}\\         
& \Leftrightarrow \val (f_1(a,b))<\val (f_2(a,b))\,\,\,\textrm{for all}\,\,\, (a,b)\models p\otimes \eta_{\gamma_1}\otimes \cdots\otimes \eta_{\gamma_n} \\
					& \Leftrightarrow P_1((\val (h_{1,j}(a)))_j,\gamma)<P_2(\val ((h_{2,j}(a)))_j,\gamma)\,\,\,\textrm{for all}\,\,\, a\models p.
\end{align*}
Therefore, setting  
\[W=\{(a,\gamma)\in \Aa ^m\times \nGin : P_1((\val (h_{1,j}(a)))_j,\gamma)<P_2((\val (h_{2,j}(a)))_j,\gamma)\},\] 
$W$ is a definable  subset and  $s^{-1}(\widehat{U})=\widehat{W}$. On the other hand, since the pullback of $W$ under $\id \times \val $ is 
\begin{align*}
\{(a,b)\in V\times \Aa^n\mid P_1(\val ((h_{1,j}(a)))_j,\val(b))<P_2((\val (h_{2,j}(a)))_j,\val (b))\},
\end{align*}
and $\min$ and $+$ are continuous, it follows that $(\id \times \val )^{-1}(W)$ is $\vgs$-open subset and so, by definition,  $W$ is $\vgs$-open as required.
\end{proof}

\begin{lem}\label{lem:t-section} 
Let $f\colon (X,\T) \to (Y, \T')$ be a morphism in $\fT$ with a section $s\colon (Y,\T')\to (X,\T)$ which is also a morphism in $\fT$.  If $U$ is a $\T$-open subset  of $X$  which is $\T$-normal, then $s^{-1}(U)$ is a $\T'$-open subset of $Y$ which is  $\T'$-normal. In particular, if every $\T$-open subset of $X$ is a finite union of $\T$-open which are $\T$-normal subsets, then every  $\T'$-open subset of $Y$ is a finite union of $\T'$-open which are  $\T'$-normal subsets.
\end{lem} 

\begin{proof} 
Let $C_1,C_2$ be two disjoint closed $\T'$-subsets of $s^{-1}(U)$. For $i=1,2$, let $D_i\coloneqq f^{-1}(C_i)\cap U$. Since $f$ is a morphism in $\fT,$ both $D_1$ and $D_2$ are closed $\T$-sets. Hence, there are disjoint open $\T$-subsets $U_1, U_2\subseteq U$ such that $D_i\subseteq U_i$ for $i=1,2$. Since $s$ is a morphism in $\fT,$ the sets $W_1\coloneqq s^{-1}(U_1)$ and $W_2\coloneqq s^{-1}(U_1)$ are disjoint open $\T'$-subsets of $s^{-1}(U)$ and $C_i\subseteq W_i$ for $i=1,2$ since $f\circ s=\id $. 

For the last part, just note that if $f^{-1}(W)=W'_1\cup \ldots \cup W'_l,$ then $W=s^{-1}(W'_1)\cup \ldots \cup s^{-1}(W'_l)$.
\end{proof}



\begin{cor}\label{cor:mix-normal} 
Let $V$ be a variety over $K$. Then every  $\vgs$-locally closed subset $X$ of $V\times \nGin$ is the union of finitely many basic $\vgs$-open subsets of $X$ which are weakly $\vgs$-normal. In fact, every basic $\vgs$-open subset of $V\times \nGin$ is weakly $\vgs$-normal.
\end{cor}

\begin{proof}
It suffices to show the result for $\vgs$-open subset as every $\vgs$-closed subset of a weakly $\vgs$-normal set is again weakly $\vgs$-normal.  By the isomorphism of sites $(V\times \Aa^n)_{\vgs}$ and $(\widehat{V\times \Aa ^n})_{\vgs}$ and Theorem \ref{thm quasi-normal} and Remark \ref{nrmk t+} the result holds for $\hvgs$-open subsets of $\widehat{V\times \Aa ^n}$. By Lemmas \ref{lem:section}  and \ref{lem:t-section} the result holds for $\hvgs$-open subset of $\widehat{V\times \nGin}$. So by the isomorphism of sites $(V\times \nGin)_{\vgs}$ and $(\widehat{V\times \nGin})_{\vgs}$ and Remark \ref{nrmk t+}  the result holds in $V\times \nGin$. This prove actually shows that a basic $\vgs$-open subset of $V\times \nGin$ is weakly $\vgs$-normal.
\end{proof}

\section{Cohomology, finiteness, invariance and vanishing}\label{section cohomo finiteness and invariance}

In this section we deduce the Eilenberg-Steenrod axioms for the cohomology on the sites introduced in previous sections on definable subsets in $\bGi,$ on definable subsets in ACVF and on their stable completions. We also  show finiteness, invariance and vanishing results for these cohomologies. 

\subsection{The Eilenberg-Steenrod axioms}\label{subsection homotopy ax in bGi}
The main work in order to show the Eilenberg-Steenrod axioms consists in showing the homotopy axiom. To achieve this goal we will use the Vietoris-Begle theorem (Theorem \ref{thm vietoris-b}) applied to the tilde 
\[\tilde{\pi }\colon \widetilde{X\times [a,b]}\into \tilde{X}\] of the projection map. In particular we need to verify the assumptions of the Vietoris-Begle theorem in this context. In $\bGi$ these assumption have been verified in the literature. For definable sets $X$ and $\widehat{X}$ in ACVF, some work is required.



We start with the following lemma whose proof is exactly the same as that of  it's o-minimal version given in  \cite[Proposition 2.33]{ep4}:

\begin{lem}\label{lem obj tA}
Let $(X,\T)$ be an object of $\fT$. Suppose that for every $\alpha \in \tilde{X}$ any chain of specializations of $\alpha $ has finite length. If every $\T$-open subset of $X$ is the union of finitely many $\T$-open subsets which are $\T$-normal, then for every $\alpha \in \tilde{X},$ there is an open, normal  and constructible subset $U$ of $\tilde{X}$ such that $\alpha \in U$ and $\alpha $ is closed in $U$.\qed
\end{lem}

In the o-minimal context, the assumption of finiteness of chains of specializations is verified in \cite[Lemma 2.11]{ejp}. We will now prove an analogue result in ACVF. Let us first make some general observations in this context.  

\begin{nrmk}\label{rem:types_pairs} 
Let $V$ be an affine varierty over  $K$, say $V=\mathrm{Spec}(A)$ for $A=K[T_1,\ldots,T_n]/J$, where $J$ is an ideal of $K[T_1,\ldots,T_n]$. Note that by quantifier elimination in ACVF, the set $\widetilde{V}$ is in bijection with the set of pairs $(I,v)$ where $I\in V$ and $v$ is a valuation on the fraction field $\mathrm{Frac}(A/I)$ which extends $\val$. Formally, the bijection sends $p\mapsto (\supp(p),v_p)$ where $\supp(p)\coloneqq \{f\in A: f(x)=0 \in p\}$ and $v_p$ is determined by setting $v_p(f/g)\geq 0$ if and only if the formula $\val(f(x))\geq \val(g(x))$ belongs to $p$ (where $x=(x_1,\ldots, x_n)$). We will denote the fraction field  $\mathrm{Frac}(A/I)$ by $F_p$ and $\cR_p$ denote the valuation ring of $F_p$ with respect to $v_p$.   
\end{nrmk}

As in Definition \ref{def:tilde}, recall that $\widetilde{V}$ is equipped with the topology generated by $\widetilde{U}$ where $U$ is $\vgs$-open in $V$. Given $p,q\in \widetilde{V}$, we write $p\prec q$ to state that $p$ is a specialization of $q$, that is, that $p$ belongs to the closure of $\{q\}$. Note that if $p\prec q$, then $\supp(q)\subseteq \supp(p)$ as the set $\{ x\in V: \val (f(x))=\infty\}$ is $\vgs$-closed. In particular, we may assume that $F_p\subseteq F_q$. 

\begin{lem}\label{lem spec length in acvf}
Let $V$ be a varierty over $K$. Consider $\widetilde{V}$ with the topology generated by $\widetilde{U}$ where $U$ is $\vgs$-open in $V$. Then any chain of specializations in $\widetilde{V}$ is bounded by $\dim(V)$.
\end{lem}
\begin{proof}
Since specialization is local, it suffices to work with the case when $V$ is affine. For affine $V$, let $A=\mathcal{O}_V(V)$ be the regular functions on $V$. Suppose that $p\prec q$ for $p,q\in \widetilde{V}$. As observed in the previous paragraph we may suppose $F_q\subseteq F_p$. In addition, note that $\cR_q\subseteq \cR_p$. Indeed, if there was $f\in A$ such that $\val(f(x))\geq 0\in q$ but $\val(f(x))<0\in p$, since the later formula defines a $\vgs$-open, we must have that $\val(f(x))<0\in q$, a contradiction. Hence, every chain $p_1\prec \cdots\prec p_n\prec p$ of specializations of $p\in \widetilde{V}$, induces a chain of valued subfields $(F_p,\cR_{p})\subseteq (F_{p_n},\cR_{p_n})\subseteq \cdots\subseteq (F_{p_1},\cR_{p_1})$. By \cite[Corollary 3.4.6]{ep-val}, the length of every such a chain is bounded by the dimension of $V$.   
\end{proof}

\begin{cor}\label{cor spec length}
Let $V$ be a variety over $K$ and let $X\subseteq V\times \nGim$ be a definable subset. Then any chain of specializations in $\widetilde{X}$ has finite length bounded by $\dim(V)+m$. The same holds for $\tilde{\widehat{X}}$.
\end{cor}

\begin{proof}
Consider the commutative diagram

\begin{equation*}
\xymatrix{
\widetilde{V\times \Aa ^m}  \ar[r]^{} \ar[d]^{}  & \widetilde{V\times \nGim} \ar[d]^{} \\
\widetilde{\widehat{V\times \Aa ^m}} \ar[r]^{}  & \widetilde{\widehat{V\times \nGim}}
}
\end{equation*}
where the top arrow is $\widetilde{\id \times \val}$, the bottom arrow is $\widetilde{\widehat{\id \times \val}}$ and the vertical arrows are homeomorphisms induced by isomorphisms of sites (Remark \ref{nrmk site iso tilde}).

If there were a chain of specializations in $\widetilde{V\times \nGim}$ of length $>\dim(V)+m$, then there would be  a chain of specializations in $\widetilde{\widehat{V\times \nGim}}$ of length $>\dim(V)+m.$ Since the bottom arrow has a section $\tilde{s}: \widetilde{\widehat{V\times \nGim}} \to \widetilde{\widehat{V\times \Aa ^m}}$  which is the tilde of an injective morphism of $\hvgs$-sites (Lemma \ref{lem:section}), we would also have a chain of specializations of length $>\dim(V)+m$ in $\widetilde{V\times \Aa ^m}$ which contradicts Lemma \ref{lem spec length in acvf}.
\end{proof}

\begin{nrmk}\label{nrmk spec chains and dim} One could improve the previous result by bounding the length of the specialization chain by an appropriate notion of dimension for definable subsets of $K^n\times \Gamma_\infty^m$. We would like to point out that a good candidate to play such a role was introduced by F. Martin in his PhD thesis \cite{phd_martin}, where he more generally defined a dimension function for subsets of $K^n\times\Gamma_\infty^m\times k_K^r$. 
\end{nrmk}

In the o-minimal context we know from Fact \ref{fact CxN is normal}, that if $W\subseteq \nGim$ is definably normal and $[a,b]\subseteq \nGi$ is a closed interval, then $W\times [a,b]$ is also definably normal. In the ACVF context we have:

\begin{lem}\label{lem UxI is normal}
Let $V$ be a variety over $K$. If $W$ is a basic $\vgs$-open subset of $V\times \nGim$ and $[a, b]\subseteq \nGi$ is a closed interval, then $W\times [a,b]$ is weakly $\vgs$-normal.
\end{lem}

\begin{proof}
We have that the pullback of $W\times [a,b]$ under $\id \times \val \colon V\times \Aa ^{m+1}\to V\times \nG_{\infty}^{m+1}$ is the intersection of the pullback of $W\times \nGi$ and the pullback of $V\times\nGim\times [a,b]$. The first pullback is a basic $\vgs$-open subset of $V\times \Aa ^{m+1},$ and so it is weakly $\vgs$-normal by Theorem \ref{thm quasi-normal}.  Since the second pullback is a $\vgs$-closed subset $V\times \Aa ^{m+1}$ it follows that $(\id \times \val)^{-1}(W\times [a,b])$ is weakly $\vgs$-normal (being a $\vgs$-closed subset of a weakly $\vgs$-normal set). From Lemmas \ref{lem:section}  and \ref{lem:t-section}  and Remark \ref{nrmk t+}  it follows that $W\times [a,b]$ is weakly $\vgs$-normal.
\end{proof}

In the o-minimal context, since $[a,b]$ is definably compact, the definable map $[a,b]\to {\rm pt}$ is definably proper (\cite[Definition 3.10 and Remark 3.11]{emp}), and since $\bSg$ has definable Skolem functions, the map $[a,b]\to  {\rm pt}$ is proper in $\Df$ (\cite[Definition 3.3 and Theorem 3.15]{emp}). Hence if $X$ is a definable subset of $\nGin,$ then  $\pi \colon X\times [a,b]\into X$ maps closed definable subsets to closed definable subsets. In the ACVF context we have:

\begin{lem}\label{lem proj is closed}
Let $V$ be a variety over $K$. Let $X$ be a definable subset of $V\times \nGin$ and  $[a,b]\subseteq \nGi$ be a closed interval. Then the projection $\pi \colon X\times [a, b]\to X$ is $\vgs$-closed i.e. maps $\vgs$-closed  subsets into $\vgs$-closed subsets. 
\end{lem}

\begin{proof}
Let $C\subseteq X\times [a,b]$ be a $\vgs$-closed subset. Then $\widehat{C}$ is a closed relatively pro-definable subset of $\widehat{X\times [a,b]}$. It is enough to show that $\widehat{\pi (C)}$ is closed. If it is not closed, then by \cite[Proposition 4.2.13]{HrLo}, there is a definable type $\mathfrak{p}$ on $\widehat{X}$ concentrating on $\widehat{\pi (C)}$ with limit $p$ in the closure of $\widehat{\pi (C)}$ in $\widehat{X}$ such that $p\notin \widehat{\pi (C)}$. Applying twice \cite[Lemma 4.2.6 (2)]{HrLo} to the surjective definable map $\pi \colon C\to \pi (C)$, we have that there is a definable type $\mathfrak{q}$ on $\widehat{X\times [a,b]}$ concentrating on $\widehat{C}$ such that $\widehat{\pi}_\ast(\mathfrak{q})=\mathfrak{p}$. Since $\widehat{X\times [a,b]}\cong \widehat{X}\times [a,b]$, the topology is the product topology, and definable types on $[a,b]$ have limits, $\mathfrak{q}$ must have a limit $q$ in $\widehat{X\times [a,b]}$ (namely, the tensor product of the limits of the corresponding projections). By \cite[Lemma 4.2.4]{HrLo}, $q\in \widehat{C}$ and so $p=\widehat{\pi}(q)\in \widehat{\pi (C)}$.
\end{proof}

\begin{thm}[Homotopy axiom in stable completions]\label{thm homotopy ax stable}
Let $V$ be a variety over $K$. Let $X\subseteq V\times \nGim$ be a $\vgs$-locally closed subset  and let $\cF\in \mod (A_{\widehat{X}_{\hvgs}})$. Let $\Psi$ be a family of $\hvgs$-supports on $\widehat{X}$. Let $[a,b]\subseteq [0,\infty ]$ be a closed interval, $\pi \colon X\times [a,b]\to X$ be the projection, and  for $d\in [a,b]$ let 
\[i_d\colon X\into X\times [a,b]\]
be the continuous definable map given by $i_d(x)=(x,d)$ for all $x\in X$. Then
\[ \widehat{i_a}^*=\widehat{i_b}^*:H^n_{\Psi \times [a,b]}(\widehat{X}\times [a,b];\widehat{\pi} ^{-1}\cF)\into H^n_{\Psi }(\widehat{X};\cF) \]
for all $n\geq 0$.
\end{thm}

\begin{proof}
The homotopy axiom will follow once we show that the projection map
$\pi \colon X\times [a,b]\into X$ induces an isomorphism
\[\widehat{\pi }^*\colon H^n_{\Psi }(\widehat{X};\cF)\into H^n_{\Psi \times [a,b]}(\widehat{X}\times [a,b];\widehat{\pi }^{-1}\cF)\]
since by functoriality we obtain
\[\widehat{i_a}^*=\widehat{i_b}^*=(\widehat{\pi }^*)^{-1}\colon H^n_{\Psi \times [a,b]}(\widehat{X}\times [a,b];\widehat{\pi }^{-1}\cF) \into H^n_{\Psi }(\widehat{X};\cF)\]
for all $n\geq 0$.  Equivalently we
need to show that
\[\tilde{\widehat{\pi }}^*\colon H^n_{\tilde{\Psi }}(\tilde{\widehat{X}};{\tilde \cF})\into H^n_{\tilde{\Psi \times [a,b]}}(\tilde{\widehat{X}\times [a,b]};\tilde{\widehat{\pi }}^{-1}{\tilde \cF})\]
is an isomorphism. For
this we need to verify the hypothesis of the Vietoris-Begle theorem (Theorem \ref{thm vietoris-b}). 

By Lemma \ref{lem proj is closed} (and the fact that $\widehat{\pi(Z)}=\widehat{\pi }(\widehat{Z})$ (\cite[Lemma 4.2.6]{HrLo})), $\widehat{\pi }$ maps closed subsets to closed subsets. It follows that $\tilde{\widehat{\pi }}$ maps closed constructible subset to closed (constructible) subsets. 

Let $\alpha \in \tilde{\widehat{X}}$. Given Corollaries \ref{cor spec length} and \ref{cor:mix-normal} together with Remark \ref{nrmk t+}, it follows from Lemma \ref{lem obj tA} that  there is an open, normal constructible subset $U$ of $\tilde{\widehat{X}}$ such that $\alpha \in U$ and $\alpha $ is closed in $U$. Furthermore we may assume that $U$ is of the form $\tilde{\widehat{W}}$ where $W$ is a basic $\vgs$-open subset of $X$.  From Lemma \ref{lem UxI is normal} it follows that $W\times [a,b]$ is weakly $\vgs$-normal, and hence, by Remark \ref{nrmk t+}, $\widehat{W\times [a,b]}\simeq \widehat{W}\times [a,b]$ is $\hvgs$-normal. Since $\tilde{\widehat{\pi }}^{-1}(\tilde{\widehat{W}})=\tilde{\widehat{\pi }^{-1}(\widehat{W})}=\tilde{\widehat{W}\times [a,b]},$  this set is normal and constructible by Proposition \ref{prop main normal def normal}. 

Since $\widehat{\pi ^{-1}(V)}=\widehat{\pi }^{-1}(\widehat{V})$, we have a commutative diagram 
\begin{equation*}
\xymatrix{
\widehat{X}\times [a,b]  \ar[r]^{\hat{\pi}} \ar[d]^{}  & \widehat{X} \ar[d]^{} \\
X\times [a,b] \ar[r]^{\pi}  & X
}
\end{equation*}
of morphisms of sites which induces a commutative diagram 
\begin{equation*}
\xymatrix{
\tilde{\widehat{X}\times [a,b]}  \ar[r]^{\tilde{\widehat{\pi}}} \ar[d]^{}  & \tilde{\widehat{X}} \ar[d]^{} \\
\tilde{X\times [a,b]} \ar[r]^{\tilde{\pi}}  & \tilde{X}
}
\end{equation*}
where the vertical arrows are homeomorphisms induced by isomorphisms of sites (Remark \ref{nrmk site iso tilde}). It follows that $\tilde{\widehat{\pi}}^{-1}(\alpha )$ is homeomorphic to $\tilde{\pi }^{-1}(\beta )$ for some $\beta \in \tilde{X}$.
On the other hand, as in the proof of  \cite[Claim 4.5]{ejp}, and since the value group $\bGi$ is stably embedded, $\tilde{\pi }^{-1}(\beta )$ is homeomorphic to $\tilde{[a,b](\bSg ')}$ where $\bSg '$ is the definable closure of $\nSg \cup \{b\}$ with $b$ an element realizing the type $\beta ,$ and therefore, $\tilde{\pi }^{-1}({\beta   })$ is connected and $H^q(\tilde{\pi }^{-1}({\beta  }); \tilde{\pi }^{-1}\tilde{\cF}_{|\tilde{\pi }^{-1}({\beta   })})=0$ for all $q>0$. So we conclude that $\tilde{\widehat{\pi}}^{-1}(\alpha )$ is connected and acyclic as required.

\end{proof}

\begin{cor}[Homotopy axiom in ACVF]\label{cor homotopy ax acvf}
Let $V$ be a variety over $K$. Let $X\subseteq V\times \nGim$ be a $\vgs$-locally closed subset  and let $\cF\in \mod (A_{X_{\vgs}})$. Let $\Psi $ a family of $\vgs$-supports on $X$. Let $[a,b]\subseteq [0,\infty ]$ be a closed interval, $\pi \colon X\times [a,b]\to X$ be the projection, and  for $d\in [a,b]$ let 
\[i_d\colon X\into X\times [a,b]\]
be the continuous definable map given by $i_d(x)=(x,d)$ for all $x\in X$. Then
\[ i_a^*=i_b^*\colon H^n_{\Psi \times [a,b]}(X\times [a,b];\pi ^{-1}\cF)\into H^n_{\Psi }(X;\cF) \]
for all $n\geq 0$.
\end{cor}

\begin{proof}
As explained in the previous proof, the homotopy axiom will follow once we show that the projection map
$\pi \colon X\times [a,b]\into X$ induces an isomorphism
\[\pi ^*\colon H^n_{\Psi }(X;\cF)\into H^n_{\Psi \times [a,b]}(X\times [a,b];\pi ^{-1}\cF).\]
But this follows from the isomorphism
\[\widehat{\pi }^*\colon H^n_{\Psi }(\widehat{X};\cF)\into H^n_{\Psi \times [a,b]}(\widehat{X}\times [a,b];\widehat{\pi }^{-1}\cF)\]
 proven in Theorem \ref{thm homotopy ax stable} and the commutative diagram 
\begin{equation*}
\xymatrix{
\widehat{X}\times [a,b]  \ar[r]^{\hat{\pi}} \ar[d]^{}  & \widehat{X} \ar[d]^{} \\
X\times [a,b] \ar[r]^{\pi}  & X
}
\end{equation*}
of morphisms of sites with the vertical arrows being isomorphisms of sites. 
\end{proof}

Using Theorem \ref{thm basis of open normal} and the $\bGi$ analogues of Lemmas \ref{lem spec length in acvf}, \ref{lem UxI is normal} and \ref{lem proj is closed} mentioned above, arguing as in the proof of Theorem \ref{thm homotopy ax stable} we obtain:

\begin{thm}[Homotopy axiom in $\bGi$]\label{thm homotopy ax bgi}
Let $\bG=(\nG, <, +,\ldots )$ be an arbitrary o-minimal expansion of an ordered group.  Suppose that $X\subseteq \nGim$ is a definably locally closed  set and $\cF\in \mod (A_{X_{\df}})$. Let $\Psi $ a family of definable supports on $X$. Let $[a,b]\subseteq [0,\infty ]$ be a closed interval, $\pi \colon X\times [a,b]\to X$ the projection, and  for $d\in [a,b]$ let 
\[i_d\colon X\into X\times [a,b]\]
be the continuous definable map given by $i_d(x)=(x,d)$ for all $x\in X$. Then
\[ i_a^*=i_b^*\colon H^n_{\Psi \times [a,b]}(X\times [a,b];\pi ^{-1}\cF)\into H^n_{\Psi }(X;\cF) \]
for all $n\geq 0$.
\end{thm}

\begin{thm}\label{thm:Eilenberg-Steenrod} The sheaf cohomology associated to 
\begin{enumerate}
    \item the category of definable sets in $\bGi$ with their associated o-minimal site, 
    \item the full sub-category of definable sets in $\mathrm{ACVF}$ with their associated $\vgs$-site, 
    \item the category of stable completions with their associated $\hvgs$-site, 
\end{enumerate} 
satisfies the Eilenberg-Steenrod axioms. 
\end{thm}

\begin{proof} The homotopy axiom corresponds to Theorem \ref{thm homotopy ax bgi} for (1), to Corollary \ref{cor homotopy ax acvf} for (2) and to Theorem \ref{thm homotopy ax stable} for (3). For all (1)-(3), the remaining axioms follow by standard arguments. Indeed, once we apply the tilde functor in each case, the proofs of the exactness (\cite[Chapter II, Section 12, (22)]{b}) and excision axioms (\cite[Chapter II, Section 12, 12.8 or 12.9]{b}) are  purely algebraic. The dimension axiom is also immediate. 
\end{proof}

\begin{nrmk}\label{nrmk exact triple mv}
From the Eilenberg-Steenrod axioms for cohomology one obtains as usual  the exactness for triples (\cite[Chapter II, Section 12, (24)]{b})  and the Mayer-Vietoris long exact sequence  (\cite[Chapter II, Section 13, (32)]{b}). 
\end{nrmk}

\subsection{Finiteness and invariance results in ~\texorpdfstring{$\bGi$}{G}}\label{subsection finite and inv in bGi}
Here we will show finiteness and invariance  results for o-minimal cohomology with definably compact supports of definably locally closed subsets of $\nGim$. Using Hrushovski and Loeser's main theorem (\cite[Theorem 11.1.1]{HrLo}) we will deduce in the next subsection finiteness and invariance results for cohomology in ACVF and in stable completions. \\

For the finiteness and invariance results in $\bGi$ we first show that the arguments from  \cite{bf} for  $\bG =(\nG, < , +, \ldots )$ and for definably compact sets can be extended to $\bGi$ and to definably locally closed sets with only minor modifications. We include the details in the main technical lemma (Lemma \ref{lem bf cells}) and refer the reader to \cite{bf} for the other results which simply follow from the Eilenberg-Steenrod axioms. Concerning invariance, the extension to arbitrary sheaves adapts techniques from \cite{ep3} and, we include the details as the same argument will be used in the ACVF case.\\

Below we let  $L$ be  an $A$-module, where $A$ is a commutative ring with unit. \\


In \cite[Lemma 3.2]{bf} it is shown that for every {\it bounded} cell in $\nGm$  there is a deformation retraction of $C$ onto a cell of strictly lower dimension, via a $\nG$-definable homotopy.   
 As observed in \cite{bf} the proof extends to $\bGi$.

\begin{lem}\label{lem cells contract}
If $C\subseteq [0,\infty ]^m$ is a cell of dimension $n>0,$ then there is a deformation retraction of $C$ onto a cell of strictly lower dimension. So by induction every cell in $[0,\infty ]^m$ is $\bGi$-definably contractible to a point.
\end{lem}

\begin{proof}
If $C$ is the graph of a function, we can reason by induction on $m$. So assume $C=(f,g)_B$. If $g<\infty$ and $C$ is of maximal dimension, then $C\subseteq \nGm$  and we can apply \cite[Lemma 3.2]{bf}. If $g<\infty$ and $C$ is not of maximal dimension, then one can apply induction. If $g=\infty$, let $h=f+1$. Then $f<h<g$ and the $\bGi$-definable map given $H\colon [0, \infty]\times C\to C$  by
\[H(t, (x,y))=\begin{cases}
(x, h(x)-t),  &\text{ if $y< h(x)-t$}\\
(x,y), &\text{ if $ h(x)-t \le y\le h(x)+t$}\\
(x, h(x)+t),  &\text{ if $y> h(x)+t$}
\end{cases}
\]
is a deformation retract of $C$ onto the cell $\Gamma(h)$. 
\end{proof}

As observed after the proof of \cite[Corollary 3.3]{bf} that result (for bounded cell in $\nGn$) extends to $\bGi$ due to Lemma \ref{lem cells contract} and the homotopy axiom (Theorem \ref{thm homotopy ax bgi}), recall that cells are definably locally closed (Remark \ref{nrmk local def comp}):

\begin{lem}\label{lem bf acyclic cells}
Let $C\subseteq [0,\infty ]^m$ be a cell. Then $C$ is  acyclic, i.e. $H^p(C;L_C)=0$ for $p>0$ and $H^0(C; L_C)=L$.\qed
\end{lem}

We also have the analogue of \cite[Lemma 7.1]{bf} with small modifications in the construction:

\begin{lem}\label{lem bf cells}
Let $C\subseteq [0,\infty ]^m$ be a cell of dimension $r$. There is a $\bGi$-definable family $\{C_{(t,s)} : 0<t,s<\infty \}$ of closed and bounded subsets $C_{(t,s)} \subseteq C$ such that:
\begin{enumerate}
\item
$C = \bigcup_{(t,s)}C_{(t,s)}$.
\item If $r>0$ and $0 < t' < t<\frac{s}{2}<s<s',$  then $C_{(t,s)}
\subseteq C_{(t',s')}$ and this inclusion induces an isomorphism 
\[H^{p}(C \backslash C_{(t,s)}; L_C) \simeq  H^{p}(C \backslash C_{(t',s')}; L_C).\]
\item
If $r>0$, then the o-minimal cohomology of $C \backslash C_{(t,s)}$ is given by
\begin{equation*}
H^{p}(C \backslash C_{(t,s)}; L_C) =
\begin{cases}
L ^{1+\chi _{1}(r)}\qquad \textmd{if} \qquad p\in \{0, r-1\}\\
\\
\,\,\,\,\,\,\,\,\,\,\,\,\,\,\,\,\,\,\,\,\,\,\,\,\,\,\,\,\,\,\,
\\
0\qquad \,\,\,\,\,\,\, \,\,\,\,\,\,\,\,\,\,\, \, \textmd{if} \qquad p\notin \{0, r-1\}
\end{cases}
\end{equation*}
where $\chi _1\colon {\mathbb Z}\to \{0,1\}$ is the characteristic function of the subset $\{1\}$.
\item
If $K\subseteq C$ is a definably compact subset, the there are $0<t,s$ such that $K\subseteq C_{(t,s)}$.\\
\end{enumerate}
\end{lem}

\begin{proof}
We define the definable family $\{C_{(t,s)} : 0<t, s <\infty \}$ by induction on $l\in \{1,\ldots , m-1\}$ and the definition of the $\bGi$-cell $C$ in the following way.
\begin{enumerate}
\item[$\bullet$]
If $l=1$ and $C$ is a singleton $\{d\}$ in $\nGi$, we define $C_{(t,s)}=C$. Clearly $C_{(t,s)}$ is a  closed and bounded subset.
\item[$\bullet$]
If $l= 1$ and $C = (d, e)\subseteq \nGi,$ then 
\begin{align*}
C_{(t,s)} =
\begin{cases}
[d + \gamma (t,s), e -\gamma (t,s)]\quad \quad \quad \quad \quad \quad   \textrm{if}\quad e< \infty\\
\,\, \\
[d+\gamma (t,s), (d+s)-\gamma(t,s)] \quad \quad \quad  \,\,\, \textrm{otherwise}
\end{cases}
 \end{align*}
 where 
 \begin{align*}
 \gamma (t,s) = 
 \begin{cases}
 \min\{\frac{e-d}{2}, t\} \quad \quad \quad  \textrm{if}\quad e< \infty\\
\,\, \\
\min \{\frac{s}{2}, t\} \quad \quad \quad \quad  \textrm{otherwise}
\end{cases}
 \end{align*}
 (in this way $C_{(t,s)}$ is non empty). Clearly $C_{(t,s)}$ is a  closed and bounded subset.
\item[$\bullet$]
If $l> 1$ and $C = \Gamma(f)$ where  $B\subseteq [0,\infty ]^{l}$ is $\bGi$-cell. By induction $B_{(t,s)}$ is defined and is a  closed and bounded subset. We put $C_{(t,s)}=\Gamma (f_{|B_{(t,s)}})$. Clearly $C_{(t,s)}$ is a  closed and bounded subset.

\item[$\bullet$]
If $l > 1$ and $C = (f, g)_{B}$ where $B\subseteq [0,\infty ]^{l}$ is $\bGi$-cell and  $f<g$. By induction $B_{(t,s)}$ is defined and is a  closed and bounded subset. Recall that either $g<\infty $ or $g=\infty $. We put 
\begin{align*}
C_{(t,s)} =
\begin{cases}
[f +\gamma (t,s), g - \gamma (t,s)]_{B_{(t,s)}}\quad \quad \quad \quad \quad \quad   \textrm{if}\quad g< \infty\\
\,\, \\
[f+\gamma (t,s), (f+s)-\gamma(t,s)]_{B_{(t,s)}} \quad \quad \quad \,\, \textrm{otherwise}
\end{cases}
 \end{align*}
 where 
 \begin{align*}
 \gamma (t,s) = 
 \begin{cases}
 \min\{\frac{g-f}{2}, t\} \quad \quad \quad  \textrm{if}\quad g< \infty\\
\,\, \\
\min \{\frac{s}{2}, t\} \quad \quad \quad \quad  \textrm{otherwise.}
\end{cases}
 \end{align*}
 By induction   $B_{(t,s)}$ is  a closed and bounded subset of $B$. Also,  for each $x\in B_{(t,s)}$, the fiber $(\pi _{|C})^{-1}(x)\cap C_{(t,s)}$ is closed and bounded. Let $(x,y)\in C$ be an element in the closure of $C_{(t,s)}$. Then $x\in B_{(t,s)}$ and $(x,y)\in (\pi _{|C})^{-1}(x)\cap C_{(t,s)}\subseteq  C_{(t,s)}$. So  $C_{(t,s)}$ is a closed and bounded subset of $C$. 

\end{enumerate}

We observe that from this construction we obtain:

\begin{clm}\label{clm bf cover cells}
For $(t,s)$ as above there is a covering ${\mathcal{U}}_{C} = \{U_i : i \in I\}$ of $C \backslash C_{(t,s)}$ by relatively open bounded subsets of $C$ such that:
\begin{enumerate}
\item
The index set $I$ is the family of the closed faces of an $r$-dimensional cube. (So $|I| = 2r$).
\item
If $E \subseteq I,$ then $U_{E} \coloneqq  \bigcap_{i \in E}U_{i}$ is either empty or a $\bGi$-cell. (So in particular $H^{p}(U_{E}; L_C) = 0$ for $p > 0$ and, if $U_{E} \neq \emptyset, H^{0}(U_{E}; L_C) = L$.)
\item
For $E \subseteq I,$ $U_{E} \neq \emptyset$ iff the faces of the cubes belonging to $E$ have a non-empty intersection.
\end{enumerate}
So the nerve of ${\mathcal U}_C$ is isomorphic to the nerve of a covering of an $r$-cube by its closed faces.
\end{clm}

\begin{proof}
To show that there is a covering satisfying the properties above, we define ${\mathcal{U}}_C$ by induction on $l\in \{1, \ldots , m-1\}$. We distinguish four cases
according to definition of the $C_{(t,s)}$.

\begin{enumerate}
\item[$\bullet$]
If $l= 1$ and $C$ is a singleton in $\nGi$, then $\mathcal{U}_{C}=\{C\}$.
\item[$\bullet$]
If $l= 1$ and $C = (d, e)\subseteq \nGi,$ then 
\begin{align*}
\mathcal{U}_C =
\begin{cases}
\{(d,d + \gamma (t,s)) ,(e -\gamma (t,s),e)\}\quad \quad \quad \quad \quad \quad   \textrm{if}\quad e< \infty\\
\,\, \\
\{(d,d+\gamma (t,s)), ((d+s)-\gamma(t,s),\infty)\} \quad \quad \quad   \textrm{otherwise}
\end{cases}
 \end{align*}
\item[$\bullet$]
If $l > 1$ and $C = \Gamma(f)$ where $B\subseteq [0,\infty ]^{l}$ is $\bGi$-cell. By definition $C_{(t,s)}=\Gamma (f_{|B_{(t,s)}})$. By induction we have a
covering $\mathcal{V}_B$ of $B \backslash B_{(t,s)}$ with the stated
properties, and we define ${\mathcal{U}}_C$ to be a covering of $C \backslash C_{(t,s)}$ induced by the natural homeomorphism between the graph of $f$ and its domain.
\item[$\bullet$]
If $l > 1$ and $C = (f, g)_{B}$ where $B\subseteq [0,\infty ]^{l}$ is $\bGi$-cell and  $f<g$. By definition 
\begin{align*}
C_{(t,s)} =
\begin{cases}
[f +\gamma (t,s), g - \gamma (t,s)]_{B_{(t,s)}}\quad \quad \quad \quad \quad \quad   \textrm{if}\quad g< \infty\\
\,\, \\
[f+\gamma (t,s), (f+s)-\gamma(t,s)]_{B_{(t,s)}} \quad \quad \quad \,\, \textrm{otherwise.}
\end{cases}
 \end{align*}
 By induction we have that $B \backslash B_{(t,s)}$ has a covering
${\mathcal V}_B = \{V_{j} : j \in J\}$ with the stated properties, where $J$ is the set of closed faces of the cube $[0, 1]^{r-1}$.
Define a covering ${\mathcal{U}}_C = \{U_i : i \in I\}$ of $C \backslash C_{(t,s)}$ as follows. As index set $I$ we take the closed faces of the cube $[0, 1]^r$. Thus $|I|=|J|+2,$ with the two extra faces  corresponding to the ``top" and ``bottom" face of $[0, 1]^r$. We
associate to the top face the open set $(g - \gamma (t,s), g)_{B_{(t,s)}}$  if $g<\infty $ or $((f+s) - \gamma (t,s), \infty )_{B_{(t,s)}}$  if $g=\infty $ and the bottom face the open set $(f, f
+ \gamma (t,s))_{B_{(t,s)}}$.  The other open sets of the covering are the preimages of the sets $V_j$ under the restriction of the projection $\Gamma ^{l+1}_{\infty } \rightarrow \nGil$. This defines a covering of $C \backslash C_{(t,s)}$ with the stated properties.
\end{enumerate}
\end{proof}

Property (1) of the lemma is clear. By (the proof of) Claim \ref{clm bf cover cells} there are open covers ${\mathcal U}'_C$ of $C \backslash C_{(t',s')}$ and ${\mathcal U}_C$ of $C \backslash C_{(t,s)}$ satisfying the assumptions of \cite[Lemma 5.5]{bf}. Hence property (2) of the lemma holds. Finally, if $r>1$, then property (3) follows from  Claim \ref{clm bf cover cells}  and \cite[Corollary 5.2]{bf}. On the other hand, if $r=1$, then $C\setminus C_{(t,s)}$ is by construction a disjoint union $D\sqcup E$ of two $\bGi$-cells in $[0,\infty ]^m$  of dimension $r=1$. Therefore, in this case, the result follows from Lemma \ref{lem bf acyclic cells}, since $H^{*}(C \backslash C_{(t,s)}; L_C) \simeq H^*(D;L_D)\oplus H^*(E;L_E)$.

It remains to show property (4). Note that if  $0< t'< t<s<s',$ then by construction $C_{(t',s')}$ contains the interior relative to $C$ of  $C_{(t,s)}$. Thus $C$ a directed union $\bigcup_{(t,s)}U_{(t,s)}$ of a $\bGi$-definable family of  relatively open $\bGi$-definable subsets. In particular, if $K\subseteq C$ is a $\bGi$-definably compact subset, then $\{K\cap U_{(t,s)} : 0<t<s\}$ is a $\bGi$-definable family of open $\bGi$-definable subsets of $K$  with the property that every finite subset of $K$ is contained in one of the $K\cap U_{(t,s)}$. Therefore, since $K$ is closed and bounded (Remark \ref{nrmk def comp}), by \cite[Corollary 2.2 (ii)]{PePi07},  there are  $0<t<s$  such that $K\subseteq U_{(t,s)}\subseteq C_{(t,s)}$.
\end{proof}

From Lemma \ref{lem bf cells} and computations using  excision axiom (\cite[Chapter II, 12.9]{b}) and long exactness sequence (\cite[Chapter II, Section 12 (22)]{b}) we obtain just like in \cite[Lemma 7.2]{bf}: 

\begin{lem}\label{lem bf x bd1}
Let $X\subseteq [0,\infty ]^m$ be a definable set  and $C\subseteq X$ be a cell of maximal dimension. Then for every $0<t'<t<\frac{s}{2}<s<s'$  we have isomorphisms induced by inclusions:
\[
\pushQED{\qed} 
H^{*}(X \backslash C_{(t,s)}; L_X) \simeq H^{*}(X \backslash C_{(t',s')}; L_X).\qedhere
\popQED
\]
\end{lem}



The next result is the analogue of \cite[Corollary 7.3]{bf}, but we have to explain how to use \cite[Lemma 6.7]{bf} in our context:

\begin{lem}\label{lem bf x bd2}
Let $X\subseteq [0,\infty ]^m$ be a closed definable set and $C\subseteq X$ be a cell of maximal dimension. Then for every $0<t<s$  we have isomorphisms induced by inclusions:
\[H^{*}(X \backslash C_{(t,s)}; L_X) \simeq H^{*}(X \backslash C; L_X).\] 
\end{lem}

\begin{proof}
Let $Z=X\setminus C$ and for $0<t<s$  let $Y_{(t,s)}=X\setminus C_{(t,s)}$.  Let $V$ be an  open definable neighborhood of $Z$ in $X$. Since $K=X\setminus V\subseteq C$ is a definably compact subset, by Lemma \ref{lem bf cells} (4), there are  $0<t<s$  such that $K\subseteq C_{(t,s)}\subseteq C$. Therefore, $Z\subseteq Y_{(t,s)}\subseteq V$. On the other hand, $\tilde{Y_{(t,s)}}$ for all $0<t<s$ are taut in $\tilde{X}$ being open subsets (\cite[page 73]{b}) and, since $X$ is definably normal (being definably compact) and $Z$ is a closed subset, the family of all closed subsets of $\tilde{X}$ is a normal and constructible family of supports on $\tilde{X}$ and so by Corollary \ref{cor coho around},  $\tilde{Z}$ is taut in $\tilde{X}$.  Therefore, by the purely topological \cite[Lemma 6.4]{bf}, together with Lemma \ref{lem bf x bd1} and \cite[Remark 6.6]{bf} we have
\[H^{*}(X \backslash C; L_X) \simeq \lind{0<t<s}H^{*}(X \backslash C_{(t,s)}; L_X)\simeq H^{*}(X \backslash C_{(t,s)}; L_X).\] 
\end{proof}

\begin{nrmk}\label{nrmk c in Gi}
Let $X\subseteq \nGim$ be a definably locally closed subset.  By Remark \ref{nrmk def completions} there is a definably compact subset $P\subseteq \nGim$ such that  $X$ is an open definable subset of $P$. Since by Remark \ref{nrmk def normal} $P$ is definably normal and every definably compact subset of $X$ is a closed definable subset of $P$, it follows from Definition \ref{defn tnormal} (2) that the family, which we will denote by  $c$, of all definably compact definable subsets of $X$ is a family of definably normal supports on $X$.
\end{nrmk}

\begin{thm}\label{thm fg coho def groups}
Let $A$ be a  noetherian ring and let $L$ be a  finitely generated  $A$-module. If $X\subseteq \nGim$ is a definably locally closed subset, then $H^p_c(X;L_X)$ is  finitely generated for each $p$. Moreover,   $H^p_c(X;L_X)=0$ for $p>\dim X$.
\end{thm}

\begin{proof}
Let $i\colon X\hookrightarrow P\subseteq [0,\infty ]^{2m}$ be  a definable completion of $X$  (Remark \ref{nrmk def completions}). Since $P$ is a closed subset of $[0,\infty ]^{2m},$ from Lemma \ref{lem bf x bd2} and computations using  Mayer-Vietoris sequence (\cite[Chapter II, Section 13 (32)]{b}, see Remark \ref{nrmk exact triple mv}) we obtain just like in \cite[Theorem 7.4]{bf} that $H^p(P;L_P)$ is  finitely generated for each $p$ and  $H^p(P;L_P)=0$ for $p>\dim P$. Since $X$ is an open definable subset of $P,$ $P\setminus X$ is a closed subset of $[0,\infty ]^{2m}$ and similarly we have  $H^p(P\setminus X;L_{P\setminus X})$ is  finitely generated for each $p$ and  $H^p(P\setminus X;L_{P\setminus X})=0$ for $p>\dim P\setminus X$. 

By Equation (\ref{loc closed rest}) on page \pageref{loc closed rest}, if $Z\subseteq P$ is a definably locally closed subset, then we have $L_Z=i_{!}(L_{P\,|Z})$. Then  Corollary \ref{cor hphi and !} together with the short exact coefficients sequence 
\[0\rightarrow L_X\stackrel{}\rightarrow L_P\stackrel{}\rightarrow L_{P\setminus X}\rightarrow 0\]
yields the long exact cohomology sequence
\[\ldots \rightarrow H^{l-1}(P\setminus X;L_{P\setminus X})\stackrel{\alpha}\rightarrow H^l_c(X;L_X)\stackrel{\beta}\rightarrow H^l(P;L_P)\rightarrow H^l(P\setminus X;L_{P\setminus X})\rightarrow  \ldots .\]
Since $P$ is the closure of $X$ we have $\dim X=\dim P$  and $\dim P\setminus X<\dim P$ (\cite[Chapter 4, (1.8)]{vdd}). It follows that $H^p_c(X;L_X)=0$ for all $p>\dim X$. 
It  follows also that $\ker \beta ={\rm Im}\,\alpha $ is finitely generated and ${\rm Im}\,\beta $ is finitely generated. Thus 
we see that $H^p_c(X;L_X)$ is finitely generated for each $p$. \end{proof}

Below we will require the following, which is obtained going to $\tfT$  and applying the corresponding criterion in the topological case (\cite[Chapter II, 16.12]{b}) (point (iii) is a little bit stronger here) observing that $\widetilde{U}$ with $U\in \op(X_{\T})$ forms a filtrant basis for the topology of $\tilde{X}_{\T}.$ 

\begin{lem}\label{lem criteria for T-sheaves}
Let  $(X,\T)$ be an object of $\fT$  and let $\mathfrak{R}$ be a class of objects of $\mod (A_{X_{\T}}).$ Suppose that $\mathfrak{R}$ satisfies:
    \begin{itemize}
    \item[(i)] for each exact sequence $\exs{\cF'}{\cF}{\cF''}$ with $\cF' \in \mathfrak{R}$ we have $\cF \in \mathfrak{R}$ if and only if $\cF'' \in \mathfrak{R}$;
    \item[(ii)] $\mathfrak{R}$ is stable under filtrant $\Lind$;
    \item[(iii)] $A_V \in \mathfrak{R}$ for any $V \in \op(X_{\T}).$ 
    \end{itemize}
Then $\mathfrak{R}=\mod(A_{X_{\T}})$. \qed 
\end{lem}


Recall that if $\bG '$ is an elementary extension of $\bG$ or an o-minimal expansion of $\bG$ and  $Y\subseteq \nGim$ is a definable subset, then we have a natural morphism of sites
\[Y(\bG '_{\infty})_{\df} \to Y_{\df}.\]
If $\cF \in \mod(A_{Y_{\df}})$, then we have a corresponding  $\cF(\bG' _{\infty})\in \mod (A_{Y(\bG'_{\infty })_{\df}})$ given by inverse image. \\

\begin{thm}\label{thm bf inv gp-int}
Let $\bG '$ be an elementary extension of $\bG$ or an o-minimal expansion of $\bG$. If $X\subseteq \nGim$ is a definably locally closed subset, then  for every $\cF\in \mod(A_{X_{\df}})$ we have
an isomorphism
\[H^*_c(X;\cF)\simeq H^*_c(X(\bG '_{\infty}); \cF(\bG '_{\infty})).\]
\end{thm}



\begin{proof}
Let $i\colon X\hookrightarrow P\subseteq [0,\infty ]^{2m}$ be  a definable completion of $X$  (Remark \ref{nrmk def completions}). Since $P$ is a closed subset of $[0,\infty ]^{2m},$ for any closed definable subset $Z\subseteq P$, from Lemma \ref{lem bf x bd2} and computations using  Mayer-Vietoris sequence (\cite[Chapter II, Section 13 (32)]{b}, see Remark \ref{nrmk exact triple mv}) we obtain just like in  \cite[Theorems 8.1 and 8.3]{bf} an isomorphism
\[H^*(Z;A_{Z})\simeq H^*(Z(\bG '_{\infty}); A_{Z(\bG '_{\infty})}).\]

Since $P$ is definably compact it is definably normal and so the family of closed definable subsets of $P$ is a family of definably normal supports on $P$. Also $X$ is an open definable subset of $P,$ hence by Corollary \ref{cor hphi and !}, if $\cF\in \mod (A_{X_{\df}})$, then
\[H^*(P;i_{!}\cF)=H_{c}^*(X; \cF).\]
By invariance of definably compactness (see Remark \ref{nrmk def comp}) and  Corollary \ref{cor hphi and !} in $\bG '_{\infty }$ we also have
\[H^*(P(\bG '_{\infty});i^{\bG '_{\infty}}_{!}\cF(\bG '_{\infty}))=H_{c}^*(X(\bG '_{\infty}); \cF(\bG '_{\infty})).\]
Moreover by Equation (\ref{loc closed base change}) on page \pageref{loc closed base change} applied to the comutative diagram (with tildes omitted) 
\begin{equation*}
\xymatrix{
X(\bG '_{\infty})  \ar@{^{(}->}[r]^{i^{\bG '_{\infty }}}  \ar[d]^{r_|} & P(\bG '_{\infty }) \ar[d]^{r} \\
X \ar@{^{(}->}[r]^{i} & P
}
\end{equation*}
we get  $(i_{!}\cF)(\bG '_{\infty })=i^{\bG '_{\infty}}_{!}\cF(\bG '_{\infty})$ and so
\[H^*(P(\bG '_{\infty});(i_{!}\cF)(\bG '_{\infty}))=H_{c}^*(X(\bG '_{\infty}); \cF(\bG '_{\infty})).\]

The result  will follow once we show that, for every $\cG\in \mod(A_{P_{\df}})$, we have
an isomorphism
\[H^*(P;\cG)\simeq H^*(P(\bG '_{\infty}); \cG(\bG '_{\infty})).\]
But this is obtained  applying Lemma \ref{lem criteria for T-sheaves} in the following way. For (i)  the exact sequence $0 \to \cF' \to \cF \to \cF'' \to 0$ implies the following chain of morphisms in cohomology ($j \in \ZZ$)
{\tiny
$$
\xymatrix{
\cdots \ar[r] & H^j(P;\cF') \ar[r] \ar[d] & H^j(P;\cF) \ar[d] \ar[r] & H^j(P;\cF'') \ar[d] \ar[r] & \cdots \\
\cdots \ar[r] & H^j(P(\bG '_{\infty}); \cF'(\bG '_{\infty})) \ar[r] & H^j(P(\bG '_{\infty}); \cF(\bG '_{\infty})) \ar[r] & H^j(P(\bG '_{\infty}); \cF''(\bG '_{\infty})) \ar[r] & \cdots
}
$$
}
\noindent
Therefore, if we have  isomorphims for  $\cF',$ then using the five lemma  we have isomormphisms for $\cF$ if and only if we have isomorphisms for $\cF'' $. 

For  (ii) first note that: (a) sections commute with filtrant $\Lind$ (\cite[Example 1.1.4 and Proposition 1.2.12]{ep2}); (b) filtrant $\Lind$ of $\T$-flabbly sheaves is $\T$-flabby (\cite[Proposition 2.3.4 (i)]{ep2}); (c) the full additive subcategory of $\T$-flabby objects is $\Gamma (U;\bullet)$-injective for every $U\in \op(P_{\T})$ (\cite[Proposition 2.3.5]{ep2}). This implies that $H^*(P;\Lind \cF_i)=\Lind H^*(P;\cF_i).$ Since the same facts holds in $\bG '_{\infty}$, we have $H^*(P(\bG '_{\infty });\Lind (\cF_i(\bG '_{\infty })))=\Lind H^*(P(\bG '_{\infty });\cF_i(\bG '_{\infty }))$ and (ii) follows since $\Lind (\cF_i(\bG '_{\infty }))=(\Lind \cF_i)(\bG '_{\infty })$ (inverse image commutes with $\Lind$).

For (iii), if $V\in \op(P_{\T})$, then   the exact sequence $0 \to A_V \to A_P \to A_{P\setminus V} \to 0$ implies the following chain of morphisms in cohomology ($j \in \ZZ$)
{\tiny
$$
\xymatrix{
\cdots \ar[r] & H^j(P;A_V) \ar[r] \ar[d] & H^j(P;A_P) \ar[d] \ar[r] & H^j(P;A_{P\setminus V}) \ar[d] \ar[r] & \cdots \\
\cdots \ar[r] & H^j(P(\bG '_{\infty}); A_{V(\bG '_{\infty})}) \ar[r] & H^j(P(\bG '_{\infty}); A_{P(\bG '_{\infty})}) \ar[r] & H^j(P(\bG '_{\infty}); A_{(P\setminus V)(\bG '_{\infty})}) \ar[r] & \cdots
}
$$
}
\noindent
where we are using that $P(\bG '_{\infty})\setminus V(\bG '_{\infty})=(P\setminus V)(\bG '_{\infty})$. Since $P\setminus V$ is a closed definable subset of $P$ we have as observed above  $H^*(P\setminus V; A_{P\setminus V})\simeq H^*((P\setminus V)(\bG '_{\infty}); A_{(P\setminus V)(\bG '_{\infty})})$. But by Corollary \ref{cor hphi and !} $H^*(P; A_{P\setminus V})\simeq H^*(P\setminus V; A_{P\setminus V})$ and by the same result in $\bG '_{\infty }$ we have $H^*(P(\bG '_{\infty}); A_{(P\setminus V)(\bG '_{\infty})})\simeq H^*((P\setminus V)(\bG '_{\infty}); A_{(P\setminus V)(\bG '_{\infty})})$. So (iii) follows now using the five lemma.
 \end{proof}

\subsection{Finiteness and invariance results in ACVF and stable completions}\label{subsection finite and inv in acvf}
Here we will use the  finiteness and invariance  results for o-minimal cohomology with definably compact supports of definably locally closed subsets of $\nGim$ and  Hrushovski and Loeser's main theorem (\cite[Theorem 11.1.1]{HrLo}) to deduce  finiteness and invariance results for cohomology in ACVF and in stable completions. \\

As before, we let $K$ be a model of ACVF. First we recall a few things. Let $V$ and $W$ be a algebraic varieties over $K$. Let $Y\subseteq V\times \nGim$ and $Z\subseteq W\times \nGin$ be definable subsets. Then:  
\begin{enumerate}
\item[(1)]
A pro-definable map $g\colon Y\to \widehat{Z}$ is {\it $\vgs$-continuous} if and only if the pullback of an open definable subset of $\widehat{Z}$  is a $\vgs$-open subset of $Y$ (see page 48 in \cite{HrLo}). Since open definable subsets of $\widehat{Z}$ are exactly the $\hvgs$-open subsets of $\widehat{Z},$ in our terminology, $g\colon Y\to \widehat{Z}$ is $\vgs$-continuous if and only if $g$ is a morphism  $g\colon Y_{\vgs} \to \widehat{Z}_{\hvgs}$ of sites. 

\item[(2)] 
A pro-definable map $g\colon Y\to \widehat{W}$ is {\it g-continuous} (resp. {\it v-continuous}) if and only if for any regular function $h\in \cO_W(W)$ the pullback under $\widehat{(\val \circ h)}\circ g $ of a g-open (resp. v-open) subset of $\nGi$ is g-open (resp. v-open) subset of $Y$, where a subset of $\nGim$ is g-open (resp. v-open) if its pullback under $\val \colon \Aa ^m\to \nGim$ is g-open (resp. v-open). It follows that if $g\colon Y\to \widehat{W}$ is both g-continuous and v-continuous, then $g$ is $\vgs$-continuous.

\item[(3)]
Any pro-definable map $g\colon Y\to \widehat{Z}$  has  a {\it canonical extension} $G\colon \widehat{Y}\to \widehat{Z}$ which is a pro-definable map (\cite[Lemma 3.8.1]{HrLo}) $G$ is given by: if $p\in \widehat{Y}(K)$ and $p|K={\rm tp}(c/K),$ then $G(p)\in \widehat{Z}(K)$ is such that $G(p)|K(c)={\rm tp}(d/K)$ with $d$ a realization of ${\rm tp}(g(c)/K(c))$. It follows that, if $g\colon Y\to \widehat{Z}$ is $\vgs$-continuous  and $U$ is a $\vgs$-open subset of $Z$ then $\widehat{g^{-1}(\widehat{U})})=G^{-1}(\widehat{U})$ and the canonical extension $G$ is a morphism of $\hvgs$-sites.
\end{enumerate}

Let $V$ be a quasi-projective variety over $K$. Let $X\subseteq V\times \nGim$ be a definable subset. By 
\cite[Theorem 11.1.1]{HrLo} (see Fact \ref{fact main HrLo}), let $H\colon I\times \widehat{X}\to \widehat{X}$ be  the continuous pro-definable deformation retraction  with image an iso-definable subset ${\mathfrak X} $ of $\widehat{X}$ definably homeomorphic to a definable subset of some $\nGik$.  Let $h\colon {\mathfrak X}\to \mathcal{X}$ be a pro-definable homeomorphism of ${\mathfrak X}$ with a definable subset of $\nGik$. 

We now make a couple of observations required below which are not explicitly stated in \cite{HrLo}.

\begin{lem}\label{lem retration is vgs}
The continuous pro-definable deformation retraction $H\colon I\times \widehat{X}\to \widehat{X}$ is a morphism of $\hvgs$-sites. 
\end{lem}

\begin{proof}
The statement of \cite[Theorem 11.1.1]{HrLo} is more general: the deformation retraction respects  finitely many definable maps $\zeta _i\colon X\to \nGi$ with canonical extension $\zeta _i\colon \widehat{X}\to \nGi$ and respects the action of a finite algebraic group acting on $V$ leaving $X$ globally invariant. For this reason one can make several reductions such as assume that: $m=0$; $X=V$ and $V$ is a projective and equidimensional variety.  Then the proof proceeds by induction on $\dim (V)$. The case $\dim (V)=0$ being trivial, for the case $\dim (V)>0$ several preliminary reductions are performed allowing to essentially reduce the proof to the case of curve fibration. In the end the homotopy  $H$ is the concatenation $H_{\Gamma }^{\alpha}\circ ((H_{\widetilde{base}}\circ H_{curves})\circ H_{inf})$ of the relative curve homotopy $H_{curves}$, the liftable base homotopy $H_{\widetilde{base}}$, the purely combinatorial tropical homotopy $H_{\nG}^{\alpha }$ and the inflation homotopy $H_{inf}$. 

Furthermore we have: 
\begin{itemize}
\item[(i)]
The inflation homotopy $H_{inf}\colon [0,\infty ]\times \widehat{V}\to \widehat{V}$ is given by
\begin{eqnarray*}
H_{inf}(t,x)=
\begin{cases}
F(t,x) & x\in \widehat{V\setminus D}\\
x & x\in \widehat{D}
\end{cases}
\end{eqnarray*}
where $D$ is a closed subvariety of $V$, $F\colon [0,\infty ]\times \widehat{V\setminus D}\to \widehat{V\setminus D}$ is the canonical extension of a $\vgs$-continuous pro-definable map $f\colon [0,\infty ]\times V\setminus D\to \widehat{V\setminus D}$.  

Let $U\subseteq V$ be a $\vgs$-open subset. Then $U\setminus D$ is a $\vgs$-open subset of $V\setminus D$. Also 
\begin{eqnarray*}
H_{inf}^{-1}(\widehat{U}\setminus \widehat{D})&=& H_{inf}^{-1}(\widehat{U\setminus D})\\
&=&F^{-1}(\widehat{U\setminus D})\\
&=& \widehat{f^{-1}(\widehat{U\setminus D})}
\end{eqnarray*}
and so $H_{inf} ^{-1}(\widehat{U}\setminus \widehat{D})$ is a $\hvgs$-open subset of $[0,\infty ]\times \widehat{V\setminus D}$. It follows that $H_{inf}^{-1}(\widehat{U})=H_{inf} ^{-1}(\widehat{U}\setminus \widehat{D})\cup ([0,\infty ]\times (\widehat{U\cap D}))$ is a $\hvgs$-subset of $[0,\infty ]\times \widehat{V}$. By \cite[Lemma 10.3.2]{HrLo} $H_{inf}$ is continuous and so $H_{inf}^{-1}(\widehat{U})$ is an open subset of $[0,\infty ]\times \widehat{V}$. Therefore, by Fact \ref{fact top on hat and vgs}, $H_{inf}^{-1}(\widehat{U})$ is $\hvgs$-open subset. Hence $H_{inf}$ is a morphism of $\hvgs$-sites.

\item[(ii)]
The relative curve homotopy $H_{curve}\colon [0,\infty ]\times \widehat{V_0\cup D_0}\to \widehat{V_0\cup D_0}$ is the canonical extension of a g-continuous and v-continuous homotopy $h_{curves}\colon [0,\infty ]\times V_0\cup D_0\to \widehat{V_0\cup D_0}$.  Hence by the remarks make above, $H_{curves}$ is a morphism of $\hvgs$-sites.

\item[(iii)] 
The base homotopy $H_{\widetilde{base}}$ is obtained from the canonical extension of a homotopy $h_{base}\colon I\times U\to \widehat{U}$ which is obtained by the induction procedure. Hence, by induction, $H_{\widetilde{base}}$ is also a morphism. By the remarks make above, $H_{curves}$ is a morphism of $\hvgs$-sites.

\item[(iv)]
The tropical homotopy $H_{\nG}^{\alpha }$ is constructed in the $\nGi$ setting where $\vgs$-continuous is the same as definable and continuous, so this homotopy is a also a morphism of $\hvgs$-sites.
\end{itemize}
We can therefore conclude that $H$ is a morphism of $\hvgs$-sites as required.
\end{proof}

\begin{lem}\label{lem homeo of retraction is vgs}
 The  pro-definable homeomorphism  $h\colon {\mathfrak X}\to \mathcal{X}$ of ${\mathfrak X}$ with a definable subset of $\nGik$ is a morphism of $\hvgs$-sites
\end{lem}

\begin{proof}
Let $U$ be a $\vgs$-open subset of $X$. Then $\widehat{U}$ is a relatively definable subset of $\widehat{X}$. So, since ${\mathfrak X}$ is an iso-definable subset,   $ \widehat{U}\cap {\mathfrak X}$ is a relatively definable subset of ${\mathfrak X}$ (Lemma \ref{lem cap with rel-iso-def} (c)). By Lemma \ref{lem images and pullbacks}, $(h^{-1})^{-1}(\widehat{U})$ is a relatively definable subset of ${\mathcal X}$, hence it is an open definable subset, so a $\hvgs$-open subset (Facts \ref{fact top on hat and vgs} and  \ref{fact top on hat}). Since $h^{-1}:{\mathcal X}\to {\mathfrak X}$ is a morphism of $\hvgs$-sites, so is $h.$
\end{proof}

\begin{lem}\label{lem P and Q for loc closed}
If $X\subseteq V\times \nGim$ is a $\vgs$-locally closed subset, then there is a bounded $\vgs$-closed subset $P\subseteq V\times \nGim$ such that $X\subseteq P$ and $X$ is $\vgs$-open in $P$.
\end{lem}

\begin{proof}
Indeed, let $V'$ be a projective variety over $K$ such that $V$ is an open subset of $V'$. Then $X$ is still a $\vgs$-locally closed  subset of $V'\times \nGim$. Consider the natural definable map $l\colon V'\times \nGim \to V'\times [0,\infty ]^{2m}$ which is $\id _{V'}$ on the first  coordinate and on the second  coordinate is given as in Remark \ref{nrmk def completions}. Then $l\colon V'\times \nGim\to l(V'\times \nGim)\subseteq V'\times [0,\infty ]^{2m}$ is a definable homeomorphism, a morphism of $\vgs$-sites and  $l(X)$ is $\vgs$-locally closed in $l(V'\times \nGim)$. So there is a $\vgs$-closed subset $Z$ of 
 $l(V'\times \nGim)$ such that $l(X)$ is $\vgs$-open in $Z$. Since $l(V'\times \nGim)$ is $\vgs$-closed in $V'\times [0,\infty ]^{2m}$, and $Z$ and $Z\setminus l(X)$  are bounded and $\vgs$-closed subset of $V'\times \nG _{\infty}^{2m}$. Let $P=l^{-1}(Z)$ and $Q=l^{-1}(Z\setminus l(X))$. Then applying \cite[Lemma 4.2.6]{HrLo} to the surjective definable maps $l^{-1}:Z\to l^{-1}(Z)$ and $l^{-1}:Z\setminus l(X)\to l^{-1}(Z\setminus l(X))$ we see that $\widehat{P}=\widehat{l^{-1}}(\widehat{Z})$ and $\widehat{Q}=\widehat{l^{-1}}(\widehat{Z\setminus l(X)})$, hence they are definably compact by Remark \ref{nrmk def compact in hat} and \cite[Proposition 4.2.9]{HrLo}. Hence,  $P$ and $Q$ are bounded $\vgs$-closed subsets  (Remark \ref{nrmk def compact in hat}), $X$ is $\vgs$-open in $P$ and $Q=P\setminus X$. 
\end{proof}

\begin{lem}\label{lem retraction of vgs-locally closed}
If $X\subseteq V\times \nGim$ is a $\vgs$-locally closed subset, then ${\mathcal X}\subseteq \nGik$ is a definably-locally closed subset. 
\end{lem}

\begin{proof}
 Let $V'$ be a projective variety (over the same algebraically closed field) such that $V$ is an open subset of $V'$. Then $X$ is still a $\vgs$-locally closed  subset of $V'\times \nGim$. By Lemma \ref{lem P and Q for loc closed} let $P$ and $Q$ be bounded $\vgs$-closed subsets such that  $X$ is $\vgs$-open in $P$ and $Q=P\setminus X$.

 Now we may assume that the pro-definable deformation retraction $H:I\times \widehat{X}\to \widehat{X}$ is the restriction of a  pro-definable deformation  retraction $H:I\times \widehat{V'}\to \widehat{V'}$ preserving $V$, $P$, $X$ and $Q$. Let ${\mathfrak P}$, ${\mathfrak Q}$ (and ${\mathfrak X}$) be the images 
 under $H$  of $\widehat{P}$, $\widehat{Q}$ (and $\widehat{X}$) respectively. Then ${\mathfrak P}$, ${\mathfrak Q}$ and ${\mathfrak X}$ are iso-definable subset of $\widehat{V'}$ and so ${\mathfrak Q}$ and ${\mathfrak X}$ are iso-definable subset of  ${\mathfrak P}$. Let $h:{\mathfrak P}\to {\mathcal P}$ be a pro-definable 
 homeomorphism between ${\mathfrak P}$ and a definable subset ${\mathcal P}$ of $\nGik$. Then ${\mathcal Q}=h({\mathfrak Q})$ and ${\mathcal X}=h({\mathfrak X})$ 
 are definable subsets of ${\mathcal P}$. By \cite[Proposition 4.2.9]{HrLo} ${\mathcal P}$ and ${\mathcal Q}$ are definably compact definable subsets of $\nGik$, in particular they are closed. Thus, ${\mathcal X}={\mathcal P}\setminus {\mathcal Q}$ is an open definable subset of a closed definable subset, i.e. it is  definably locally closed.
 \end{proof}

As above, let $H\colon I\times \widehat{X}\to \widehat{X}$ be a continuous pro-definable deformation retraction with image an iso-definable 
subset ${\mathfrak X} $ of $\widehat{X}$ for which there is  a pro-definable homeomorphism $h\colon {\mathfrak X}\to \mathcal{X}$  with a definable subset
 ${\mathcal X}$ of $\nGik$.  Then we have a retraction $r=H\circ i_1\colon \widehat{X}\to {\mathfrak X}$ where $i_1\colon \widehat{X}\to I\times {\widehat X}$ is  the natural inclusion $\widehat{X}\to \{e_I\}\times \widehat{X}$ where $e_I$ is the end point of $I$.

As in \cite[Remark 11.1.3 (5)]{HrLo} we have:

\begin{nrmk}\label{nrmk proper retraction}
The pro-definable retraction $\widehat{X}\to {\mathfrak X}$ and the pro-definable homeomorphism  $h\colon {\mathfrak X}\to \mathcal{X}$  can be assumed to be definably proper i.e., the pullback of a definably compact set is definably compact.

Indeed, since $h\colon{\mathfrak X} \to {\mathcal X}$ is a pro-definable homeomorphism, it is definably proper. On the hand, take $V'$ a projective variety (over the same algebraically closed field) such that $V$ is an open subset of $V'$ and assume that the pro-definable deformation retraction $H:I\times \widehat{X}\to \widehat{X}$ is the restriction of a pro-definable deformation retraction $H:I\times \widehat{V'}\to \widehat{V'}$. If ${\mathfrak V}'$ is the image of $H:I\times \widehat{V'}\to \widehat{V'}$, then the retraction $\widehat{V'}\to {\mathfrak V}'$ is definably proper, hence its restrition $\widehat{X}\to {\mathfrak X}$ is also.
\end{nrmk}

\begin{nrmk}\label{nrmk c in ACVF and hats}
Let $X\subseteq V\times \nGim$ be a definable subset. Below we denote by $c$ the family of $\vgs$-supports on $X$ whose elements are the closed $\vgs$-subsets of $X$ each contained in some bounded $\vgs$-closed subset of $X$. We will also denote by $c$ the family of $\hvgs$-supports on $\widehat{X}$ whose elements are the definably compact definable subsets of $\widehat{X}$. 

Note the following:
\begin{itemize}
\item[$\bullet$] 
The isomorphism of sites $\widehat{X}_{\hvgs}\to X_{\vgs}$ is a morphism in $\fT$ and the inverse image of the family $c$ of $\vgs$-supports of $X$ is the family $c$ of $\hvgs$-supports on $\widehat{X}$. The later follows from Remark \ref{nrmk def compact in hat} and the fact that definable subset of $\widehat{X}$ is the same as $\hvgs$-subset of $\widehat{X}$. Hence the isomorphism of sites $\widehat{X}_{\hvgs}\to X_{\vgs}$ induces an isomorphism in cohomology 
$$H^*_c(X;\cF)\simeq H^*_c(\widehat{X};\widehat{\cF}).$$
\item[$\bullet$]
If $X$ is a $\vgs$-locally closed subset, then  the family $c$ of $\hvgs$-supports on $\widehat{X}$ is a family of $\hvgs$-normal supports on $\widehat{X}$. Indeed, by Corollary \ref{cor:mix-normal} $V\times \nGim$ is weakly $\vgs$-normal, so $P$ (of Lemma \ref{lem P and Q for loc closed})  being a $\vgs$-closed subset of $V\times \nGim$ is also weakly $\vgs$-normal. Since $X$ is a $\vgs$-open subset of $P$ and every bounded $\vgs$-closed subset of $X$ is also a $\vgs$-closed subset of $P$ the result follows from Definition \ref{defn wtnormal} (2).
\end{itemize}
\end{nrmk}

It follows from the last two remarks that:

\begin{nrmk}\label{nrmk r* and h* are iso}
If $X\subseteq V\times \nGim$ is a $\vgs$-locally closed definable subset, then we have induced homomorphisms  
$$r^*\colon H^*_c({\mathfrak X};L_{{\mathfrak X}})\to H^*_c(\widehat{X};L_{\widehat{X}})$$ 
and
$$h^*\colon H^*_c(\mathcal{X};L_{\mathcal{X}})\to H^*_c({\mathfrak X};L_{{\mathfrak X}})$$ 
which are isomorphisms. In particular, 
$$H^*_c(\widehat{X};L_{\widehat{X}})\simeq  H^*_c(\mathcal{X};L_{\mathcal{X}}).$$

Indeed, let $j\colon {\mathfrak X}\to \widehat{X}$ be the inclusion. Since ${\mathfrak X}=H(e_I, \widehat{X})$ we have $r\circ j=\id _{{\mathfrak X}}$; and as usual $H$ is a homotopy between $j\circ r$ and $\id _{\widehat{X}}$. Moreover, since $H$ is a morphism of $\hvgs$-sites (Lemma \ref{lem retration is vgs}), we also have that  $r\circ j$ and $j\circ r$ are morphisms of $\hvgs$-sites. Since ${\mathfrak X}=r(\widehat{X})$ we get that $r^*$ is an isomorphism by homotopy axiom (Theorem \ref{thm homotopy ax stable}); $h^*$ is an isomorphism since $h$ is a pro-definable homemorphism.

Note that by the same argument, taking the family of $\vgs$-supports  given by the $\vgs$-closed subsets, we get, for any definable subset $X\subseteq V\times \nGim$, an isomorphism 
$$H^*(\widehat{X};L_{\widehat{X}})\simeq  H^*(\mathcal{X};L_{\mathcal{X}}).$$
\end{nrmk}

\begin{thm}\label{thm finiteness in acvf}
Let $A$ be a  noetherian ring and let $L$ be a  finitely generated  $A$-module. Let $V$ be a variety over $K$.  If $X\subseteq V\times \nGim$ is a $\vgs$-locally closed definable subset, then $H^p_c(X;L_X)\simeq H^p_c(\widehat{X};L_{\widehat{X}})$ is  finitely generated for each $p$. 
\end{thm}

\begin{proof}
The variety $V$ is a finite union of open quasi-projective sub-varieties $V_1, \ldots , V_n$ and the result follows by induction on $n.$ In the case $n=1$ the variety  $V$ is quasi-projective and,  since $\mathcal{X}$ is definably locally  closed by Lemma \ref{lem retraction of vgs-locally closed}, the result follows from Theorem \ref{thm fg coho def groups} and the isomorphism of Remark \ref{nrmk r* and h* are iso}.

Now suppose that $V=V_1\cup \cdots \cup V_{n+1}$ and, by the induction hypothesis, the result holds for $U=V_1\cup \cdots \cup V_n$. Let $Y=X\cap (U\times \nGim)$ and $Z=X\cap (V_{n+1}\times \nGim)$. Then $Y$ (resp $Z$) is a   $\vgs$-locally closed definable subset of $U\times \nGim$ (resp. $V_{n+1}\times \nGim$) and  $H^p_c(Y;L_Y)$ (resp. is $H^p_c(Z;L_Z)$)  finitely generated for each $p$. Note also that $Y\cap Z$ is a  $\vgs$-locally closed definable subset of  $V_{n+1}\times \nGim$ and  $H^p_c(Y\cap Z;L_{Y\cap Z})$  finitely generated for each $p$.

Since $Y, Z$ and $Y\cap Z$ are $\vgs$-locally closed subsets of $X$, the short exact coefficients sequence 
$$0\to L_{Y\cap Z}\to L_Y\oplus L_Z\to L_X\to 0$$
together with Corollary \ref{cor hphi and !} give the Mayer-Vietoris  sequence
\[\ldots \rightarrow H^l_c(Y;L_Y)\oplus H^l_c(Z;L_Z)\stackrel{\alpha}\rightarrow H^l_c(X;L_X)\stackrel{\beta}\rightarrow H^{l+1}_c(Y\cap Z;L_{Y\cap Z})\rightarrow  \ldots .\]
It  follows also that $\ker \beta ={\rm Im}\,\alpha $ is finitely generated and ${\rm Im}\,\beta $ is finitely generated. Thus 
we see that $H^p_c(X;L_X)$ is finitely generated for each $p$.
\end{proof}

Let $K'$ be an elementary extension of $K$. Let as before $H\colon I\times \widehat{X}\to \widehat{X}$ be a continuous pro-definable deformation retraction with image an iso-definable 
subset ${\mathfrak X} $ of $\widehat{X}$ for which there is  a pro-definable homeomorphism $h\colon {\mathfrak X}\to \mathcal{X}$  with a definable subset
 ${\mathcal X}$ of $\nGik$. Then we have commutative diagrams of morphisms of $\hvgs$-sites
 
\begin{equation*}
\xymatrix{
I(K')\times \widehat{X}(K') \ar[r]^{\,\,\,\,\,\,\,\,\,\,\,\,\,\,\,H^{K'}}  \ar[d]^{} & \widehat{X}(K') \ar[d]^{} \\
I\times \widehat{X} \ar[r]^{H} & \widehat{X}
}
\end{equation*} 

and 

\begin{equation*}
\xymatrix{
{\mathfrak X}(K') \ar[r]^{h^{K'}}  \ar[d]^{} & \mathcal{X}(K') \ar[d]^{} \\
{\mathfrak X} \ar[r]^{h} & \mathcal{X}
}
\end{equation*} 
where the vertical arrows are the morphisms of $\hvgs$-sites taking $\widehat{U}$ to $\widehat{U}(K')$.

Note that $H^{K'}\colon I(K')\times \widehat{X}(K') \to \widehat{X}(K')$ is a continuous a pro-definable deformation retraction with image an iso-definable  subset ${\mathfrak X}(K')$ of $\widehat{X}(K')$, $h^{K'}\colon {\mathfrak X}(K')\to \mathcal{X}(K')$ is a pro-definable homeomorphism  to a definable subset ${\mathcal X}(K')$ of $\nGik(K')$. Furthermore, if $X$ is a $\vgs$-locally closed definable subset, then $X(K')$ is also.

\begin{thm}\label{thm inv in acvf}
Let $V$ be a  variety over $K$. Let $K'$ be an elementary extension of $K$.  If $X\subseteq V\times \nGim$ is a $\vgs$-locally closed definable subset, then  for every $\cF\in \mod(A_{X_{\vgs}})$ we have
an isomorphism
\[H^*_c(X;\cF)\simeq H^*_c(X(K'); \cF(K')).\]
\end{thm}

\begin{proof}
First let us proof that for every $\vgs$-locally closed definable subset $X\subseteq V\times \nGim$  we have an isomorphism
\[H^*_c(X;A_{X})\simeq H^*_c(X(K'); A_{X(K')}).\]

The variety $V$ is a finite union of open quasi-projective sub-varieties $V_1, \ldots , V_n$ and the result follows by induction on $n.$ In the case $n=1$ the variety  $V$ is quasi-projective and the result  follows from the observations above, Remark \ref{nrmk r* and h* are iso} and Theorem \ref{thm bf inv gp-int} taken in $K$ and in $K'$. 

The inductive step is obtained from the following commutative diagram with the Mayer-Vietoris sequences as in the proof of Theorem \ref{thm finiteness in acvf} taken in $K$ and in $K'$, the inductive hypothesis and the five lemma
{\tiny 
\begin{equation*}
\xymatrix{
 \ar[r]^{}   & H^l_c(Y;A_{Y})\oplus H^l_c(Z;A_{Z})\ar[r]^{}  \ar[d]^{} & H^l_c(X;A_{X})\ar[r]^{}  \ar[d]^{} & H^{l+1}_c(Y\cap Z;A_{Y\cap Z}) \ar[r]^{} \ar[d]^{} &  \\
\ar[r]^{}  & H^l_c(Y(K');A_{Y(K')})\oplus H^l_c(Z(K');A_{Z(K')})\ar[r]^{}  & H^l_c(X(K');A_{X(K')})\ar[r]^{}  & H^{l+1}_c((Y\cap Z)(K');A_{(Y\cap Z)(K')}) \ar[r]^{}  & 
}.
\end{equation*}
}

By Lemma \ref{lem P and Q for loc closed} let $P$ be a definably compact $\vgs$-closed subset such that $X\subseteq P$ is a $\vgs$-open subset. It follows from the above that for every $\vgs$-closed subset $Z$ of $P$ we have   
\[H^*(Z;A_{Z})\simeq H^*(Z(K'); A_{Z(K')}).\]

With this set up, the rest of the proof is exactly the same as that of Theorem \ref{thm bf inv gp-int} using Corollary \ref{cor hphi and !}, Remark \ref{nrmk def compact in hat} instead of Remark \ref{nrmk def comp} and Lemma \ref{lem criteria for T-sheaves}.
\end{proof}

\begin{nrmk}\label{rem cohom general sheaves}
Note that for the cohomologies $H^*(X;L_X)$ and $H^*(X;\cF)$ without supports we could also get finiteness and  invariance results respectively if we knew the corresponding results in $\bGi$. Indeed, by the isomorphism $H^*(X;L_X)\simeq H^*(\mathcal{X};L_{\mathcal{X}})$ of Remark \ref{nrmk r* and h* are iso}, in the quasi-projective case, we would get the finiteness and invariance results for any definable subset $X\subseteq V\times \nGim$. For the general case $X$ would have to assumed $\vgs$-normal (in order to use Corollary \ref{cor hphi and !}) and the results would follows by exactly the same arguments as above.
\end{nrmk}

\subsection{Vanishing of cohomology}\label{sec vanishing acvf}

Below are a couple of quick remarks about vanishing of cohomology in our setting.

By the isomorphisms of Remark \ref{nrmk r* and h* are iso} and Theorem \ref{thm fg coho def groups} we immeadiately obtain:

\begin{thm}
Let $V$ be a quasi-projective variety over $K$. If $X\subseteq V\times \nGim$ be a $\vgs$-locally closed subset, then $H^p_c(X;L_X)=0$ for all $p> \dim {\mathcal X}$.
\qed
\end{thm}

When there is a good notion of dimension, recall Remark \ref{nrmk spec chains and dim}, we can obtain the following results.

\begin{thm}\label{thm vanish vg-normal}
Let $X$ be a definable subset of a variety $V$ over $K$. If $X$ is  $\vgs$-normal, then $H^p(X;\cF)=0$ for every $p>\dim(\overline{X}^\mathrm{Zar})$ and every $\cF\in \mod (A_{X_{\vgs}})$, where $\overline{X}^\mathrm{Zar}$ denotes the Zariski closure of $X$.
\end{thm}

\begin{proof}
Going to $\widetilde{\widehat{X}}$ we obtain a normal spectral space and the result follows from \cite[Corollary 6]{cr}, since $\dim _{{\rm Krull}}(\widetilde{\widehat{X}})$ i.e., the maximal length of a chain of proper specializations of points in $\widetilde{\widehat{X}}$, is bounded by $\dim(\overline{X}^\mathrm{Zar})$  by Corollary \ref{cor spec length}.
\end{proof}

Based on Propositions \ref{prop phi soft}, \ref{prop injective for soft and flabby} and \ref{prop soft open and phi-dim}  together with Theorem \ref{thm vanish vg-normal} the prove of the following result is classical. Compare with \cite[Theorem 3.12]{ep1} in the o-minimal case, with \cite[Corollary 9.4]{D3} in the locally semi-algebraic case, or with \cite[Chapter II, 16.2 and 16.4]{b} in the topological case.

\begin{thm}\label{thm vanish vg-loc closed}
Let $V$ be a variety over $K$. If $X$ is  $\vgs$-locally closed subset of $V$, then $H^p_c(X;\cF)=0$ for every $p>\dim(\overline{X}^\mathrm{Zar})$ and every $\cF\in \mod (A_{X_{\vgs}})$.
\end{thm}

\begin{proof}
We work in $\widetilde{\widehat{X}}$  and we let $n=\dim(\overline{X}^\mathrm{Zar})$. Since the functor $\Gamma _c(X; \bullet )$ is left exact and the full additive subcategory of $\mod(A_{\widetilde{\widehat{X}}})$ of $c$-soft sheaves is $\Gamma _c(X; \bullet )$-injective (Proposition \ref{prop injective for soft and flabby}), by the general result of homological algebra \cite[Exercise I.19]{ks1} it is enough to show that if   $0\into {\mathcal F}\into {\mathcal I}^0\into {\mathcal I}^1\into \cdots \into {\mathcal I}^n\into 0$ is an exact sequence of sheaves in $\mod (A_{\widetilde{\widehat{X}}})$ such that ${\mathcal I}^k$ is $c$-soft  for $0\leq k\leq n-1$, then ${\mathcal I}^n$ is $c$-soft.

By Proposition  \ref{prop phi soft} (4) it suffices to prove that ${\mathcal I}^n_{|Z}$ is soft for every constructible subset $Z$ of $\widetilde{\widehat{X}}$ which is in the normal and constructible family of supports $c$.  Since $c$ is normal, there is a constructible neighborhood $Y$ of $Z$ in $\widetilde{\widehat{X}}$ which is in $c$. If we show that ${\mathcal I}^n_{|Y}$ is soft, then it will follow that ${\mathcal I}^n_{|Z}$ is soft (Proposition  \ref{prop phi soft} (2)).

Let $U$ be an open and constructible subset of $Y$. By hypothesis and Proposition \ref{prop soft open and phi-dim} each $({\mathcal I}^k_{|Y})_U$ is acyclic for $0\leq k \leq n-1$. Let ${\mathcal Z}^k={\rm ker}(({\mathcal I}^k_{|Y})_U\into ({\mathcal I}^{k+1}_{|Y})_U).$ Then the long exact cohomology sequences of the short exact sequences $0\into {\mathcal Z}^k\into ({\mathcal I}^k_{|Y})_U\into {\mathcal Z}^{k+1}\into 0$ show that
{\small
$$H^q(Y;({\mathcal I}^n_{|Y})_U)=H^q(Y;{\mathcal Z}^n)=H^{q+1}(Y;{\mathcal Z}^{n-1})=\cdots =H^{q+n}(Y;{\mathcal Z}^0)=H^{q+n}(Y;({\mathcal F}_{|Y})_U).$$
}

Since $Y$ is a normal, constructible open subset of $\widetilde{\widehat{X}}$ we have $H^p(Y;{\mathcal G})=0$ for $p>n$ and every sheaf ${\mathcal G}$ on $Y$ by Theorem \ref{thm vanish vg-normal} applied to an appropriate $\vgs$-normal open definable subset of $X$ whose tilde of the hat is $Y$. Thus  $H^1(Y;({\mathcal I}^n_{|Y})_U)=0$. Since $U$ was an arbitrary  open and constructible subset of $Y$, it follows from Proposition \ref{prop soft open and phi-dim} that  ${\mathcal I}^n_{|Y}$ is soft as required.
\end{proof}

\section{Relation with Berkovich spaces}\label{sec:rel-berkovich}

In this section we relate our results to classical results of Berkovich spaces. Let us first recall the relation --already stablished in \cite{HrLo}-- between the stable completion of algebraic varieties and their analytification.

Let $(F,\val)$ be a non-archimedean valued field of rank 1 (i.e $\val (F)\subseteq \RR _{\infty}$). We allow $F$ to be trivially valued and do not assume it to be complete. Consider $\bF\coloneqq  (F, \RR)$ as a substructure of a model of ACVF. Let $V$ be an algebraic variety over $F$ and $X$ be an $\bF$-definable subset of $V$. As in \cite[Chapter 14]{HrLo}, we define $B_{\bF}(X)$, the {\it model-theoretic Berkovich space of $X$}, to be the space of types over $\bF$, concentrated on $X$, which are almost orthogonal to $\RR$. The last condition means that for any $\bF$-definable map $h\colon X\to \RR_{\infty }$ and for any $p\in B_{\bF}(X)$, the push-forward $h_*(p)$ is concentrated on a point of $\RR _{\infty }$ which is denoted by $h(p)$. The topology on $B_{\bF}(X)$ is the topology having as a basis finite unions of finite intersections of sets of the following form
\[
\{p\in X\cap U\mid f(p)\in I\}, 
\]
where $U$ is a Zariski open set of $V$, $f\in \cO_V^{\val}(U)$ and $I$ is an open  interval on $\RR_{\infty }$. In addition, the construction is functorial, that is, given an $\bF$-definable variety $V'$, and $\bF$-definable subset $X'$ of $V'$ and an $\bF$-definable morphism $f\colon X\to X'$, there is an induced map $B_{\bF}(f)\colon B_{\bF}(X)\to B_{\bF}(X')$.

Suppose now that $F$ is complete. As a set, Berkovich's analytification $V^\an$ of $V$ can be described as pairs $(x, u_x)$ with $x$ a point (in the schematic sense) of $V$ and $u_x\colon F(x)\to \RR_{\infty }$ a valuation extending $\val $ on the residue field $F(x)$ of the stalk at $x$. The topology on $V^\an$ is the coarsest topology such that for every $f\in \cO_V(U)$ the maps $x\mapsto u_x\circ f_x$ where $f_x\in \cO_{V, x}\simeq F(x)$ are all continuous. By results in \cite{ber90}, the topological space $V^\an$ is Hausdorff (recall $V$ is assume to be separated), locally compact and locally path connected. The induced topology on the subset $V(F)$ coincides with the valuation topology, and if $F$ is algebraically closed, then this subset is dense. 

As explained in \cite[Section 14.1]{HrLo}, we have:

\begin{fact}\label{fact B_F(V) and V^an}
If $F$ is complete, then the model-theoretic Berkovich space $B_{\bF}(V)$ is canonically homeomorphic to $V^\an$. Moreover, the image of $B_{\bF}(X)$ in $V^\an$ is a semi-algebraic subset in the sense of \cite{duc}, and every semi-algebraic subset of $V^\an$ is of this form. \qed
\end{fact}

Let us now recall the relation between $B_{\bF}(V)$ and $\widehat{V}$. Let $F^{max}$ be the unique, up to isomorphism over $\bF$, maximally complete algebraically closed field, containing $F$, with value group $\RR$, and residue field equal to the algebraic closure of the residue field of $F$. Recall that maximally complete means that every family of balls with the finite intersection property has a non empty intersection. By \cite[Lemma 14.1.1 and Proposition 14.1.2]{HrLo} we have:

\begin{fact}\label{fact hat and berkovich}
There is a continuous surjective closed map
\[
\pi \colon \widehat{X}(F^{max})\to B_{\bF}(X).
\]
If $F=F^{max},$ then $\pi $ is a homeomorphism.\qed
\end{fact}

For the rest of this section we let $\widehat{X}$ denote $\widehat{X}(K)$ for some (any) model $K$ of ACVF containing $\bF$. This means in particular that statements about $\widehat{X}$ hold for $\widehat{X}(K)$ for any such model. The following remark provides an example of this.

\begin{nrmk}\label{rem:defconect} Let $K$ be a model of ACVF containing $\bF$. Observe that one can express (over $\bF$) the fact that $X(K)$ is covered by two disjoint $\vgs$-open subsets. Therefore, if two such disjoint $\vgs$-open subsets exist in $K$, they exist in every model of ACVF extending $\bF$. By Fact \ref{fact top on hat and vgs}, this shows that $\widehat{X}(K)$ is definably connected if and only if $\widehat{V}(K')$ is definable connected for any model of ACVF $K'$ containing $\bF$. 
\end{nrmk}

\begin{prop}\label{prop connected and def-connected}
Let $X$ be an $\bF$-definable subset of an algebraic variety $V$ over $F$. Then, the following are equivalent:
\begin{enumerate}
\item[(1)] $\widehat{X}$ is definably connected 
\item[(2)] $\widehat{X}(F^{max})$ is connected.
\end{enumerate}
In addition, if $\widehat{X}(F^{max})$ is connected so is $B_\mathbf{F}(X)$. In particular, when $F$ is complete, if $\widehat{V}(F^{max})$ is connected then $V^\an$ is connected.
\end{prop}

\begin{proof}
By Remark \ref{rem:defconect}, (1) above is equivalent to $\widehat{X}(F^{max})$ is definably connected. Suppose first $V$ is quasi-projective. By the main theorem of Hrushovski and Loeser \cite[Theorem 11.1.1]{HrLo}, there is a pro-definable deformation retraction of $\widehat{X}(F^{max})$ to a definable subset ${\mathcal X}$ of $\mathbb{R}_\infty^n$. Hence, $\widehat{X}(F^{max})$ is definably connected (resp. connected) if and only if ${\mathcal X}$ is definably connected (resp. connected). But ${\mathcal X}$ is definably connected if and only if it is connected (the proof in o-minimal expansions of $(\RR, <)$ given in \cite[Chapter III, (2.18) and (2.19) Exercise 7]{vdd} uses only cell decomposition and works in $\RR_{\infty}$).

If $V$ is not quasi-projective, consider an open immersion $V\rightarrow W$ where $W$ is a complete variety and $V$ is Zariski dense. By Chow's lemma, there is an epimorphism $f\colon W'\to W$ where  $W'$ is a projective variety. Consider the quasi-projective variety $V'=f^{-1}(V)$ and let  $X'=f^{-1}(X) \subseteq V'$. By Lemma \ref{lem vgs components},  $\widehat{X'}(F^{max})$ has finitely many definably connected components, say $C'_1, \ldots , C'_k$, which by the above are also the finitely many connected components of the topological space $\widehat{X'}(F^{max})$. The image $C_i$ of each connected component $C'_i$ of $\widehat{X'}(F^{max})$ under the restriction of $f_|$ is a connected subset of $\widehat{X}(F^{max})$. Note that these connected subsets of $\widehat{X}(F^{max})$ are also definably connected subsets. 

Now we follow the idea in the proof of \cite[Chapter III, (2.18)]{vdd}, as above in $\RR _{\infty}$ with cells for the $C_i$'s,  to show that some union of the $C_i$'s is both a definably connected component and a connected component of $\widehat{X}(F^{max})$. From which it follows that  $\widehat{X}(F^{max})$ is definably connected if and only if it is connected.  Let us construct inductively a sequence $i_1,\ldots, i_j,\ldots $ such that for each $j$, $\bigcup _{l\leq j}C_{i_l}$ is both definably connected and connected. Set  $i_1=1$ and if $i_1, \ldots , i_q$ are given let $i_{q+1}=\min \{i: i\notin \{i_1, \ldots , i_q\} \,\,\textrm{and}\,\, C_i\cap (\bigcup _{l\leq q}C_{i_l})\neq \emptyset \}$.  Let $p$ be the maximal length of such a sequence.  We claim that $C:=\bigcup _{j=1}^pC_{i_j}$ is both a definably connected component and a connected component of $\widehat{X}(F^{max})$. 

To see the claim it is enough to show that if $Y\subseteq \widehat{X}(F^{max})$ is a definably connected (resp. connected) subset such that $Y\cap C\neq \emptyset$, then $Y\subseteq C$. Indeed, if this holds, taking $Y$ a definably connected (resp. connected) open neighborhood of a point of $C$ (resp. a point of the complement of $C$) we see that $C$ is open (resp. closed). Now let $I=\{i:C_i\cap Y\neq \emptyset \}$ and $C_Y=\bigcup _{i\in I} C_i$. Since the $C_i$'s cover $\widehat{X}(F^{max})$ we have $Y\subseteq C_Y$. So $C_Y$ is definably connected (resp. connected) being of the form $Y\cup \bigcup _{i\in I}C_i$. Since  $\emptyset \neq C\cap Y\subseteq C\cap C_Y$, it follows that $C\cup C_Y$ is definably connected (resp. connected). By maximality of $C$ we have $C\cup C_Y=C$ and so $Y\subseteq C$ as required.

The last statement of the Proposition follows from Fact \ref{fact hat and berkovich}.  
\end{proof}

\begin{nrmk}\label{rem:connected}  
In general, the connectedness of $B_{\mathbf{F}}(X)$ does not imply $\widehat{X}(F^{max})$ is connected. Take $F$ to be $\mathbb{R}$ with the trivial valuation and let $X$ be defined by $x^2+x+1$. The space $B_\mathbf{F}(X)$ is connected consisting of a single point that corresponds to the Galois orbit of the two 3rd roots of unity. In contrast, $\widehat{X}(F^{max})$ consists of two isolated points, and is therefore not connected.
\end{nrmk}

As an application we recover  a result of Ducros \cite{duc} about the  number of connected components of semi-algebraic subsets of $V^\an$. Let $\rho$ denote the homeomorphism from $B_\bF(V)$ to $V^\an$. 

\begin{thm}[Ducros \cite{duc}] Suppose $F$ is complete. Let $V$  be a variety over $F$ and let $Y$ be a semi-algebraic subset of $V^\an$. Then $Y$ has finitely many connected components. 
\end{thm}

\begin{proof} 
Since $Y$ is a semi-algebraic subset of $V^\an$, there is an $\bF$-definable subset $X$ of $V$ such that $Y=\rho (B_\bF(X))$. 
By Lemma \ref{lem vgs components}, Proposition \ref{prop connected and def-connected} and Fact \ref{fact hat and berkovich}, $\widehat{X}(F^{max})$ has finitely many connected components. Now, the image under $\pi$ (as defined in Fact \ref{fact hat and berkovich}) of a connected component of $\widehat{X}(F^{max})$ lies in a connected component of $B_\bF(X)$. Since $\pi$ is surjective, this shows the result.  
\end{proof}

Note also that, if $F$ is algebraically closed, then each connected component of $Y$ is also a semi-algebraically connected component. Indeed, by Lemmas \ref{lem vgs and def conn} and \ref{lem vgs components}, $\widehat{X}(F^{max})$ has finitely many definably connected components each of which is of the form $\widehat{U_i}(F^{max})$ for some $\vgs$-clopen subsets $U_i$ ($i=1,\ldots, k$)  of $X$. By  Proposition \ref{prop connected and def-connected}, each  $\widehat{U_i}(F^{max})$ is also a connected component of $\widehat{X}(F^{max})$.  Since $\pi :\widehat{X}(F^{max})\to B_F(X)$ (as defined in Fact \ref{fact hat and berkovich}) is continuous, surjective and closed and, $\pi (\widehat{U_i}(F^{max}))=B_F(U_i)$, each $B_F(U_i)$ is a closed connected subset of $B_F(X)$. Let  $I_1,\ldots , I_q$ be a finite partition  of $\{1,\ldots, k\}$ such that for each $j\leq q$, $\bigcup _{l\in I_j}B_F(U_l)$ is connected and $I_j$ is maximal with this property. Then as in the last part of the proof of Proposition \ref{prop connected and def-connected}, each $\bigcup _{l\in I_j}B_F(U_l)=B_F(\bigcup _{l\in I_j}U_l)$ is  a connected component of $\widehat{X}(F^{max})$ with $V_j=\bigcup _{l\in I_j}U_l$ a $\vgs$-clopen subset of $X$. Since $F$ is a model of ACVF, by Remark \ref{rem:defconect},  each $V_j$ is $\bF$-definable.


\begin{fact}[{\cite[Proposition 14.1.2]{HrLo}}]\label{berk compact} 
Let $X$ be an $\bF$-definable subset of an algebraic variety $V$ over $F$. Then the following are equivalent:
\begin{itemize}
\item $\widehat{X}$ is definably compact, 
\item $\widehat{X}(F^{max})$ is compact,
\item $B_{\bF}(X)$ is compact. 
\end{itemize}
In particular, when $F$ is complete, $\widehat{V}$ is definably compact if and only if $V^\an$ is compact.\qed
\end{fact}

The proof of the previous fact in \cite{HrLo} does not seem to include the proof that $B_{\bf{F}}(X)$ is Hausforff. For the reader's convenience, we include here an argument. 

\begin{lem}
Let $X$ be an $\bF$-definable subset of an algebraic variety $V$ over $F$. Then $B_\bF(X)$  is a Hausdorff topological space.
\end{lem}

\begin{proof} 
Note that it suffices to show that when $X=V$ is a complete variety over $F$, then $B_\bF(V)$ is Hausdorff. Indeed, by Nagata's theorem there is an open immersion $V\to V'$ where $V'$ is a complete variety. Since $B_\bF(X)$ is a subspace of $B_\bF(V')$ it is also Hausdorff.  

We first show that points in $B_\bF(V)$ are closed. Note that for each affine open subset $W\subseteq V$, $B_\bF(W)$ is an open subspace of $B_\bF(V)$. Hence it suffices to show that points in $B_\bF(W)$ are closed for affine $W$. Actually, this is true since $B_\bF(W)$ is even Hausdorff. Let $p,q$ be two distinct points in $B_\bF(W)$, by quantifier elimination in ACVF, there is a regular function $f$ on $W$ such that $\val \circ f_*(p)=r_1\neq \val \circ f_*(q)=r_2$ for some $r_1,r_2 \in \mathbb{R}_\infty$.  Then take two disjoint open subset $r_1\in U_1$ and $r_2\in U_2$. Their inverse image under $\val \circ f$ will be two disjoint opens in $B_\bF(W)$ separating $p,q$.

So suppose now that $V$ is a complete variety. By Remark \ref{nrmk def compact in hat}  $\widehat{V}$ is definably compact (and Hausdorff by \cite[Proposition 3.7.8]{HrLo}) so, by Fact \ref{berk compact}, $B_F(V)$ is quasi-compact. Consider the restriction map $\pi:\widehat{V}(F^{max})\to B_\bF(V)$, note that the map is closed. On the other hand, given $p_1,p_2 \in B_\bF(V)$, the fibers $\pi^{-1}(p_1), \pi^{-1}(p_2)$ are two disjoint compact subset of $\widehat{V}(F^{max})$, hence can be separated by open subsets $U_1$ and $U_2$, then $p_1\in B_\bF(V)\backslash \pi(\widehat{V}(F^{max})\backslash U_1)$ and $p_2 \in B_\bF(V)\backslash \pi(\widehat{V}(F^{max})\backslash U_2)$ are two disjoint open subsets of $B_\bF(V)$ by the fact that $\pi$ is closed. 
\end{proof}

Recall that in general $\widehat{X}$ may fail to be a locally compact space (Remark \ref{nrmk hat top not loc compact}). However,  since $F^{max}=(F^{max})^{max}$, by Fact \ref{fact hat and berkovich}, $\widehat{X}(F^{max})$ is homeomorphic to $B_{F^{max}}(X)$ and, by  the next result, $\widehat{X}(F^{max})$ is a locally compact topological space when $X$ is an $\bF$-definable $\vgs$-locally closed subset of an algebraic variety $V$ over $F$.

\begin{lem}\label{lem vgs loc closed and bf loc closed}
If $X$ is an $\bF$-definable $\vgs$-locally closed subset of an algebraic variety $V$ over $F$, then $B_\bF(X)$ is a locally closed subset of $B_\bF(V)$. In particular, $B_\bF(X)$ is a locally compact topological space.
\end{lem}

\begin{proof}
Let $U$ be an $\bF$-definable $\vgs$-open subset of $V$ and let $Y$ be an $\bF$-definable $\vgs$-closed subset of $V$ such that $X=Y\cap U$. Then  $\widehat{Y}(F^{max})$ is a closed subset of $\widehat{V}(F^{max})$.  By Fact \ref{fact hat and berkovich}, $\pi (\widehat{Y}(F^{max}))=B_\bF(Y)$ and $B_{\bF}(Y)$  is a closed subset of $B_\bF(V)$ since $\pi$ is a closed map.

On the other hand, since $V\setminus U$ is an $\bF$-definable $\vgs$-closed subset of $V$, by the above and the fact that $B_\bF(V)\setminus B_\bF(U)=B_\bF(V\setminus U)$,  $B_\bF(U)$ is an open subset of $B_\bF(V).$ Therefore, $B_\bF(X)=B_\bF(U)\cap B_\bF(Y)$ is a locally closed subset of $B_\bF(V)$. 

By Nagata's theorem there is an open immersion $V\to V'$ where $V'$ is a complete variety (over $F$). Since $V'$ is complete, $\widehat{V'}$ is definably compact (Remark \ref{nrmk def compact in hat}) and so, by Fact \ref{berk compact}, $B_\bF(V')$ is compact. Thus the result follows since (the homeomorphic image of) $B_\bF(X)$ is also locally closed in $B_\bF(V')$.
\end{proof}

\begin{nrmk}\label{rem:paracompact} Let $V$ be an algebraic variety over $F$. Then $B_\bF(V)$ is paracompact. Since $V$ is a scheme of finite type, we may assume $V$ is affine. Furthermore, since $B_\bF(V)$ is locally compact by Lemma \ref{lem vgs loc closed and bf loc closed}, it suffices to it is $\sigma$-compact (see \cite[Page 21]{b}). Now, $B_\bF(V)$ is covered by the sets $B_\bF(X)\subseteq U$ where $X$ is a closed ball with radius in $\mathbb{Q}$. Each such set is compact by Fact \ref{berk compact}. 
\end{nrmk}

\begin{fact}\label{deformation for Bf}
{\em
Let $V$ a quasi-projective variety over $F$ and $X$ an $\bF$-definable subset. Then  there is a strong deformation retraction $\bI\times B_\bF(X)\to B_\bF(X)$ with image a subset $\mathbf{\mathfrak{X}}$ homeomorphic to a semi-algebraic subset $\mathbf{\mathcal{X}}$ of some $\RR ^k$. This deformation retraction is induced by a strong deformation retraction $H:I\times \widehat{X}\to \widehat{X}$ (of \cite[Theorem 11.1.1]{HrLo}) in the sense that there is a commutative diagram 
\begin{equation*}
\xymatrix{
I(F^{max})\times \widehat{X}(F^{max}) \ar[r]^{}  \ar[d]^{i\times \pi} & \widehat{X}(F^{max}) \ar[d]^{\pi}   \\
\bI\times B_\bF(X) \ar[r]^{}  & B_\bF(X)
}
\end{equation*}
of continuous maps, where $i:I(F^{max})=\widehat{I}(F^{max})\to \bI=I(\bF)=I(\RR _{\infty})$ is the canonical identification. See  \cite[Corollary 14.1.6, Proposition 14.1.3 and Theorem 14.2.1]{HrLo}

If $F\leq F''$ is a field extension then we have a commutative diagram of strong deformation retractions 
\begin{equation*}
\xymatrix{
\bI\times B_{\bF ''}(X) \ar[r]^{}  \ar[d]^{} & B_{\bF ''}(X) \ar[d]^{}   \\
\bI\times B_\bF(X) \ar[r]^{}  & B_\bF(X)
}
\end{equation*}
where the vertical arrows are the canonical restrictions. Furthermore, there is a finite Galois extension $F\leq F'$ such that:
\begin{enumerate}
    \item 
    If $F'\leq F''$ then the image $\mathbf{\mathfrak{X}}''$ of $\bI\times B_{\bF ''}(X)\to B_{\bF ''}(X)$ is homeomorphic to the image $\mathbf{\mathfrak{X}}'$ of $\bI\times B_{\bF '}(X)\to B_{\bF '}(X)$.
    \item
    the image $\mathbf{\mathfrak{X}}$ of $\bI\times B_{\bF }(X)\to B_{\bF }(X)$ is homeomorphic to $\mathbf{\mathfrak{X}}'/{\rm Gal}(F'/F)$.
\end{enumerate}

See \cite[Theorem 14.2.3 and the discussion at the begining of Section 14.2]{HrLo}. \qed
}
\end{fact}

An immediate consequence of this fact are the following results about cohomology:

\begin{thm}\label{thm finiteness and inv for Bf}
If $X$ is an $\bF$-definable subset of a quasi-projective variety $V$ over $F$. Then:
\begin{enumerate}
    \item 
    $H^p(B_\bF(X);L_{B_\bF(X)})$ is finitely generated for every $p\geq 0$.
    \item
    There is a finite Galois extension $F\leq F'$ such that for every $F'\leq F''$  we have isomoprhisms
    $$H^*(B_{\bF '}(X); L_{B_{\bF '}(X)})\simeq  H^*(B_{\bF ''}(X); L_{B_{\bF ''}(X)}).$$
\end{enumerate}
\end{thm}

\begin{proof}
The result  follows from Fact \ref{deformation for Bf}, the homotopy axiom in topology (\cite[Chapter II, 11.8]{b}) and the fact that a semi-algebraic subset of $\RR ^k$ is homeomorphic to a finite simplicial complex.
\end{proof}

Note that if $B_{\bF}(X)$ is paracompact (see Remark \ref{rem:paracompact}), then these results can be extended to an arbitrary algebraic variety $V$ over $F$ using the topological analogue of Corollary \ref{cor hphi and !} (\cite[Chapter II, 10.2]{b}) and the corresponding Mayer-Vietoris sequence \cite[Chapter II, (27)]{b} as in the proof of Theorem \ref{thm finiteness in acvf} and Theorem \ref{thm inv in acvf} respectively.

\begin{thm}\label{thm finiteness and inv for Bf with comp supp}
If $X$ is an $\bF$-definable $\vgs$-locally closed subset of an algebraic variety $V$ over $F$. Then:
\begin{enumerate}
    \item 
    $H^p_c(B_\bF(X);L_{B_\bF(X)})$ is finitely generated for every $p\geq 0$.
    \item
    There is a finite Galois extension $F\leq F'$ such that for every $F'\leq F''$  we have isomoprhisms
    $$H^*_c(B_{\bF '}(X); L_{B_{\bF '}(X)})\simeq  H^*_c(B_{\bF ''}(X); L_{B_{\bF ''}(X)}).$$
\end{enumerate}
\end{thm}

\begin{proof}
Suppose first that $V$ is a quasi-projective variety over $F$. Note that since $X$ is an $\bF$-definable $\vgs$-locally closed subset, then $\mathbf{\mathcal{X}}$ is locally closed. The proof of this is similar to that of Lemma  \ref{lem retraction of vgs-locally closed} using $B_{\bF}(X)$ instead of $\widehat{X}$ and Fact \ref{berk compact} instead of \cite[Prposition 4.2.9]{HrLo}.  In this case  the result follows from the homotopy axiom in topology (\cite[Chapter II, 11.8]{b}) and the fact that a semi-algebraic subset of $\RR ^k$ is homeomorphic to a finite simplicial complex.  

In the general case, since the family of compact supports is paracompactifying,   the results are obtained using the topological analogue of Corollary 
\ref{cor hphi and !} (\cite[Chapter II, 10.2]{b}) and the corresponding Mayer-Vietoris sequence \cite[Chapter II, (27)]{b} as in the proof of Theorem \ref{thm finiteness in acvf} and Theorem \ref{thm inv in acvf} respectively.
\end{proof}

Let $K$ be a model of ACVF extending $\bF$. Below when we write $\widehat{X}(K)$ we mean $\widehat{X}(K)$ equipped with the $\hvgs$-site and when we write $\widehat{X}(K)_\top$ we mean the underlying topological space. Observe that since we have a homeomorphism between $\widetilde{X(K)}$ and $\widetilde{\widehat{X}(K)}$ induced by the isomorphism of sites (Remark \ref{nrmk site iso tilde}), the natural inclusion $\widehat{X}(K)\to \widetilde{X(K)}$ induces a natural inclusion $i:\widehat{X}(K)\to \widetilde{\widehat{X}(K)}$ and $\widehat{X}(K)$ with the induced topology is $\widehat{X}(K)_{\top}$.

Given $\cF \in \mod(A_{\widehat{X}(K)_{\hvgs}})$ we let $\cF _{\top}\in \mod(A_{\widehat{X}(K)_{\top}})$ be the sheaf induced by the natural morphism of sites  $\widehat{X}(K)_{\top} \to \widehat{X}(K)_{\hvgs}$. Note that this is the same as $\tilde{\cF}_{|\widehat{X}(K)_{\top}}$. \\

As observed above, when $X$ is an $\bF$-definable $\vgs$-locally closed subset of an algebraic variety $V$ over $F$, then $\widehat{X}(F^{max})_\top$ is a locally compact topological space. Below we use $c$ to denote either the family of compact supports in $\widehat{X}(K)_{\top}$, as usual or, the family of definably compact supports in $\widehat{X}(K)$ as before.  

\begin{nrmk}\label{nrmk cohom berk and hvgs}
If $D\in \widetilde{c}$ is constructible, then the following hold: (1)  $D\cap \widehat{X}(F^{max})_{\top}$ is a quasi-compact subset of $\widetilde{\widehat{X}(F^{max})}$; (2) $D\cap \widehat{X}(F^{max})_{\top}$ has a fundamental system of normal and constructible locally closed neighborhoods in $\widetilde{\widehat{X}(F^{max})}$ and (3) $D\cap \widehat{X}(F^{max})_{\top}$ is a compact subset of $\widehat{X}(F^{max})_{\top}$. 

Indeed, there is $C\in c$ (in $\widehat{X}(F^{max})$) such that $D=\widetilde{C}$ and hence, $D\cap \widehat{X}(F^{max})_{\top}=\widetilde{C}\cap X=C$. Since $C$ is a definably compact subset of $\widehat{X}(F^{max})$ it is a compact subset of $\widehat{X}(F^{max})_{\top}$ (Fact \ref{berk compact}). Hence $C$ is a quasi-compact subset of $\widetilde{\widehat{X}(F^{max})}$. Finally, since $c$ (in $\widehat{X}(F^{max})$) is a definably normal family of supports on $\widehat{X}(F^{max})$, it follows that $C$ has a fundamental system of normal and constructible locally closed neighborhoods in $\widetilde{\widehat{X}(F^{max})}$.

The above verifies the assumptions of Corollary \ref{cor coho around} for the space $\widetilde{\widehat{X}(F^{max})}$ and the subspace $\widehat{X}(F^{max})_{\top}$. Therefore for every $\cF \in \mod(A_{\widehat{X}(F^{max})_{\hvgs}})$ we have an isomorphism 
\[H^*_c(\widehat{X}(F^{max});\cF) \simeq H^*_c(\widehat{X}(F^{max})_{\top}; \cF _{\top}).\]

Note also that if $X$ is $\vgs$-normal, then considering the family of $\vgs$-supports given by the $\vgs$-closed subsets, we get a family of $\vgs$-normal supports on $X$ and, as above, we verify the assumptions of  Corollary \ref{cor coho around} for the space $\widetilde{\widehat{X}(F^{max})}$ and the subspace $\widehat{X}(F^{max})_{\top}$. Hence for every $\cF \in \mod(A_{\widehat{X}(F^{max})_{\hvgs}})$ we have an isomorphism 
\[H^*(\widehat{X}(F^{max});\cF) \simeq H^*(\widehat{X}(F^{max})_{\top}; \cF _{\top}).\]
\end{nrmk}

In the o-minimal context, $\RR$ plays the analogous role of $F^{max}$ in Remark \ref{nrmk cohom berk and hvgs} and in a similar way we have:

\begin{nrmk}
Let $\bG=(\RR, < , +, \ldots )$ be an o-minimal expansion of $(\RR, <, + )$ and let $X\subseteq \RR^m_{\infty }$ be a definably locally closed subset.  Let $X_{\top}$ be $X$ equipped with the topology generated by the open definable subsets. Then $X_{\top}$ is a locally closed subset, in particular, is a locally compact topological space. Note also that if we consider  the natural inclusion $i:X\to \widetilde{X}$ where $\widetilde{X}=\widetilde{X}_{\df}$ is the o-minimal spectrum of $X$, then $X$ with the induced topology is $X_{\top}$. 

Given $\cF \in \mod(A_{X_{\df}})$  let $\cF _{\top}\in \mod(A_{X_{\top}})$ denote the sheaf induced by the natural morphism of sites $X_\top \to X_\df$. Note that this is the same as  $\tilde{\cF}_{|X_{\top}}$. Let also $c$ denote either  the family of compact supports on $X_{\top}$, as usual or, the family of definably compact supports on $X$ as before.  Using Remark \ref{nrmk def comp} instead of Fact \ref{berk compact}, we verify the assumptions of Corollary \ref{cor coho around} for $\widetilde{X}$ and the subspace $X_\top$ exactly in the same way as in Remark \ref{nrmk cohom berk and hvgs}. Therefore  for every $\cF\in \mod(A_{X_{\df}})$ we have an isomorphism
\[H^*_c(X;\cF)\simeq H^*_c(X_{\top}; \cF _{\top}).\]

Similary, as above, if $X$ is definably normal, then 
for every $\cF\in \mod(A_{X_{\df}})$ we have an isomorphism
\[H^*(X;\cF)\simeq H^*(X_{\top}; \cF _{\top}).\]
\end{nrmk}


The topological cohomology of the model-theoretic Berkovich spaces and the $\vgs$-cohomology of their corresponding stable completions are related by the following results:

\begin{thm}\label{thm spec seq bf and hvgs c-sup}
Let $X$ be an $\bF$-definable $\vgs$-locally closed subset of an algebraic variety $V$ over $F$ and let $\cF\in \mod(A_{\widehat{X}(F^{max})_{\hvgs}})$. Then we have a spectral sequence
\[H^p_c(B_\bF(X);R^q\pi _*(\cF_{\top}))\Rightarrow  H^{p+q}_c(\widehat{X}(F^{max});\cF)\] 
which, when $F=F^{max}$, induces isomomrphisms
\[H^p_c(B_{\bF }(X);\pi _*(\cF _{\top}))\simeq   H^{p}_c(\widehat{X}(F^{max});\cF).\]
\end{thm}

\begin{proof}
Considering  the continuous map $\pi \colon\widehat{X}(F^{max})_\top \to B_\bF(X)$  of locally compact topological spaces  (Fact \ref{fact hat and berkovich} and Lemma \ref{lem vgs loc closed and bf loc closed})  we have the corresponding Leray spectral sequence (\cite[Chapter IV, 6.1]{b}) 
\[H^p_c(B_\bF(X);R^q\pi _*(\cF _{\top}))\Rightarrow  H^{p+q}_c(\widehat{X}(F^{max})_\top ;\cF _\top )\] 
and,  by Remark \ref{nrmk cohom berk and hvgs}, we have an isomorphism
\[H^*_c(\widehat{X}(F^{max})_{\top}; \cF _{\top}) \simeq H^*_c(\widehat{X}(F^{max});\cF ).\]
On the other hand, when $F=F^{max}$, we get isomorphisms 
\[H^p_c(B_{\bF }(X);\pi _*(\cF _{\top}))\simeq H^p_c(\widehat{X}(F^{max})_{\top}; \cF _{\top}) \]
since $\pi$ is a homeomorphism (Fact \ref{fact hat and berkovich}).
\end{proof}

Combining Theorems \ref{thm spec seq bf and hvgs c-sup} and \ref{thm finiteness and inv for Bf with comp supp} we have:

\begin{cor}\label{cor spec seq bf and hvgs c-sup}
Let $X$ be an $\bF$-definable $\vgs$-locally closed subset of an algebraic variety $V$ over $F$. Then there is a finite Galois extension $F\leq F'$ such that  we have an isomorphism
\[H^*_c(B_{\bF '}(X);L_{B_{\bF '}(X)})\simeq   H^{*}_c(\widehat{X}(F^{max});L_{\widehat{X}(F^{max})}).\]
In particular, when  $F$ is complete,  we have an isomorphism
\[H^*_c(V^\an\widehat{\otimes }_FF'; {\mathbb Z}) \simeq  H^*_c(\widehat{V}(F^{max});{\mathbb Z}).\]
\end{cor}

\begin{proof}
By Theorem \ref{thm finiteness and inv for Bf with comp supp} we have \[H^*_c(B_{\bF '}(X);L_{B_{\bF '}(X)})\simeq   H^*_c(B_{F '^{max}}(X);L_{B_{F '^{max}}(X)}).\] Since $F'^{max}=F^{max}$  and $F^{max}=(F^{max})^{max}$, by Theorem \ref{thm spec seq bf and hvgs} we have
\[H^*_c(B_{F '^{max}}(X);L_{B_{F '^{max}}(X)})\simeq   H^{*}_c(\widehat{X}(F^{max});L_{\widehat{X}(F^{max})}).\]
\end{proof}

Similarly, using the second part of Remark \ref{nrmk cohom berk and hvgs} and the  Leray spectral sequence (\cite[Chapter IV, 6.1]{b}) we have:

\begin{thm}\label{thm spec seq bf and hvgs}
Let $X$ be an $\bF$-definable $\vgs$-normal subset of an algebraic variety $V$ over $F$ and let $\cF\in \mod(A_{\widehat{X}(F^{max})_{\hvgs}})$. Then we have a spectral sequence
\[H^p(B_\bF(X);R^q\pi _*(\cF_{\top}))\Rightarrow  H^{p+q}(\widehat{X}(F^{max});\cF)\] 
which, when $F=F^{max}$, induces isomomrphisms
\[
\pushQED{\qed} 
H^p(B_{\bF }(X);\pi _*(\cF _{\top}))\simeq   H^{p}(\widehat{X}(F^{max});\cF).
\qedhere
\popQED
\]
\end{thm}

Combining this with Theorem \ref{thm finiteness and inv for Bf} we obtain:

\begin{cor}\label{cor spec seq bf and hvgs}
Let $X$ be an $\bF$-definable $\vgs$-normal subset of a quasi-projective variety $V$ over $F$. Then there is a finite Galois extension $F\leq F'$ such that  we have an isomorphism
\[H^*(B_{\bF '}(X);L_{B_{\bF '}(X)})\simeq   H^{*}(\widehat{X}(F^{max});L_{\widehat{X}(F^{max})}).\]
In particular, when  $F$ is complete,  we have an isomorphism
\[
\pushQED{\qed} 
H^*(V^\an\widehat{\otimes }_FF'; {\mathbb Z}) \simeq  H^*(\widehat{V}(F^{max});{\mathbb Z}).
\qedhere
\popQED
\]
\end{cor}

We conclude the paper with a couple of remarks


Let $X$ be an $\bF$-definable subset of an algebraic variety $V$ over $F$. If $F\leq F''\leq F^{max}$ is a  field extension, then the natural restriction 
$\pi _{F'',F}: B_{\bF ''}(X)\to B_{\bF}(X)$ is a proper map. Indeed, let $V'$ be a complete variety over $F$ such that $V$ is an open subset. Then $\pi _{F'',F}: B_{\bF ''}(V')\to B_\bF(V')$ is proper (by Fact \ref{berk compact}) and we have a commutative diagram 
\begin{equation*}
\xymatrix{
\widehat{V'}(F^{max})\ar[d]_{\pi }\ar[rd]^{\pi }  &   \\
B_{\bF ''}(V') \ar[r]^{\pi _{F'',F}}  & B_\bF(V')
}
\end{equation*}
of surjective restrictions such that $\pi ^{-1}(B_\bF(X))=\pi ^{-1}(B_{\bF ''}(X))=\widehat{X}(F^{max}) $ (\cite[Proposition 14.1.2]{HrLo}). If $C\subseteq B_{\bF}(X)$ is compact, then it is a compact subset of $B_\bF(V')$ and so, $\pi _{F'',F}^{-1}(C)$ is a compact subset of $B_{\bF ''}(V')$ and $\pi ^{-1}(\pi _{F'',F}(C))$ is a compact subset of $\widehat{X}(F^{max})$. Therefore, $\pi (\pi ^{-1}(\pi _{F'',F}^{-1}(C)))=\pi _{F'',F}^{-1}(C)$ is a compact subset of $B_{\bF ''}(X)$ since $\pi :\widehat{X}(F^{max})\to B_{\bF ''}(X)$ is surjective. It follows that:

\begin{nrmk}\label{nrmk c-cohom Bf and Bfalg}
If $X$ is an $\bF$-definable  $\vgs$-locally closed subset of an algebraic variety $V$ over $F$ and $\cF \in \mod (A_{B_{\bF}(X)})$, then we have an isomorphism:
\[
\lind {F''/F} H^*_c(B_{\bF ''}(X); \cF _{F''})\simeq H^*_c(B_{\bF ^{{\rm alg}}}(X);\cF _{F^{{\rm alg}}})
\]
where the limit is taken over all finite Galois extensions $F\leq F''$ contained in $F^{{\rm alg}}$,  $\cF _{F^{{\rm alg}}}=\pi _{F^{{\rm alg}}, F}^{-1}\cF$ and $\cF _{F''}=\pi _{F'',F}^{-1}\cF$.

Indeed, since  $B_{\bF ^{{\rm alg}}}(X)$ and $\lpro {F''/F}B_{\bF ''}(X)$ are homeomorphic, the result  follows from  \cite[Chapter II, 14.5]{b}.
\end{nrmk}

Now let $G={\rm Aut}(F^{{\rm alg}}/F^h)$ be the absolute Galois group of the Henselization $F^h$ of $F$. Recall that this is  the group of valued field automorphism of $F^{{\rm alg}}$ over $F$. As observed in \cite[page 205]{HrLo},  $G$ acts continuously on $B_{\bF ^{{\rm alg}}}(V)$ and $B_\bF(V)=B_{\bF ^{{\rm alg}}}(V)/G$. Therefore, since this action leaves $B_{\bF ^{{\rm alg}}}(X)$ invariant,  $G$ acts continuously on $B_{\bF ^{{\rm alg}}}(X)$ and $B_\bF(X)=B_{\bF ^{{\rm alg}}}(X)/G$.

Since $G=\lpro {F''/F^h}{\rm Gal}(F''/F^h)$ where the limit is taken over all finite Galois extensions $F^h\leq F''$ contained in $F^{{\rm alg}}$, then when $F=F^h$ (i.e. it is Henselian),  the isomorphism of Remark \ref{nrmk c-cohom Bf and Bfalg} is an isomorphism of $G$-modules. However, unlike the \'etale cohomology, we do not have the Cartan-Leray spectral sequence relating the cohomologies  of $B_{\bF}(X)$ and $B_{\bF^{{\rm alg}}}(X)$ through the absolute Galois group $G=\mathrm{Gal}(F^{{\rm alg}}/F^h)$. Indeed, in this case, $G$ does not act freely and properly on $B_{\bF^{{\rm alg}}}(X)$. \\ 

Regarding vanishing of cohomology we have: 

\begin{nrmk}\label{nrmk vanishing B_F}
Let $X$ be an $\bF$-definable  subset of a quasi-projective variety $V$ over $F$. Since the finite simplicial complex ${\mathcal X}$ of Fact \ref{deformation for Bf} to can be assumed to have dimension less or equal to $\dim (\overline{X}^\mathrm{Zar})$, then by homotopy axiom in topology (\cite[Chapter II, 11.8]{b}),
we have: 
\[
H^p(B_{\bF }(X); L_{B_{\bF }(X)})=0 \,\,\textrm{for every}\,\, p>\dim (\overline{X}^\mathrm{Zar})
\]
and, similarly, when $X$ is a $\vgs$-locally closed subset, we have 
\[
H^p_c(B_{\bF }(X); L_{B_{\bF }(X)})=0 \,\,\textrm{for every}\,\, p>\dim (\overline{X}^\mathrm{Zar}).
\]

It is worthwhile to point out that the first these vanishing results was also observed in the paper \cite{basu} where it plays a crucial role.  
\end{nrmk}

\section{Some tameness results in families}\label{sec Betti}
In \cite{HrLo} Hrushovski and Loeser show finiteness of homotopy (even homeomorphism) types of  uniform families of (model-theoretic) Berkovich spaces (see \cite[Theorems 14.3.1 and 14.4.4]{HrLo}).  In higher ranks one has to replace topological homotopy (resp. homeomorphism) type by definable homotopy (resp. homeomorphism) type. As observed by Hrushovski and Loeser at the beginning \cite[Section 14.3]{HrLo}, in higher ranks we no longer have such finiteness results. In fact one  can have a uniform family of triangles in $\nG ^2$, corresponding to skeleta of a uniform family of elliptic curves,  without finitely many definable homotopy (equivalently definable homeomorphism) types. 

Taking advantage of the invariance results for our cohomology theories with definably compact supports, we can proof finiteness of cohomological complexity in uniform families. Namely:

\begin{thm}\label{thm: finite cohom complexity hat}
Let $K$ be an algebraically closed non-trivially valued non-archimedean field. Let $V$ and $V'$ be algebraic varieties over $K$. Let $W\subseteq V'\times \nGik$ be a definable subset and let $Z\subseteq V\times W$ be a $\vgs$-locally closed  subset. For $w\in W$ let $Z_w=\{v\in V: (v,w)\in Z\}$. Then as $w$ runs through $W$ there are finitely many possibilities for the  $\hvgs$-cohomology
\[
H^*_c(\widehat{Z}_{w};{\mathbb Q})
\]
with definably compact supports.
\end{thm}

\begin{proof}
First notice that it is enough to prove the result assuming that $V$ is a quasi-projective variety. In fact, in general, if $V$ is a variety, then $V$ is a finite union of open quasi-projective sub-varieties $V_1, \ldots, V_n$ and the result follows by induction on $n$ using the Mayer-Vietoris sequence as in the proof of Theorem \ref{thm finiteness in acvf}.

By the uniform version of the main theorem of Hrushovski and Loeser (\cite[Proposition 11.7.1]{HrLo}) there is a uniformly pro-definable family $H_{w}:I\times \widehat{Z_{w}}\to \widehat{Z_{w}}$, a definable subset ${\mathcal Z}_{w}\subseteq \nG_{\infty}^{\dim V}$, an iso-definable subset ${\mathfrak Z}_{w}$  of $\widehat{Z}_{w}$, and $h_{w}:{\mathfrak Z}_{w}\to {\mathcal Z}_{w}$, pro-definable uniformly in $w$, such that for each $w$, $H_{w}$ is a deformation retraction of $\widehat{Z}_{w}$ onto ${\mathfrak Z}_{w}$, and $h_{w}$ is a pro-definable homeomorphism.

By stable embeddedness of $\nGi$ (see for example \cite[Proposition 2.7.1]{HrLo}) and elimination of imaginaries in $\nGi$, there is a definable subset $D\subseteq  \Gamma_\infty^{m'}$ (for some $m'$), a definable surjection $\tau \colon W\to D$  and a definable family $\mathcal{Z}'_{\gamma}$ with $\gamma$ ranging over $D$, such that for every $w\in W$
\[
\mathcal{Z}_{w} = \mathcal{Z}'_{\tau (w)}. 
\]

Let $K'$ be a sufficiently saturated  elementary extension of $K$ such that its value group $\nG'$ expands to a model $\bG ''$ of real closed fields.  Then we have a commutative diagrams of morphisms of $\hvgs$-sites

\begin{equation*}
\xymatrix{
I(K')\times \widehat{Z}_{w}(K') \ar[r]^{\,\,\,\,\,\,\,\,\,\,\,\,\,\,\,H_{w}^{K'}}  \ar[d]^{} & \widehat{Z}_{w}(K') \ar[d]^{} \\
I\times \widehat{Z}_{w} \ar[r]^{H_{w}} & \widehat{Z}_{w}
}
\end{equation*} 

and 

\begin{equation*}
\xymatrix{
{\mathfrak Z}_{w}(K') \ar[r]^{h_{w}^{K'}}  \ar[d]^{} & \mathcal{Z}_{w}(K') \ar[d]^{} \\
{\mathfrak Z}_{w} \ar[r]^{h_{w}} & \mathcal{Z}_{w}
}
\end{equation*} 
uniformly in  $w$, with $H_{w}^{K'}$ a pro-definable deformation retraction, $h_{w}^{K'}$ a pro-definable homeomorphism and where  the vertical arrows are the morphisms of $\hvgs$-sites taking $\widehat{U}$ to $\widehat{U}(K')$. 

Note that since $Z$ is $\vgs$-locally closed, each $Z_{w}$ is also $\vgs$-locally closed subset. By Lemma \ref{lem retraction of vgs-locally closed}, each $\mathcal{Z}_w$ is a definably locally closed subset. It follows that we take  $\hvgs$-cohomologies with definably compact supports and obtain the following commutative diagram of isomorphisms

\begin{equation*}
\xymatrix{
& & H_c^*(\mathcal{Z}'_{\gamma }(\Gamma '');\QQ)\ar[d]^{\sim } \\
H_c^*(\widehat{Z}_{w}(K');\QQ) \ar[r]^{\sim }  \ar[d]^{\sim } & H_c^*(\mathcal{Z}_{w}(K');\QQ) \ar[d]^{\sim} \ar[r]^{=} & H_c^*(\mathcal{Z}'_{\gamma }(\Gamma ');\QQ)\ar[d]^{\sim }\\
H_c^*(\widehat{Z}_{w};\QQ) \ar[r]^{\sim} & H_c^*(\mathcal{Z}_{w};\QQ) \ar[r]^{=}& H_c^*(\mathcal{Z}'_{\gamma };\QQ)
}
\end{equation*} 
where $\gamma =\tau (w)$, the vertical isomorphisms are given by the invariance results (Theorems \ref{thm inv in acvf} and \ref{thm bf inv gp-int}) and the horizontal isomorphism are induced by the pro-definable deformation retractions (Remark \ref{nrmk r* and h* are iso}). 

By the definable trivialization theorem in the real closed field $\bG ''$ (\cite[Chapter 9]{vdd}), there is a partition of $D(\nG '')$ by finitely many $\bG''$-definable subsets $D_1,\ldots, D_r$ such that for every $1\leq i\leq r$ and for all $\gamma , \gamma '\in D_i$ we have that ${\mathcal Z}_{\gamma}$ is $\bG ''$-definably homeomorphic to $\mathcal{Z}_{\gamma '}$. In particular, we have \[
H^*_c(\mathcal{Z}_{\gamma }(\nG '');{\mathbb Q})\simeq H^*_c(\mathcal{Z}_{\gamma '}(\nG '');{\mathbb Q})
\]
for all $\gamma , \gamma '\in D_i$ and for every $1\leq i\leq r$ and the result follows.
\end{proof}

While  Theorem \ref{thm: finite cohom complexity hat} above can be see as a  higher rank analogue of  \cite[Theorem 14.3.1]{HrLo} concerning finiteness of homotopy types in uniform families in the rank one case, the next result is the analogue of \cite[Theorem 14.4.4]{HrLo}.

\begin{cor}\label{cor: finite cohom complexity hat}
Let $K$ be an algebraically closed non-trivially valued non-archimedean field. Let $V$ be an algebraic variety over $K$. Let $X\subseteq V$ be a $\vgs$-locally closed definable subset and let $G:X\to \nGi$ be a definable map. Then there is a finite partition of $\nGi$ into intervals such that the fibers of $\widehat{G}: \widehat{X}\to \nGi$ over each interval have canonically isomorphic  $\hvgs$-cohomology groups with definable compact supports.
\end{cor}

\begin{proof}
Let $Z\subseteq V\times \nGi$ be the graph of $G:X\to \nGi$. Then for each $w\in \nGi$ we have $Z_w=G^{-1}(w)$ and $\widehat{Z}_w=\widehat{G}^{-1}(w)$.  Note that the $W$ and the $D$ of the proof of Theorem \ref{thm: finite cohom complexity hat} are both equal to $\nGi$ here, in particular, $\tau $ is the identity. So that proof  gives us that, by the definable trivialization theorem in the real closed field $\nG ''$ (\cite[Chapter 9]{vdd}), there is a partition of $\nG_{\infty} ''$ by finitely many $\bG''$-definable subsets $D_1,\ldots, D_r$ such that for every $1\leq i\leq r$ and for all $\gamma , \gamma '\in D_i$ we have that ${\mathcal Z}_{\gamma}$ is $\bG ''$-definably homeomorphic to $\mathcal{Z}_{\gamma '}$. In particular, we have 
\[
H^*_c(\mathcal{Z}_{\gamma }(\nG '');{\mathbb Q})\simeq H^*_c(\mathcal{Z}_{\gamma '}(\nG '');{\mathbb Q})
\]
for all $\gamma , \gamma '\in D_i$ and for every $1\leq i\leq r$. Now the definable trivialization theorem ensures that we can take the $D_i$'s to be $\bG ''$-definable with parameters in $\nGi$, hence by o-minimality, they are interval in $\nG_{\infty}''$ with endpoints in $\nGi$. To conclude take the intervals $D_1(\nGi), \ldots, D_r(\nGi)$ partitioning $\nGi$.
\end{proof}

\begin{nrmk}\label{rem:coefficients} 
Both in Theorem \ref{thm: finite cohom complexity hat} and Corollary \ref{cor: finite cohom complexity hat} one may replace $\QQ$ by any $A$-module, for $A$ a commutative ring with unit. 
\end{nrmk}

Given a definable set $X\subseteq V\times\nGin$, let $\pi _0^{\df}$ be the functor taking $\widehat{X}$ into its set of definably connected components. Note that by Lemma \ref{lem vgs components}, $\pi_0^\df(\widehat{X})$ is a finite set. Moreover, $\pi _0^{\df}$ is invariant under elementary extensions of $K$ (Lemma \ref{lem vgs and def conn}) and in $\bGi$ is also invariant under o-minimal expansions. Since $\pi _0^{\df}$ is furthermore invariant under pro-definable deformation retractions, using the same transfer arguments as above, we can prove the following higher rank analogue of results by A. Abbes and T. Saito (\cite[Theorem 5.1]{abbes-saito}) and by J. Poineau (\cite[Th\'eor\`eme 2]{poineau}).

\begin{cor}\label{cor: finite pi0 complexity hat}
Let $K$ be an algebraically closed non-trivially valued non-archimedean field. Let $V$ be an algebraic variety over $K$. Let $X\subseteq V$ be a definable subset and let $G:X\to \nGi$ be a definable map. Then there is a finite partition of $\nGi$ into intervals such that over each interval $I$, for any $\epsilon ', \epsilon \in I$, we have a canonical bijection between $\pi _0^{\df}(\widehat{G}^{-1}(\epsilon '))$ and $\pi _0^{\df}(\widehat{G}^{-1}(\epsilon ))$.
\end{cor}

We end with some remarks containing open questions which will be considered in the sequel to this paper:

\begin{nrmk}\label{nrmk: final questions 1}
One can ask if there is also a higher rank analogue of local contractability (\cite[Theorem 14.4.1]{HrLo}) e.g. definable local contractibility or local acyclicity with respect to $\vgs$-cohomology groups. Following the proof in \cite{HrLo} this question reduces to the corresponding question in $\bGi$.  See \cite{El} for such results in $\bG$.
\end{nrmk}

\begin{nrmk}\label{nrmk: final questions 2}
An issue that was not settled here is that of  finiteness and invariance  for the $\hvgs$-cohomology without supports. By Remark \ref{nrmk r* and h* are iso}, that reduces to knowing if we have such results for o-minimal cohomology without supports in $\bGi$. Such a result would gives us also  finiteness of cohomological complexity as above but without supports.

Moreover, we could also obtain in higher ranks uniform sharp bounds on $\hvgs$-Betti numbers as in a result obtained by S. Basu  and D. Patel (\cite[Theorem 2]{basu}) in the case where the valued field has  rank one, using the usual topological (singular) Betti numbers. Using the strategy above, such a result could be obtained by reduction of the problem to o-minimal expansions of real closed fields and then using the analogue result in that setting proved earlier by Basu 
(\cite[Theorem 2.2]{basu-o-minimal}).
\end{nrmk}

\subsection*{Acknowledgments} We wish to thank P. Eleftheriou and L. Prelli for interesting discussions on the results presented in Section \ref{section t-top in bGi}. We would also like to thank A. Ducros, E. Hrushovski and F. Loeser for interesting discussions. 

P. Cubides Kovacsics was partially supported by the ERC project TOSSIBERG (Grant Agreement 637027) and individual research grant
\emph{Archimedische und nicht-archimedische Stratifizierungen h\"oherer
Ordnung}, funded by the DFG. M.  Edmundo was supported by the National Funding from FCT - Funda\c{c}\~ao para a Ci\^encia e a Tecnologia, under the projects UIDB/04561/2020 and UIDP/04561/2020. J.Ye was partially supported by NSF research grant DMS1500671 and Fran\c cois Loeser's Institut Universitaire de France grant.

\bibliographystyle{siam}
\bibliography{biblio}

\end{document}